\documentclass[reqno,11pt]{amsart}
\usepackage{amssymb,amsthm,amsmath}
\usepackage[breaklinks=true]{hyperref}

\theoremstyle{plain}
\newtheorem{theorem}{Theorem}[section]
\newtheorem{corollary}[theorem]{Corollary}
\newtheorem{lemma}[theorem]{Lemma}
\newtheorem{proposition}[theorem]{Proposition}

\theoremstyle{definition}
\newtheorem{definition}[theorem]{Definition}
\newtheorem{remark}[theorem]{Remark}
\newtheorem{example}[theorem]{Example}

\numberwithin{equation}{section}
\numberwithin{figure}{section}

\newcommand{\C}{\mathbb{C}}
\newcommand{\F}{\mathbb{F}}
\newcommand{\I}{\mathrm{i}}
\newcommand{\N}{\mathbb{N}}
\newcommand{\Q}{\mathbb{Q}}
\newcommand{\R}{\mathbb{R}}
\newcommand{\T}{\mathbb{T}}
\newcommand{\Z}{\mathbb{Z}}

\newcommand{\bb}{\mathbf{b}}
\newcommand{\bc}{\mathbf{c}}
\newcommand{\bone}{\mathbf{1}}
\newcommand{\bm}{\mathbf{m}}
\newcommand{\bn}{\mathbf{n}}
\newcommand{\bs}{\mathbf{s}}
\newcommand{\bu}{\mathbf{u}}
\newcommand{\bv}{\mathbf{v}}
\newcommand{\bw}{\mathbf{w}}
\newcommand{\bx}{\mathbf{x}}
\newcommand{\bX}{\mathbf{X}}
\newcommand{\by}{\mathbf{y}}
\newcommand{\bz}{\mathbf{z}}
\newcommand{\bzero}{\mathbf{0}}

\newcommand{\cC}{\mathcal{C}}
\newcommand{\cH}{\mathcal{H}}
\newcommand{\cK}{\mathcal{K}}
\newcommand{\cP}{\mathcal{P}}
\newcommand{\cS}{\mathcal{S}}
\newcommand{\cU}{\mathcal{U}}

\newcommand{\Cov}{\mathop{\mathrm{Cov}}}
\newcommand{\diag}{\mathop{\mathrm{diag}}}
\newcommand{\HTN}{HTN}
\newcommand{\moment}{\mathcal{M}}
\newcommand{\pos}{\cH_{\mathrm{pos}}}
\newcommand{\sign}{\mathop{\mathrm{sign}}}
\newcommand{\rk}{\mathop{\mathrm{rank}}}
\newcommand{\rd}{\mathrm{d}}
\newcommand{\std}{\,\rd}
\newcommand{\Var}{\mathop{\mathrm{Var}}}

\begin{document}

\title{A panorama of positivity}

\author[A.~Belton]{Alexander Belton}
\address[A.~Belton]{Department of Mathematics and Statistics, Lancaster
University, Lancaster, UK}
\email{\tt a.belton@lancaster.ac.uk}

\author[D.~Guillot]{Dominique Guillot}
\address[D.~Guillot]{University of Delaware, Newark, DE, USA}
\email{\tt dguillot@udel.edu}

\author[A.~Khare]{Apoorva Khare}
\address[A.~Khare]{Indian Institute of Science;
Analysis and Probability Research Group; Bangalore, India}
\email{\tt khare@iisc.ac.in}

\author[M.~Putinar]{Mihai Putinar}
\address[M.~Putinar]{University of California at Santa Barbara, CA,
USA and Newcastle University, Newcastle upon Tyne, UK} 
\email{\tt mputinar@math.ucsb.edu, mihai.putinar@ncl.ac.uk}

\date{\today}

\thanks{D.G.~is partially supported by a University of Delaware Research
Foundation grant, by a Simons Foundation collaboration grant for
mathematicians, and by a University of Delaware Research Foundation
Strategic Initiative grant. A.K.~is partially supported by Ramanujan
Fellowship SB/S2/RJN-121/2017 and MATRICS grant MTR/2017/000295 from SERB
(Govt.~of India), by grant F.510/25/CAS-II/2018(SAP-I) from UGC (Govt.~of
India), and by a Young Investigator Award from the Infosys Foundation.}

\begin{abstract}
This survey contains a selection of topics unified by the concept of
positive semi-definiteness (of matrices or kernels), reflecting natural
constraints imposed on discrete data (graphs or networks) or continuous
objects (probability or mass distributions). We put emphasis on entrywise
operations which preserve positivity, in a variety of guises. Techniques
from harmonic analysis, function theory, operator theory, statistics,
combinatorics, and group representations are invoked. Some partially
forgotten classical roots in metric geometry and distance transforms are
presented with comments and full bibliographical references. Modern
applications to high-dimensional covariance estimation and regularization
are included.
\end{abstract}



\maketitle

\tableofcontents


\section{Introduction}

Matrix positivity, or positive semidefiniteness, is one of the most
wide-reaching concepts in mathematics, old and new. Positivity of a
matrix is as natural as positivity of mass in statics or positivity of
a probability distribution. It is a notion which has attracted the
attention of many great minds. Yet, after at least two centuries of
research, positive matrices still hide enigmas and raise challenges
for the working mathematician.

The vitality of matrix positivity comes from its breadth, having many
theoretical facets and also deep links to mathematical modelling. It
is not our aim here to pay homage to matrix positivity in the large.
Rather, the present survey, split for technical reasons into two
parts, has a limited but carefully chosen scope.

Our panorama focuses on entrywise transforms of matrices which
preserve their positive character. In itself, this is a rather bold
departure from the dogma that canonical transformations of matrices
are not those that operate entry by entry. Still, this apparently
esoteric topic reveals a fascinating history, abundant characteristic
phenomena and numerous open problems. Each class of positive matrices
or kernels (regarding the latter as continuous matrices) carries a
specific toolbox of internal transforms. Positive Hankel forms or
Toeplitz kernels, totally positive matrices, and group-invariant
positive definite functions all possess specific \emph{positivity
preservers}. As we see below, these have been thoroughly studied for
at least a century.

One conclusion of our survey is that the classification of positivity
preservers is accessible in the dimension-free setting, that is, when
the sizes of matrices are unconstrained. In stark contrast, precise
descriptions of positivity preservers in fixed dimension are elusive,
if not unattainable with the techniques of modern mathematics.
Furthermore, the world of applications cares much more about matrices
of fixed size than in the free case. The accessibility of the latter
was by no means a sequence of isolated, simple observations. Rather,
it grew organically out of distance geometry, and spread rapidly
through harmonic analysis on groups, special functions, and
probability theory. The more recent and highly challenging path
through fixed dimensions requires novel methods of algebraic
combinatorics and symmetric functions, group representations, and
function theory.

As well as its beautiful theoretical aspects, our interest in these
topics is also motivated by the statistics of big data. In this
setting, functions are often applied entrywise to covariance matrices,
in order to induce sparsity and improve the quality of statistical
estimators (see \cite{GuillotRajaratnam2012,
GuillotRajaratnam2012b,Rothman2009}). Entrywise techniques have
recently increased in popularity in this area, largely because of
their low computational complexity, which makes them ideal to handle
the ultra high-dimensional datasets arising in modern applications. In
this context, the dimensions of the matrices are fixed, and correspond
to the number of underlying random variables. Ensuring that positivity
is preserved by these entrywise methods is critical, as covariance
matrices must be positive semidefinite. Thus, there is a clear need to
produce characterizations of entrywise preservers, so that these
techniques are widely applicable and mathematically justified. We
elaborate further on this in the second part of the survey.

We conclude by remarking that, while we have tried to be comprehensive
in our coverage of the field of matrix positivity and the entrywise
calculus, our panorama is far from being complete. We apologize for any omissions. 


\section{From metric geometry to matrix positivity}

\subsection{Distance geometry}

During the first decade of the 20th century, the concept of a metric
space emerged from the works of Fr\'echet and Hausdorff, each having
different and well-anchored roots, in function spaces and in set theory
and measure theory. We cannot think today of modern mathematics and
physics without referring to metric spaces, which touch areas as diverse
as economics, statistics, and computer science. Distance geometry is one
of the early and ever-lasting by-products of metric-space
theory. One of the key figures of the Vienna Circle, Karl Menger, started
a systematic study in the 1920s of the geometric and topological features
of spaces that are intrinsic solely to the distance they carry. Menger
published his findings in a series of articles having the generic name
``Untersuchungen \"uber allgemeine Metrik,'' the first one being
\cite{Menger-MathAnn28}; see also his synthesis \cite{Menger-AJM31}. His
work was very influential in the decades to come \cite{Blumenthal-book},
and by a surprising and fortunate stroke not often encountered in
mathematics, Menger's distance geometry has been resurrected in recent
times by practitioners of convex optimization and network analysis
\cite{Dattorro-LAA08,Liberti-SIAM-Rev14}.

Let $( X, \rho )$ be a metric space. One of the naive, yet
unavoidable, questions arising from the very beginning concerns the
nature of operations $\phi( \rho )$ which may be performed on the
metric and which enhance various properties of the topological space
$X$. We all know that $\rho / ( \rho + 1 )$ and $\rho^\gamma$, if
$\gamma \in ( 0, 1 )$, also satisfy the axioms of a metric, with the
former making it bounded. Less well known is an observation due to
Blumenthal, that the new metric space $( X, \rho^\gamma )$ has the
four-point property if $\gamma \in ( 0 , 1 / 2 ]$: every four-point
subset of $X$ can be embedded isometrically into Euclidean space 
\cite[Section~49]{Blumenthal-book}.

Metric spaces which can be embedded isometrically into Euclidean
space, or into infinite-dimensional Hilbert space, are, of course,
distinguished and desirable for many reasons. We owe to Menger a
definitive characterization of this class of metric spaces. The core
of Menger's theorem, stated in terms of certain matrices built from
the distance function (known as Cayley--Menger matrices) was slightly
reformulated by Fr\'echet and cast in the following simple form by
Schoenberg.

\begin{theorem}[Schoenberg \cite{Schoenberg-Annals35}]\label{Tembed}
Let $d \geq 1$ be an integer and let $( X, \rho )$ be a metric space.
An $(n + 1)$-tuple of points $x_0$, $x_1$, \ldots, $x_n$ in $X$ can be
isometrically embedded into Euclidean space $\R^d$, but not into
$\R^{d - 1}$, if and only if the matrix
\[
[ \rho( x_0, x_j )^2 + \rho( x_0, x_k )^2 - %
\rho( x_j, x_k )^2 ]_{j, k = 1}^n,
\]
is positive semidefinite with rank equal to $d$.
\end{theorem} 

\begin{proof}
This is surprisingly simple. Necessity is immediate, since the
Euclidean norm and scalar product in $\R^d$ give that
\begin{align*}
& \rho( x_0, x_j )^2 + \rho( x_0, x_k )^2 - \rho( x_j, x_k )^2 \\
 & = \| x_0 - x_j \|^2 + \| x_0 - x_k \|^2 - \| ( x_0 - x_j) - %
( x_0 - x_k ) \|^2 \\
 & = 2 \langle x_0 - x_j, x_0 - x_k \rangle,
\end{align*}
and the latter are the entries of a positive semidefinite Gram matrix
of rank less than or equal to $d$.

For the other implication, we consider first a full-rank $d \times d$
matrix associated with a $( d + 1 )$-tuple. The corresponding
quadratic form
\[
Q(\lambda) = \frac{1}{2} \sum_{j, k = 1}^d ( \rho(x_0,x_j)^2 + %
\rho(x_0,x_k)^2 - \rho(x_j,x_k)^2 ) \lambda_j \lambda_k
\]
is positive definite. Hence there exists a linear change of variables 
\[
\lambda_k = \sum_{j = 1}^d a_{j k} \mu_j \qquad ( 1 \leq j \leq d )
\]
such that
\[
Q( \lambda ) = \mu_1^2 + \mu_2^2 + \cdots +\mu_d^2.
\]
Interpreting $( \mu_1, \mu_2, \ldots,\mu_d )$ as coordinates in
$\R^d$, the standard simplex with vertices
\[
e_0 = ( 0, \ldots , 0 ), \quad e_1 = ( 1, 0, \ldots, 0 ), \quad \ldots,
\quad e_d = ( 0, \ldots, 0, 1 )
\]
has the corresponding quadratic form (of distances) equal to
$\mu_1^2 + \mu_2^2 + \cdots +\mu_d^2$.  Now we perform the coordinate
change $\mu_j \mapsto \lambda_j$. Specifically, set $P_0 = 0$ and let
$P_j \in \R^d$ be the point with coordinates $\lambda_j = 1$ and
$\lambda_k =0$ if $k \neq j$. Then one identifies distances:
\begin{align*}
\| P_0 - P_j \| & = \rho( x_0, x_j ) \qquad ( 0 \leq j \leq d ) \\
\text{and} \quad %
\| P_j - P_k \| & = \rho( x_j, x_k ) \qquad ( 1 \leq j, k \leq d ).
\end{align*}
The remaining case with $n > d$ can be analyzed in a similar way,
after taking an appropriate projection.
\end{proof}

In the conditions of the theorem, fixing a ``frame'' of $d$ points and
letting the $( d + 1 )$-th point float, one obtains an embedding of
the full metric space $( X, \rho )$ into $\R^d$. This idea goes back
to Menger, and it led, with Schoenberg's touch, to the following
definitive statement. Here and below, all Hilbert spaces are assumed
to be separable.

\begin{corollary}%
[Schoenberg \cite{Schoenberg-Annals35}, following Menger]%
\label{Menger-Schoenberg}
A separable metric space $( X, \rho )$ can be isometrically embedded
into Hilbert space if and only if, for every $( n + 1 )$-tuple of
points $( x_0, x_1, \ldots, x_n )$ in $X$, where $n \geq 2$, the
matrix
\[
[ \rho( x_0, x_j )^2 + \rho( x_0, x_k )^2 - %
\rho( x_j, x_k )^2 ]_{j, k = 1}^n
\]
is positive semidefinite.
\end{corollary}

The notable aspect of the two previous results is the interplay
between purely geometric concepts and matrix positivity. This will be
a recurrent theme of our survey.

\subsection{Spherical distance geometry}

One can specialize the embedding question discussed in the previous
section to submanifolds of Euclidean space. A natural choice is the
sphere.

For two points $x$ and $y$ on the unit sphere
$S^{d - 1} \subset \R^d$, the rotationally invariant distance between
them is
\[
\rho( x, y ) = \sphericalangle( x, y ) = \arccos\langle x, y \rangle,
\]
where the angle between the two vectors is measured on a great circle
and is always less than or equal to $\pi$.

A straightforward application of the simple, but central,
Theorem~\ref{Tembed}] yields the following result.

\begin{theorem}[Schoenberg \cite{Schoenberg-Annals35}]
Let $( X, \rho )$ be a metric space and let $( x_1, \ldots, x_n )$ be
an $n$-tuple of points in $X$. For any integer $d \geq 2$, there
exists an isometric embedding of $( x_1, \ldots, x_n )$ into
$S^{d - 1}$ endowed with the geodesic distance but not
$S^{d - 2}$ if and only if
\[
\rho( x_j, x_k ) \leq \pi \qquad ( 1 \leq j, k \leq n )
\]
and the matrix $\bigl[ \cos \rho( x_j, x_k ) \bigr]_{j, k = 1}^n$
is positive semidefinite of rank $d$.
\end{theorem}

Indeed, the necessity is assured by choosing $x_0$ to be the origin in
$\R^d$. In this case,
\begin{align*}
\rho( x_0, x_j )^2 + \rho( x_0, x_k )^2 - \rho( x_j, x_k )^2 & = %
\| x_j\|^2 + \| x_k \|^2 - \| x_j - x_k\|^2 \\
 & = 2 \langle x_j, x_k \rangle \\
 & = 2 \cos \rho( x_j, x_k ).
\end{align*}

The condition is also sufficient, by possibly adding an external point
$x_0$ to the metric space, subject to the constraints that
$\rho( x_0, x_j ) = 1$ for all $j$. The details can be found in
\cite{Schoenberg-Annals35}.%
\footnote{An alternate proof of sufficiency is to note that
$A := [ \cos \rho( x_j, x_k ) ]_{j, k = 1}^n$ is a
Gram matrix of rank $r$, hence equal to $B^T B$ for some
$r \times n$ matrix  $B$ with unit columns. Denoting these columns by
$\bb_1$, \ldots, $\bb_n \in S^{r-1}$, the map
$x_j \mapsto b_j$ is an isometry since
$\rho( x_j, x_k )$ and $\sphericalangle( y_j, y_k ) \in [ 0, \pi ]$.
Moreover, since $A$ has rank $r$, the $\bb_j$ cannot all lie in a
smaller-dimensional sphere.}

\subsection{Distance transforms}

A notable step forward in the study of the existence of isometric
embeddings of a metric space into Euclidean or Hilbert space was made
by Schoenberg. In a series of articles
\cite{Schoenberg-Annals37,Schoenberg-TAMS38,Schoenberg-Annals39,vonNeumann-Schoenberg-TAMS41},
he changed the set-theoretic lens of Menger, by initiating a
harmonic-analysis interpretation of this embedding problem. This was a
major turning point, with long-lasting, unifying, and unexpected
consequences.

We return to a separable metric space $( X, \rho )$ and seek
distance-function transforms $\rho \mapsto \phi( \rho )$ which enhance
the geometry of $X$, to the extent that the new metric space
$\bigl( X, \phi( \rho) \bigr)$ is isometrically equivalent to a
subspace of Hilbert space. Schoenberg launched this whole new
chapter from the observation that the Euclidean norm is such that the
matrix
\[
[ \exp\bigl( -\| x_j - x_k \|^2 \bigr) ]_{j, k = 1}^N
\]
is positive semidefinite for any choice of points $x_1$, \ldots, $x_N$
in the ambient space. Once again, we see the presence of matrix
positivity. While this claim may not be obvious at first sight, it is
accessible once we recall a key property of Fourier transforms.

An even function $f : \R^d \to \C$ is said to be \emph{positive
definite} if the complex matrix $[ f( x_j - x_k ) ]_{j, k = 1}^N$ is
positive semidefinite for any $N \geq 1$ and any choice of points
$x_1$, \ldots, $x_N \in \R^d$. We will call $f( x - y )$ a
\emph{positive semidefinite kernel} on $\R^d \times \R^d$ in this case.
(See \cite{Stewart} for a comprehensive survey of this class of maps.)

Bochner's theorem \cite{Bochner-MathAnn33} characterizes
positive definite functions on $\R^d$ as Fourier transforms of even
positive measures of finite mass:
\[
f( \xi ) = \int e^{-\I x \cdot \xi} \std\mu( x ).
\]
Indeed,
\[
f( \xi - \eta) = %
\int e^{-\I x \cdot \xi} e^{\I x \cdot \eta} \std\mu( x )
\]
is a positive semidefinite kernel because it is the average over
$\mu$ of the positive kernel
$( \xi, \eta ) \mapsto e^{-\I x \cdot \xi} e^{\I x \cdot \eta}$. Since
the Gaussian $e^{-x^2}$ is the Fourier transform
of itself (modulo constants), it turns out that it is a
positive definite function on $\R$, whence $\exp( -\| x \|^2 )$ has
the same property as a function on $\R^d$. Taking one step further,
the function $x \mapsto \exp( -\| x \|^2 )$ is positive definite on
any Hilbert space.
 
With this preparation we are ready for a second
characterization of metric subspaces of Hilbert space.

\begin{theorem}[Schoenberg \cite{Schoenberg-TAMS38}]\label{Tsch-pos-def}
A separable metric space $( X, \rho )$ can be embedded isometrically
into Hilbert space if and only if the kernel
\[
X \times X \to ( 0, \infty ); \ %
( x, y ) \mapsto \exp( -\lambda^2 \rho( x, y )^2 )
\]
is positive semidefinite for all $\lambda \in \R$.
\end{theorem}

\begin{proof}
Necessity follows from the positive definiteness of the Gaussian
discussed above. (We also provide an elementary proof below; see
Lemma~\ref{Lpolya} and the subsequent discussion). To prove
sufficiency, we recall the Menger--Schoenberg characterization of
isometric subspaces of Hilbert space. We have to derive, from the
positivity assumption, the positivity of the matrix
\[
[ \rho( x_0, x_j )^2 + \rho( x_0, x_k )^2 - %
\rho( x_j, x_k )^2 ]_{j, k = 1}^n.
\]
Elementary algebra transforms this constraint into the requirement
that
\[
\sum_{j,k=0}^n \rho( x_j, x_k )^2 c_j c_k \leq 0 \qquad %
\text{whenever } \sum_{j = 0}^n c_j = 0.
\]
By expanding $\exp( -\lambda^2 \rho( x_j, x_k )^2 )$ as a power series
in $\lambda^2$, and invoking the  positivity of the exponential
kernel, we see that
\[
0 \leq -\lambda^2 \sum_{j, k = 0}^n \rho( x_j, x_k )^2 c_j c_k + %
\frac{\lambda^4}{2} \sum_{j, k = 0}^n \rho( x_j, x_k )^4 c_j c_k - %
\cdots
\]
for all $\lambda > 0$. Hence the coefficient of $-\lambda^2$ is
non-positive.
\end{proof}
 
The flexibility of the Fourier-transform approach is illustrated by
the following application, also due to Schoenberg
\cite{Schoenberg-TAMS38}.
 
\begin{corollary}
Let $H$ be a Hilbert space with norm $\| \cdot \|$. For every
$\delta \in ( 0, 1 )$, the metric space $( H, \| \cdot \|^\delta )$ is
isometric to a subspace of a Hilbert space.
\end{corollary}

\begin{proof}
Note first the identity
\[
\xi^\alpha = %
c_\alpha \int_0^\infty ( 1 - e^{-s^2 \xi^2} ) s^{-1 - \alpha} \std s %
\qquad ( \xi > 0, \ 0 < \alpha < 2 ),
\]
where $c_\alpha$ is a normalization constant. Consequently,
\[
\| x - y \|^\alpha = c_\alpha %
\int_0^\infty ( 1 - e^{-s^2 \| x - y \|^2} ) s^{-1 - \alpha} \std s.
\]
Let $\delta = \alpha / 2$. For points $x_0$, $x_1$, \ldots, $x_n$ in
$H$ and weights $c_0$, $c_1$, \ldots, $c_n$ satisfying
\[
c_0 + c_1 + \cdots + c_n = 0,
\]
it holds that
\[
\sum_{j, k = 0}^n \| x_j - x_k \|^{2 \delta} c_j c_k = %
-c_\alpha \int_0^\infty %
\sum_{j, k = 0}^n c_j c_k e^{-s^2 \| x_j - x_k \|^2} %
s^{-1 - \alpha} \std s %
\leq 0,
\]
and the proof is complete.
\end{proof}
 
Several similar consequences of the Fourier-transform approach are
within reach. For instance, Schoenberg observed in the same article
that if the $L^p$ norm is raised to the power $\gamma$, where
$0 < \gamma \leq p / 2$ and $1 \leq p \leq 2$, then $L^p( 0, 1 )$ is
isometrically embeddable into Hilbert space.

\subsection{Altering Euclidean distance}

By specializing the theme of the previous section to Euclidean space,
Schoenberg and von Neumann discovered an arsenal of powerful tools
from harmonic analysis that were able to settle the question of
whether Euclidean space equipped with the altered distance
$\phi\bigl( \| x - y \| \bigr)$ may be isometrically embedded into
Hilbert space
\cite{Schoenberg-Annals38,vonNeumann-Schoenberg-TAMS41}. The key
ingredients are characterizations of Laplace and Fourier transforms of
positive measures, that is, Bernstein's completely monotone functions
\cite{Bernstein-Acta29} and Bochner's positive definite functions
\cite{Bochner-MathAnn33}.
 
Here we present some highlights of the Schoenberg--von Neumann
framework. First, we focus on an auxiliary class of distance
transforms. A real continuous function $\phi$ is called \emph{positive
definite in Euclidean space $\R^d$} if the kernel
\[
( x, y ) \mapsto \phi\bigl( \| x - y \| \bigr)
\]
is positive semidefinite. Bochner's theorem and the rotation-invariance
of this kernel prove that such a function $\phi$ is characterized by the
representation
\[
\phi( t ) = \int_0^\infty \Omega_d( t u ) \std\mu( u ),
\]
where $\mu$ is a positive measure and
\[
\Omega_d\bigl( \| x \| \bigr) = %
\int_{\| \xi \| = 1} e^{\I x \cdot \xi} \std\sigma( \xi ),
\]
with $\sigma$ the normalized area measure on the unit sphere in
$\R^d$; see \cite[Theorem~1]{Schoenberg-Annals38}. By letting $d$ tend
to infinity, one finds that positive definite functions on
infinite-dimensional Hilbert space are precisely of the form
\[
\phi( t ) = \int_0^\infty e^{-t^2 u^2} \std\mu( u ),
\]
with $\mu$ a positive measure on the semi-axis. Notice that
positive definite functions in $\R^d$ are not necessarily differentiable
more than $( d - 1 ) / 2$ times, while those which are positive definite
in Hilbert space are smooth and even complex analytic in the sector
$| \arg t | < \pi / 4$.

The class of functions $f$ which are continuous on
$\R_+ := [ 0, \infty )$, smooth on the open semi-axis $( 0, \infty )$,
and such that
\[
( -1 )^n f^{( n )}( t ) \geq 0 \qquad \text{for all } t > 0
\]
was studied by S.~Bernstein, who proved that they coincide with Laplace
transforms of positive measures on $\R_+$:
\begin{equation}\label{Ebernstein}
f( t ) = \int_0^\infty e^{-t u} \std\mu( u ).
\end{equation}
Such functions are called \emph{completely monotonic} and have proved
highly relevant for probability theory and approximation theory;
see~\cite{Bernstein-Acta29} for the foundational reference. Thus we
have obtained a valuable equivalence.

\begin{theorem}[Schoenberg]
A function $f$ is completely monotone if and only if
$t \mapsto f( t^2 )$ is positive definite on Hilbert space.
\end{theorem}

The direct consequences of this apparently innocent observation are
quite deep. For example, the isometric-embedding question for altered
Euclidean distances is completely answered via this route. The
following results are from \cite{Schoenberg-Annals38} and
\cite{vonNeumann-Schoenberg-TAMS41}.

\begin{theorem}[Schoenberg--von Neumann]
Let $H$ be a separable Hilbert space with norm $\| \cdot \|$.
\begin{enumerate}
\item For any integers $n \geq d > 1$, the metric space
$( \R^d, \phi\bigl( \| \cdot \| \bigr) )$ may be isometrically
embedded into $( \R^n, \| \cdot \| )$ if and only if $\phi( t ) = c t$
for some $c > 0$.

\item The metric space $( \R^d, \phi\bigl( \| \cdot \| \bigr) )$ 
may be isometrically embedded into~$H$ if and only if
\[
\phi( t )^2 = %
\int_0^\infty \frac{1 - \Omega_d( t u )}{u^2} \std\mu( u ),
\]
where $\mu$ is a positive measure on the semi-axis such that
\[
\int_1^\infty \frac{1}{u^2} \std\mu( u ) < \infty.
\]

\item The metric space $( H, \phi\bigl (\| \cdot \| \bigr) )$ may be
isometrically embedded into $H$ if and only if
\[
\phi( t )^2 = \int_0^\infty \frac{1 - e^{-t^2 u}}{u} \std\mu( u ),
\]
where $\mu$ is a positive measure on the semi-axis such that
\[
\int_1^\infty \frac{1}{u} \std\mu( u ) < \infty.
\]
\end{enumerate}
\end{theorem}

In von Neumann and Schoenberg's article
\cite{vonNeumann-Schoenberg-TAMS41}, special attention is paid to the
case of embedding a modified distance on the line into Hilbert
space. This amounts to characterizing all \emph{screw lines} in a
Hilbert space $H$: the continuous functions
\[
f : \R \to H; \ t \mapsto f_t
\]
with the translation-invariance property
\[
\| f_s - f_t \| = \| f_{s + r} - f_{t + r}\| %
\qquad \text{for all } s, r, t \in \R.
\]
In this case, the gauge function $\phi$ is such that
$\phi( t - s ) = \| f_s - f_t \|$ and $t \mapsto f_t$ provides the
isometric embedding of $( \R, \phi\bigl( | \cdot | \bigr) )$ into $H$.
Von Neumann seized the opportunity to use Stone's theorem on
one-parameter unitary groups, together with the spectral decomposition of
their unbounded self-adjoint generators, to produce a purely
operator-theoretic proof of the following result.

\begin{corollary}
The metric space $( \R, \phi\bigl( | \cdot | ) \bigr)$ isometrically
embeds into Hilbert space if and only if
\[
\phi( t )^2 = \int_0^\infty \frac{\sin^2( t u )}{u^2} \std\mu( u ) %
\qquad (t \in \R),
\]
where $\mu$ is a positive measure on $\R_+$ satisfying
\[
\int_1^\infty \frac{1}{u^2} \std\mu( u ) < \infty.
\]
\end{corollary} 

Moreover, in the conditions of the corollary, the space
$( \R, \phi\bigl( | \cdot | \bigr) )$ embeds isometrically into $\R^d$
if and only if the measure $\mu$ consists of finitely many point
masses, whose number is roughly $d / 2$; see
\cite[Theorem~2]{vonNeumann-Schoenberg-TAMS41} for the precise
statement. To give a simple example, consider the function
\[
\phi : \R \to \R_+; \ t \mapsto \sqrt{t^2 + \sin^2 t}.
\]
This is indeed a screw function, because
\begin{align*}
\phi( t - s )^2  &  = ( t - s )^2 + \sin^2( t - s ) \\
& = ( t - s )^2 + \frac{1}{4} %
\bigl( \cos( 2 t ) - \cos( 2 s ) \bigr)^2 + \frac{1}{4} %
\bigl( \sin( 2 t ) - \cos( 2 s ) \bigr)^2.
\end{align*}

Note that a screw line is periodic if and only if it is not
injective. Furthermore, one may identify screw lines with period
$\tau > 0$ by the geometry of the support of the representing
measure: this support must be contained in the lattice
$( \pi / \tau ) \Z_+$, where $\Z_+ := \Z \cap \R_+ = \{ 0, 1, 2, \dots
\}$.  Consequently, all periodic screw lines in
Hilbert space have a gauge function $\phi$ such that
\begin{equation}\label{periodic-screws}
\phi( t )^2 = \sum_{k = 1}^\infty c_k \sin^2( k \pi t / \tau ),
\end{equation}
where $c_k \geq 0$ and $\sum_{k = 1}^\infty c_k < \infty$; see
\cite[Theorem~5]{vonNeumann-Schoenberg-TAMS41}.

\subsection{Positive definite functions on homogeneous spaces}
 
Having resolved the question of isometrically embedding Euclidean
space into Hilbert space, a natural desire was to extend the analysis
to other special manifolds with symmetry. This was done almost
simultaneously by Schoenberg on spheres~\cite{Schoenberg-Duke42} and
by Bochner on compact homogeneous spaces~\cite{Bochner-Annals41}.
 
Let $X$ be a compact space endowed with a transitive action of a group
$G$ and an invariant measure. We seek $G$-invariant distance
functions, and particularly those which identify $X$ with a subspace
of a Hilbert space. To simplify terminology, we call the latter
\emph{Hilbert distances}.

The first observation of Bochner is that a $G$-invariant symmetric
kernel $f: X \times X \to \R$ satisfies the Hilbert-space
embeddability condition,
\[
\sum_{k = 0}^n c_k = 0 \quad \implies \quad %
\sum_{j, k = 0}^n f( x_j, x_k ) c_j c_k \geq 0,
\]
for all choices of weights $c_j$ and points $x_j \in X$, if and only if
$f$ is of the form
\[
f( x, y ) = h( x, y ) - h( x_0, x_0) \qquad ( x, y \in X ),
\]
where $h$ is a $G$-invariant positive definite kernel and $x_0$ is a
point of $X$. One implication is clear. For the other, we start with a
$G$-invariant function $f$ subject to the above constraint and prove,
using $G$-invariance and integration over $X$, the existence of a
constant $c$ such that $h( x, y ) = f( x, y ) + c$ is a
positive semidefinite kernel. This gives the following result.
 
\begin{theorem}[Bochner~\cite{Bochner-Annals41}]
Let $X$ be a compact homogeneous space. A continuous invariant
function $\rho$ on $X \times X$ is a Hilbert distance if and only if
there exists a continuous, real-valued, invariant, positive definite
kernel $h$ on $X$ and a point $x_0 \in X$, such that
\[
\rho( x, y ) = \sqrt{h( x_0, x_0 ) - h( x, y )} \qquad %
( x, y \in X ).
\]
\end{theorem}
 
Privileged orthonormal bases of $G$-invariant functions, in the $L^2$
space associated with the invariant measure, provide a
canonical decompositions of positive definite kernels. These generalized
spherical harmonics were already studied by E.~Cartan, H.~Weyl and
J.~von Neumann; see, for instance~\cite{Weyl-Annals34}. We elaborate
on two important particular cases.

Let $X = \T = \{ e^{\I \theta} : \theta \in \R \}$ be the unit torus,
endowed with the invariant arc-length measure. A continuous
positive definite function $h : \T \times \T \to \R$ admits a Fourier
decomposition
\[
h( e^{\I x}, e^{\I y} ) = %
\sum_{j, k \in \Z} a_{j k} e^{\I j x} e^{-\I k y}.
\]
If $h$ is further required to be rotation invariant, we find that
\[
h( e^{\I x}, e^{\I y} ) = \sum_{k \in \Z} a_k e^{\I k ( x - y ) },
\]
where $a_k \geq 0$ for all $k \in \Z$ and $a_k = a_{-k}$ because $h$
takes real values. Moreover, the series is Abel summable:
$\sum_{k = 0}^\infty a_k = h( 1, 1 ) < \infty$. Therefore, a
rotation-invariant Hilbert distance $\rho$ on the torus has the
expression (after taking its square):
\begin{align*}
\rho( e^{\I x}, e^{\I y} )^2 = h( 1, 1 ) - h( e^{\I x}, e^{\I y}) & = %
\sum_{k = 1}^\infty a_k ( 2 - e^{\I k ( x - y ) } - %
e^{-\I k ( x - y ) }) \\
 & = 2 \sum_{k = 1}^\infty a_k ( 1 - \cos k ( x - y ) ) \\
& = 4 \sum_{k = 1}^\infty a_k \sin^2( k ( x - y ) / 2 ).
\end{align*}
These are the periodic screw lines (\ref{periodic-screws}) already
investigated by von Neumann and Schoenberg.

As a second example, we follow Bochner in examining a separable,
compact group $G$. A real-valued, continuous, positive definite and
$G$-invariant kernel $h$ admits the decomposition
\[
h( x, y ) = \sum_{k \in \Z} c_k \chi_k( y x^{-1} ),
\]
where $c_k \geq 0$ for all $k \in \Z$, $\sum_{k \in Z} c_k < \infty$
and $\chi_k$ denote the characters of irreducible representations
of~$G$. In conclusion, an invariant Hilbert distance~$\rho$ on~$G$ is
characterized by the formula
\[
\rho( x, y )^2 = \sum_{k \in \Z} c_k \bigl( %
1- \frac{ \chi_k( y x^{-1}) + \chi_k( x y^{-1} )}{2 \chi_k(1)} \bigr),
\]
where $c_k \geq 0$ and $\sum_{k \in \Z} c_k <\infty$.

For details and an analysis of similar decompositions on more general
homogeneous spaces, we refer the reader to \cite{Bochner-Annals41}.

The above analysis of positive definite functions on homogeneous
spaces was carried out separately by Schoenberg in
\cite{Schoenberg-Duke42}. First, he remarks that a continuous,
real-valued, rotationally invariant and positive definite kernel~$f$
on the sphere $S^{d - 1}$ has a distinguished Fourier-series
decomposition with non-negative coefficients. Specifically,
\begin{equation}\label{EschoenbergSphere}
f( \cos \theta ) = %
\sum_{k = 0}^\infty c_k P_k^{(\lambda)}( \cos \theta )
\end{equation}
where $\lambda = ( d - 2 ) / 2$, $P^{(\lambda)}_k$ are the
ultraspherical orthogonal polynomials, $c_k \geq 0$ for all $k \geq 0$
and $ \sum_{k=0}^\infty c_k < \infty$. This decomposition is in accord
with Bochner's general framework, with the difference lying in
Schoenberg's elementary proof, based on induction on dimension.
As with all our formulas concerning the sphere, $\theta$ represents
the geodesic distance (arc length along a great circle) between two
points.

To convince the reader that expressions in the cosine of the geodesic
distance are positive definite, let us consider points $x_1$, \ldots,
$x_n \in S^{d-1}$. The Gram matrix with entries
\[
\langle x_j, x_k \rangle = \cos \theta( x_j, x_k )
\]
is obviously positive semidefinite, with constant diagonal elements
equal to~$1$. According to the Schur product theorem \cite{Schur1911},
all functions of the form $\cos^k \theta$, where $k$ is a non-negative
integer, are therefore positive definite on the sphere.

At this stage, Schoenberg makes a leap forward and studies invariant
positive definite kernels on $S^\infty$, that is, functions
$f( \cos \theta )$ which admit representations as above for all
$d \geq 2$. His conclusion is remarkable in its simplicity.

\begin{theorem}[Schoenberg \cite{Schoenberg-Duke42}]
A real-valued function $f( \cos \theta )$ is positive definite on all
spheres, independent of their dimension, if and only if
\begin{equation}\label{cos-decomposition}
f( \cos \theta ) = \sum_{k = 0}^\infty c_k \cos^k \theta,
\end{equation}
where $c_k \geq 0$ for all $k \geq 0$ and
$\sum_{k = 0}^\infty c_k < \infty$.
\end{theorem}

This provides a return to the dominant theme, of isometric embedding
into Hilbert space.

\begin{corollary}
The function $\rho( \theta )$ is a Hilbert distance on $S^\infty$ if
and only if
\[
\rho( \theta )^2 = \sum_{k = 0}^\infty c_k ( 1 - \cos^k \theta ),
\]
where $c_k \geq 0$ for all $k \geq 0$ and
$\sum_{k = 0}^\infty c_k < \infty$.
\end{corollary}

However, there is much more to derive from Schoenberg's theorem, once
it is freed from the spherical context.

\begin{theorem}[Schoenberg \cite{Schoenberg-Duke42}]\label{Tschoenberg}
Let $f: [ -1, 1 ] \to \R$ be a continuous function. If the matrix
$[ f( a_{j k}) ]_{j, k = 1}^n$ is positive semidefinite for all $n
\geq 1$ and all positive semidefinite matrices
$[ a_{j k} ]_{j, k = 1}^n$ with entries in $[ -1, 1 ]$, then, and only
then,
\[
f( x ) = \sum_{k = 0}^\infty c_k x^k \qquad ( x \in [ -1, 1 ] ),
\]
where $c_k \geq 0$ for all $k \geq 0$ and
$\sum_{k = 0}^\infty c_k < \infty$.
\end{theorem}

\begin{proof}
One implication follows from the Schur product theorem \cite{Schur1911},
which says that if the $n \times n$ matrices $A$ and $B$ are positive
semidefinite, then so is their entrywise product
$A \circ B := [ a_{j k} b_{j k} ]_{j, k = 1}^n$. Indeed, inductively
setting $B = A^{\circ k} = A \circ \cdots \circ A$, the $k$-fold
entrywise power, shows that every monomial $x^k$ preserves positivity
when applied entrywise.  That the same property holds for functions
$f( x ) = \sum_{k \geq 0} c_k x^k$, with all $c_k \geq 0$, now follows
from the fact that the set of positive semidefinite $n \times n$
matrices forms a closed convex cone, for all $n \geq 1$.

For the non-trivial, reverse implication we restrict the test matrices
to those with leading diagonal terms all equal to $1$. By interpreting
such a matrix $A$ as a Gram matrix, we identify $n$ points on the
sphere $x_1$, \ldots, $x_n \in S^{n - 1}$ satisfying
\[
a_{j k} = \langle x_j, x_k \rangle = \cos \theta(x_j,x_k)%
\quad ( 1 \leq j, k \leq n ).
\]
Then we infer from Schoenberg's theorem that $f$ admits a
uniformly convergent Taylor series with non-negative coefficients.
\end{proof}

We conclude this section by mentioning some recent avenues of research
that start from Bochner's theorem (and its generalization in 1940, by
Weil, Povzner, and Raikov, to all locally compact abelian groups) and
Schoenberg's classification of positive definite functions on spheres.
On the theoretical side, there has been a profusion of recent
mathematical activity on classifying positive definite functions (and
strictly positive definite functions) in numerous settings, mostly
related to spheres~\cite{BC,BCX,CMS,Xu,XC,Ziegel}, two-point
homogeneous spaces\footnote{Recall \cite{Wang} that a metric space
$( X, \rho )$ is \emph{$n$-point homogeneous} if, given finite
sets $X_1$, $X_2 \subset X$ of equal size no more than $n$, every
isometry from $X_1$ to $X_2$ extends to a self-isometry of $X$. This
property was first considered by Birkhoff~\cite{Birkhoff}, and of
course differs from the more common usage of the terminology of a
homogeneous space $G / H$, whose study by Bochner was mentioned
above.}\cite{Barbosa1,Barbosa2,BGM}, 
locally compact abelian groups and homogeneous
spaces~\cite{Emonds-Fuhr,Guella-2017}, and products of
these~\cite{BPP,Berg-Porcu,Guella-2016,GMP1,GMP3,GMP2}.

Moreover, this line of work directly impacts applied fields. For
instance, in climate science and geospatial statistics, one uses
positive definite kernels and Schoenberg's results (and their sequels)
to study trends in climate behavior on the Earth, since it can be
modelled by a sphere, and positive definite functions on
$S^2 \times \R$ characterize space-time covariance functions on
it. See~\cite{Gneiting2013, MNPR, PAF, PBG} for more details on these
applications. There is a natural connection to probability theory,
through the work of L\'evy; see e.g.~\cite{Gangolli}. Other applied
fields include genomics and finance, through high-dimensional covariance
estimation. We elaborate on this in Chapter~\ref{Sstats} below.

There are several other applications of Schoenberg's work on positive
definite functions on spheres (his paper~\cite{Schoenberg-Duke42} has
more than 160 citations) and we mention here just a few of them.
Schoenberg's results were used by Musin~\cite{Musin-Annals} to compute
the kissing number in four dimensions, by an extension of Delsarte's
linear-programming method. Moreover, the results also apply to obtain
new bounds on spherical codes \cite{Musin}, with further applications
to sphere packing~\cite{Cohn1, Cohn2, Cohn3, CZ}.
There are also applications to
approximating functions and interpolating data on spheres,
pseudodifferential equations with radial basis functions,
and Gaussian random fields.

\begin{remark}\label{Rpinkus}
Another modern-day use of Schoenberg's results
in~\cite{Schoenberg-Duke42} is in Machine Learning;
see~\cite{Steinwart,Vapnik}, for example. Given a real inner-product
space $H$ and a function $f : \R \to \R$, an alternative notion of $f$
being \emph{positive definite} is as follows: for any finite set of
vectors $x_1$, \ldots, $x_n \in H$, the matrix
\[
[ f( \langle x_j, x_k \rangle ) ]_{j, k = 1}^n
\]
is positive semidefinite. This is in contrast to the notion promoted
by Bochner, Weil, Schoenberg, P\'olya, and others, which concerns
positivity of the matrix with entries
$f( \langle x_j - x_k, x_j - x_k \rangle^{1/2} )$.  It turns out that
every positive definite kernel on $H$, given by
\[
( x, y ) \mapsto f( \langle x, y \rangle )
\]
for a function $f$ which is positive definite in this alternate sense,
gives rise to a reproducing-kernel Hilbert space, which is a central
concept in Machine Learning. We restrict ourselves here to mentioning
that, in this setting, it is desirable for the kernel to be strictly
positive definite; see~\cite{pinkus04} for further clarification and
theoretical results along these lines.
\end{remark}

\subsection{Connections to harmonic analysis}

Positivity and sharp continuity bounds for linear transformations
between specific normed function spaces go hand in hand, especially
when focusing on the kernels of integral transforms. The end of 1950s
marked a fortunate condensation of observations, leading to a
quasi-complete classification of preservers of positive or bounded
convolution transforms acting on spaces of functions on locally
compact abelian groups. In particular, these results can be
interpreted as Schoenberg-type theorems for Toeplitz matrices or
Toeplitz kernels. We briefly recount the main developments.

A groundbreaking theorem of the 1930s attributed to Wiener and Levy
asserts that the pointwise inverse of a non-vanishing Fourier series
with coefficients in $L^1$ exhibits the same summability behavior of
the coefficient sequence. To be more precise, if $\phi$ is never zero and
has the representation
\[
\phi( \theta ) = \sum_{n = -\infty}^\infty c_n e^{\I n \theta}, %
\quad \text{where } \sum_{n= -\infty}^\infty | c_n | < \infty,
\]
then its reciprocal has a representation of the same form:
\[
( 1 / \phi )( \theta ) = \sum_{n= -\infty}^\infty d_n e^{\I n \theta},
\quad \text{where } \sum_{n= -\infty}^\infty | d_n | < \infty.
\]
It was Gelfand \cite{Gelfand-Sbornik41} who in 1941 cast this
permanence phenomenon in the general framework of commutative Banach
algebras. Gelfand's theory applied to the Wiener algebra
$W := \widehat{L^1( \Z )}$ of Fourier transforms of $L^1$ functions on
the dual of the unit torus proves the following theorem.

\begin{theorem}[Gelfand \cite{Gelfand-Sbornik41}]
Let $\phi \in W$ and let $f( z )$ be an analytic function defined in a
neighborhood of $\phi(\T)$. Then $f( \phi ) \in W$.
\end{theorem}

The natural inverse question of deriving smoothness properties of
inner transformations of Lebesgue spaces of Fourier transforms was
tackled almost simultaneously by several analysts.  For example, Rudin
proved in 1956 \cite{Rudin-CRAS56} that a coefficient-wise
transformation $c_n \mapsto f( c_n )$ mapping the space
$\widehat{L^1( \T )}$ into itself implies the analyticity of $f$ in a
neighborhood of zero. In a similar vein, Rudin and Kahane proved in
1958 \cite{Kahane-Rudin-CRAS58} that a coefficient-wise transformation
$c_n \mapsto f(c_n)$ which preserves the space of Fourier transforms
$\widehat{M( \T )}$ of finite measures on the torus implies that $f$ is
an entire function. In the same year, Kahane \cite{Kahane-CRAS58}
showed that no quasi-analytic function (in the sense of
Denjoy--Carleman) preserves the space $\widehat{L^1(\Z)}$ and
Katznelson \cite{Katznelson-CRAS58} refined an inverse to Gelfand's
theorem above, by showing the semi-local analyticity of transformers
of elements of $\widehat{L^1( \Z )}$ subject to some support
conditions.

Soon after, the complete picture emerged in full clarity. It was
unveiled by Helson, Kahane, Katznelson and Rudin in an Acta
Mathematica article \cite{Helson-all-Acta59}. Given a function $f$
defined on a subset $E$ of the complex plane, we say that $f$ operates
on the function algebra $A$, if $f( \phi ) \in A$ for every
$\phi \in A$ with range contained in $E$. The following metatheorem is
proved in the cited article.

\begin{theorem}[Helson--Kahane--Katznelson--Rudin \cite{Helson-all-Acta59}]
Let $G$ be a locally compact abelian group and let $\Gamma$ denote its
dual, and suppose both are endowed with their respective Haar
measures. Let $f : [-1,1] \to \C$ be a function satisfying $f(0) = 0$.
\begin{enumerate}
\item If $\Gamma$ is discrete and $f$ operates on
$\widehat{L^1( G )}$, then $f$ is analytic in some neighborhood of
the origin.

\item If $\Gamma$ is not discrete and $f$ operates on
$\widehat{L^1(G)}$, then $f$ is analytic in $[-1,1]$.

\item If $\Gamma$ is not compact and $f$ operates on
$\widehat{M(G)}$, then $f$ can be extended to an entire function. 
\end{enumerate}
\end{theorem}

Rudin refined the above results to apply in the case of various $L^p$
norms \cite{Rudin-PAMS59,Rudin-Canadian62}, by stressing the lack of
continuity assumption for the transformer $f$ in all results (similar
in nature to the statements in the above theorem). From Rudin's work
we extract a highly relevant observation, \textit{\`a la} Schoenberg's
theorem, aligned to the spirit of the present survey.

\begin{theorem}[Rudin~\cite{Rudin-Duke59}]\label{Trudin}
Suppose $f : (-1,1) \to \R$ maps every positive semidefinite Toeplitz
kernel with elements in $( -1, 1 )$ into a positive semidefinite
kernel:
\[
[ a_{j - k} ]_{j, k = -\infty}^\infty \geq 0 \qquad \implies \qquad %
[ f( a_{j - k} ) ]_{j, k = -\infty}^\infty \geq 0.
\]
Then $f$ is absolutely monotonic, that is analytic on $( -1, 1 )$ with
a Taylor series having  non-negative coefficients:
\[
f( x ) = \sum_{n = 0}^\infty c_k x^k, \qquad %
\text{where } c_k \geq 0 \text{ for all } k \geq 0.
\]
\end{theorem}

The converse is obviously true by the Schur product theorem. The
elementary proof, quite independent of the derivation of the
metatheorem stated above, is contained in \cite{Rudin-Duke59}. Notice
again the lack of a continuity assumption in the hypotheses.

In fact, Rudin proves more, by restricting the test domain of positive
semidefinite Toeplitz kernels to the two-parameter family
\begin{equation}\label{Erudin}
a_n = \alpha + \beta \cos( n \theta ) \qquad (n \in \Z)
\end{equation}
with $\theta$ fixed so that $\theta / \pi$ is irrational and $\alpha$,
$\beta \geq 0$ such that $\alpha + \beta < 1$. Rudin's proof commences
with a mollifier argument to deduce the continuity of the transformer,
then uses a development in spherical harmonics very similar to the
original argument of Schoenberg. We will resume this topic in
Section~\ref{S33}, setting it in a wider context.

With the advances in abstract duality theory for locally convex
spaces, it is not surprising that proofs of Schoenberg-type theorems
should be accessible with the aid of such versatile tools. We will
confine ourselves here to mentioning one pertinent convexity-theoretic
proof of Schoenberg's theorem, due to Christensen and Ressel
\cite{Christensen-Ressel-TAMS78}. (See also \cite{Christensen-Ressel-2}
for a complex sphere variant.)

Skipping freely over the details, the main observation of these two
authors is that the multiplicatively closed convex cone of positivity
preservers of positive semidefinite matrices of any size, with entries
in $[ -1, 1 ]$, is closed in the product topology of $\R^{[ -1, 1 ]}$,
with a compact base $K$ defined by the normalization $f( 1 ) = 1$. The
set of extreme points of $K$ is readily seen to be closed, and an
elementary argument identifies it as the set of all monomials $x^n$,
where $n \geq 0$, plus the characteristic functions $\chi_{1} \pm
\chi_{-1}$. An application of Choquet's representation theorem now
provides a proof of a generalization of Schoenberg's theorem,
by removing the continuity assumption in the statement.


\section{Entrywise functions preserving positivity in all dimensions}

\subsection{History}\label{Shistory}

With the above history to place the present survey in context, we move
to its dominant theme: entrywise positivity preservers. In analysis
and in applications in the broader mathematical sciences, one is
familiar with applying functions to the spectrum of diagonalizable
matrices: $A = U D U^*$ then $f( A ) = U f( D ) U^*$. More formally,
one uses the Riesz--Dunford holomorphic functional calculus to define
$f( A )$ for classes of matrices $A$ and functions $f$.

Our focus in this survey will be on the parallel philosophy of
\emph{entrywise calculus}. To differentiate this from the functional
calculus, we use the notation $f[ A ]$.

\begin{definition}
Fix a domain $I \subset \C$ and integers $m$, $n \geq 1$. Let
$\cP_n( I )$ denote the set of $n \times n$ Hermitian
positive semidefinite matrices with all entries in~$I$.

A function $f : I \to \C$ acts \emph{entrywise} on a matrix
\[
A = %
[ a_{j k} ]_{1 \leq j \leq m, \ 1 \leq k \leq n} \in I^{m \times n}
\]
by setting
\[
f[ A ] := [ f( a_{j k} ) ]_{1 \leq j \leq m, \ 1 \leq k \leq n} %
\in \C^{m \times n}.
\]
Below, we allow the dimensions $m$ and $n$ to vary, while keeping the
uniform notation $f[-]$.

We also let $\bone_{m \times n}$ denote the $m \times n$ matrix with
each entry equal to one. Note that
$\bone_{n \times n} \in \cP_n( \R )$.
\end{definition}

In this survey, we explore the following overarching question in
several different settings.

\emph{Which functions preserve positive semidefiniteness when applied
entrywise to a class of positive matrices?}

This question was first asked by P\'olya and Szeg\"o in their
well-known book \cite{polya-szego}. The authors observed that Schur's
product theorem, together with the fact that the positive matrices
form a closed convex cone, has the following consequence: if $f( x )$
is any power series with non-negative Maclaurin coefficients that
converges on a domain $I \subset \R$, then $f$ preserves positivity
(that is, preserves positive semidefiniteness) when applied entrywise
to positive semidefinite matrices with entries in~$I$.  P\'olya and
Szeg\"o then asked if there are any other functions that possess this
property. As discussed above, Schoenberg's theorem
\ref{Tschoenberg} provides a definitive answer to their question
(together with the improvements by Rudin or Christensen--Ressel to remove
the continuity hypothesis). Thanks to P\'olya and Szeg\"o's observation,
Schoenberg's result may be considered as a rather challenging converse to
the Schur product theorem.

In a similar vein, Rudin \cite{Rudin-Duke59} observed that if one
moves to the complex setting, then the conjugation map also preserves
positivity when applied entrywise to positive semidefinite complex
matrices. Therefore the maps
\[
z \mapsto z^j \overline{z}^k \qquad ( j, k \geq 0 )
\]
preserve positivity when applied entrywise to complex matrices of all
dimensions, again by the Schur product theorem. The same property is
now satisfied by non-negative linear combinations of these
functions. In \cite{Rudin-Duke59}, Rudin made this observation and
conjectured, \textit{\`a la} P\'olya--Szeg\"o, that these are all of
the preservers. This was proved by Herz in 1963.

\begin{theorem}[Herz \cite{Herz63}]\label{Therz}
Let $D( 0, 1 )$ denote the open unit disc in $\C$, and suppose
$f : D( 0, 1 ) \to \C$. The entrywise map $f[-]$ preserves positivity
on $\cP_n\bigl( D( 0, 1 ) \bigr)$ for all $n \geq 1$, if and only if
\[
f( z ) = \sum_{j, k \geq 0} c_{j k} z^j \overline{z}^k \qquad %
\text{for all } z \in D( 0, 1 ),
\]
where $c_{j k} \geq 0$ for all $j$, $k \geq 0$.
\end{theorem}

Akin to the above results by Schoenberg, Rudin, Christensen and
Ressel, and Herz, we mention one more Schoenberg-type theorem, for
matrices with positive entries. The following result again
demonstrates the rigid principle that analyticity and absolute
monotonicity follow from the preservation of positivity in all
dimensions.

\begin{theorem}[Vasudeva \cite{vasudeva79}]\label{Tvasudeva}
Let $f : ( 0, \infty ) \to \R$. Then $f[-]$ preserves positivity on
$\cP_n\bigl( ( 0, \infty ) \bigr)$ for all $n \geq 1$, if and only if
$f( x ) = \sum_{k = 0}^\infty c_k x^k$ on $( 0, \infty )$, where
$c_k \geq 0$ for all $k \geq 0$.
\end{theorem}

\subsection{The Horn--Loewner necessary condition in fixed dimension}

The previous section contains several variants of a ``dimension-free''
result: namely, the classification of entrywise maps that preserve
positivity on test sets of matrices of all sizes. In the next section, we
discuss a dimension-free result that parallels Rudin's work in
\cite{Rudin-Duke59}, by approaching the problem via preservers of moment
sequences for positive measures on the real line. In other words, we will
work with Hankel instead of Toeplitz matrices.

In the later part of this survey, we focus on entrywise functions that
preserve positivity when the test set consists of matrices of a fixed
size. For both of these settings, the starting point is an important
result first published by R.~Horn (who in \cite{Horn67}
attributes it to his PhD advisor C.~Loewner).

\begin{theorem}[\cite{Horn67}]\label{Thorn}
Let $f : ( 0, \infty ) \to \R$ be continuous. Fix a positive integer
$n$ and suppose $f[-]$ preserves positivity on
$\cP_n\bigl( ( 0, \infty ) \bigr)$. Then $f \in C^{n - 3}((0, \infty))$,
\[
f^{( k )}( x ) \geq 0 \qquad \text{whenever } x \in ( 0, \infty ) %
\text{ and } 0 \leq k \leq n - 3,
\]
and $f^{(n - 3)}$ is a convex non-decreasing function on
$( 0, \infty )$. Furthermore, if
$f \in C^{n - 1}\bigl( ( 0, \infty ) \bigr)$, then
$f^{( k )}( x ) \geq 0$ whenever $x \in ( 0, \infty )$ and
$0 \leq k \leq n-1$.
\end{theorem}

This result and its variations are the focus of the present section.

Theorem~\ref{Thorn} is remarkable for several reasons.

\begin{enumerate}
\item Modulo variations, it remains to this day the only known
criterion for a general entrywise function to preserve positivity in a
fixed dimension. Later on, we will see more precise conclusions drawn when
$f$ is a polynomial or a power function, but for a general function
there are essentially no other known results.

\item While Theorem~\ref{Thorn} is a fixed-dimension result, it can be
used to prove some of the aforementioned dimension-free
characterizations. For instance, if $f[-]$ preserves positivity on
$\cP_n\bigl( ( 0, \infty ) \bigr)$ for all $n \geq 1$, then, by
Theorem~\ref{Thorn}, the function $f$ is absolutely monotonic on
$( 0, \infty )$. A classical result of Bernstein on absolutely
monotonic functions now implies that $f$ is necessarily given by
a power series with non-negative coefficients, which is precisely
Vasudeva's Theorem~\ref{Tvasudeva}.

In the next section, we will outline an approach to prove a stronger
version of Schoenberg's theorem~\ref{Tschoenberg} (in the spirit of
Theorem~\ref{Trudin} by Rudin), starting from Theorem~\ref{Tvasudeva}.

\item Theorem~\ref{Thorn} is also significant because there is a sense
in which it is sharp. We elaborate on this when studying polynomial and
power-function preservers; see Chapters~\ref{Spolynomial}
and~\ref{Spower}.
\end{enumerate}

\begin{remark}\label{Rsymm}
There are other, rather unexpected consequences of Theorem~\ref{Thorn}
as well. It was recently shown that the key determinant computation
underlying Theorem~\ref{Thorn} can be generalized to yield a new class
of symmetric function identities for any formal power series. The only
such identities previously known were for the case
$f( x ) = \frac{1 - c x}{1 - x}$. This is discussed in
Section~\ref{Ssymm}.
\end{remark}

We next explain the steps behind the proof of the Horn--Loewner
theorem~\ref{Thorn}. These also help in proving certain strengthenings
of Theorem~\ref{Thorn}, which are mentioned below. In turn, these
strengthenings additionally serve to clarify the nature of the
Horn--Loewner necessary condition.

\begin{proof}[Proof of Theorem~\ref{Thorn}]
The proof by Loewner is in two steps. First he assumes $f$ to be
smooth and shows the result by induction on $n$. The base case of
$n = 1$ is immediate, and for the induction step one proceeds as
follows. Fix $a>0$, choose any vector
$\bu = (u_1, \ldots, u_n)^T \in \R^n$ with distinct coordinates, and
define
\[
\Delta( t ) := \det[ f( a + t u_j u_k ) ]_{j, k = 1}^n = %
\det f[ a \bone_{n \times n} + \bu \bu^T ] \qquad ( 0 < t \ll 1).
\]
Then Loewner shows that
\begin{align}\label{Eloewner}
\begin{aligned}
\Delta(0) & = \Delta'( 0 ) = \cdots = %
\Delta^{\binom{n}{2} - 1}( 0 ) = 0, \\
\Delta^{\binom{n}{2}}(0) & = c f( a ) f'( a ) \cdots f^{(n - 1)}(a)
\qquad \text{for some } c > 0.
\end{aligned}
\end{align}
(See Remark~\ref{Rsymm} above.)

Returning to the proof of Theorem~\ref{Thorn} for smooth functions:
apply the above treatment not to $f$ but to
$g_\tau( x ) := f( x ) + \tau x^n$, where $\tau > 0$. By the Schur
product theorem, $g_\tau$ satisfies the hypotheses, whence
$\Delta( t ) / t^{\binom{n}{2}} \geq 0$ for $t > 0$.  Taking
$t \to 0^+$, by L'H\^{o}pital's rule we obtain
\[
g_\tau( a ) g'_\tau( a ) \cdots g^{(n - 1)}_\tau( a ) \geq 0, %
\qquad \text{for all } \tau > 0.
\]
Finally, the induction hypothesis implies that $f$, $f'$, \ldots,
$f^{(n - 2)}$ are non-negative at $a$, whence $g_\tau( a )$, \ldots,
$g_\tau^{(n - 2)}( a ) > 0$. It follows that
$g_\tau^{(n - 1)}( a ) \geq 0$ for all $\tau > 0$, and hence,
$f^{(n - 1)}( a ) \geq 0$, as desired.

\begin{remark}
The above argument is amenable to proving more refined results. For
example, it can be used to prove the positivity of the first $n$ non-zero
derivatives of a smooth preserver $f$; see Theorem~\ref{Tmaster}.
\end{remark}

The second step of Loewner's proof begins by using mollifiers. 
Suppose $f$ is
continuous; approximate it by a mollified family $f_\delta \to f$ as
$\delta \to 0^+$. Thus $f_\delta$ is smooth and its first $n$
derivatives are non-negative on $(0,\infty)$. By the mean-value
theorem for divided differences, this implies that the divided
differences of each $f_\delta$, of orders up to $n-1$ are
non-negative. Since $f$ is continuous, the same holds for $f$.

Now one invokes a rather remarkable result by Boas and Widder
\cite{Boas-Widder}, which can be viewed as a converse to the
mean-value theorem for divided differences. It asserts that given an
integer $k \geq 2$ and an open interval $I \subset \R$, if all $k$th
order ``equi-spaced'' forward differences (whence divided differences) of
a continuous function $f : I \to \R$ are non-negative on $I$, then $f$ is
$k-2$ times differentiable on $I$; moreover, $f^{(k-2)}$ is continuous
and convex on $I$, with non-decreasing left- and right-hand
derivatives. Applying this result for each $2 \leq k \leq n-1$ concludes
the proof of Theorem~\ref{Thorn}.
\end{proof}

Note that this proof only uses matrices of the form
$a \bone_{n \times n} + t \bu \bu^T$, and the arguments are all
local. Thus it is unsurprising that strengthened versions of the
Horn--Loewner theorem can be found in the literature; see
\cite{BGKP-hankel,GKR-lowrank}, for example. We present here the
stronger of these variants.

\begin{theorem}[{See \cite[Section 3]{BGKP-hankel}}]\label{Thorn2}
Suppose $0 < \rho \leq \infty$, $I = ( 0, \rho )$, and $f : I \to \R$.
Fix $u_0 \in ( 0, 1 )$ and an integer $n \geq 1$, and define
$\bu := ( 1, u_0, \ldots, u_0^{n - 1} )^T$. Suppose
$f[ A ] \in \cP_2( \R )$ for all $A \in \cP_2( I )$, and also that
$f[ A ] \in \cP_n( \R )$ for all Hankel matrices
$A = a \bone_{n \times n} + t \bu \bu^T$, with $a$, $t \geq 0$ such
that $a + t \in I$. Then the conclusions of Theorem~\ref{Thorn} hold.
\end{theorem}

Beyond the above strengthenings, the notable feature here is that the
continuity hypothesis has been removed, akin to the Rudin and
Christensen--Ressel results. We reproduce here an elegant argument to
show continuity; this can be found in Vasudeva's paper
\cite{vasudeva79}, and uses only the test set $\cP_2( I )$. By
considering $f[ A ]$ for
$A = \begin{bmatrix} a & b \\ b & a \end{bmatrix}$ with
$0 < b < a < \rho$, it follows that $f$ is non-negative and
non-decreasing on $I$. One also shows that $f$ is either identically
zero or never zero on $I$. In the latter case, considering $f[ A ]$
for
$A = \begin{bmatrix} a & \sqrt{a b} \\
\sqrt{a b} & b \end{bmatrix} \in \cP_2( I )$
shows that $f$ is \emph{multiplicatively mid-convex}: the function
\[
g( y ) := \log f( e^y ) \qquad ( y < \log \rho )
\]
is midpoint convex and locally bounded on the interval $\log I$. Now
the following classical result \cite[Theorem~71.C]{roberts-varberg}
shows that $g$ is continuous on $\log I$, so $f$ is continuous on $I$.

\begin{proposition}\label{Pconvex}
Let $U$ be a convex open set in a real normed linear space.  If
$g : U \to \R$ is midpoint convex on $U$ and bounded above in an open
neighborhood of a single point in $U$, then $g$ is continuous, so
convex, on $U$.
\end{proposition}

We now move to variants of the Horn--Loewner result. Notice that
Theorems~\ref{Thorn} and~\ref{Thorn2} are results for arbitrary
positivity preservers $f(x)$. When more is known about $f$, such as
smoothness or even real analyticity, stronger conclusions can be drawn
from smaller test sets of matrices. A recent variant is the following
lemma, shown by evaluating $f[-]$ at matrices $(t u_j u_k)_{j,k=1}^n$
and using the invertibility of ``generic'' generalized Vandermonde
matrices.

\begin{lemma}[Belton--Guillot--Khare--Putinar \cite{BGKP-fixeddim} and
Khare--Tao \cite{Khare-Tao}]\label{Lhorn}
Let $n \geq 1$ and $0 < \rho \leq \infty$.
Suppose $f( x ) = \sum_{k \geq 0} c_k x^k$ is a convergent power
series on $I = [ 0,\rho )$ that is positivity preserving entrywise on
rank-one matrices in $\cP_n(I)$. Further assume that $c_{m'} < 0$ for
some $m'$.
\begin{enumerate}
\item If $\rho < \infty$, then we have $c_m > 0$ for at least $n$ values
of $m < m'$. (In particular, the first $n$ non-zero Maclaurin
coefficients of $f$, if they exist, must be positive.)

\item If instead $\rho = \infty$, then we have $c_m > 0$ for at least $n$
values of $m < m'$ and at least $n$ values of $m > m'$. (In particular,
if $f$ is a polynomial, then the first $n$ non-zero coefficients and the
last $n$ non-zero coefficients of $f$, if they exist, are all positive.)
\end{enumerate}
\end{lemma}

Notice that this lemma (a)~talks about the derivatives of $f$ at $0$
and not in $(0,\rho)$; and moreover, (b)~considers not the first few
derivatives, but the first few \emph{non-zero} derivatives. Thus, it
is morally different from the preceding two theorems, and one
naturally seeks a common unification of these three results. This was
recently achieved:

\begin{theorem}[Khare~\cite{Khare}]\label{Tmaster}
Let $0 \leq a < \infty, \epsilon \in (0,\infty), I = [a, a+\epsilon)$,
and let $f : I \to \R$ be smooth. Fix integers $n \geq 1$ and $0 \leq p
\leq q \leq n$, with $p=0$ if $a=0$, and such that $f(x)$ has $q-p$
non-zero derivatives at $x=a$ of order at least $p$. Now let
\[
m_0 := 0, \quad \ldots \quad m_{p-1} := p-1;
\]
suppose further that
\[
p \leq m_p < m_{p+1} < \cdots < m_{q-1}
\]
are the lowest orders (above $p$) of the first $q-p$ non-zero
derivatives of $f(x)$ at $x=a$.

Also fix distinct scalars $u_1$, \ldots, $u_n \in ( 0, 1 )$,
and let $\bu := ( u_1, \ldots, u_n )^T$. If
$f[ a \bone_{n \times n} + t \bu \bu^T ] \in \cP_n( \R)$ for all
$t \in [ 0, \epsilon )$, then the derivative $f^{( k )}( a )$ is
non-negative whenever $0 \leq k \leq m_{q-1}$.
\end{theorem}

Notice that varying $p$ allows one to control the number of initial
derivatives versus~the number of subsequent non-zero derivatives of
smallest order. In particular, if $p = q = n$, then the result implies
the ``stronger'' Horn--Loewner theorem~\ref{Thorn2} (and so
Theorem~\ref{Thorn}) pointwise at every $a > 0$. At the other extreme
is the special case of $p = 0$ (at any $a \geq 0$), which strengthens
the conclusions of Theorems~\ref{Thorn} and~\ref{Thorn2} for smooth
functions.

\begin{corollary}\label{Cmaster}
Suppose $a$, $\epsilon$, $I$, $f$, $n$ and $\bu$ are as in
Theorem~\ref{Tmaster}. If
$f[ a \bone_{n \times n} + t \bu \bu^T ] \in \cP_n( \R )$ for all
$t \in [ 0, \epsilon )$, then the first $n$ non-zero derivatives of
$f( x )$ at $x=a$ are positive. 
\end{corollary}

\begin{remark}
Theorem~\ref{Tmaster} further clarifies the nature of the
Horn--Loewner result and its proof. The reduction from arbitrary
functions, to continuous functions, to smooth functions, requires an
open domain $( 0, \rho )$, in order to use mollifiers, for
example. However, the result for smooth functions actually holds
pointwise, as shown by Theorem~\ref{Tmaster}.
\end{remark}

The proof of Theorem~\ref{Tmaster} combines novel arguments together
with the previously mentioned techniques of Loewner. The refinement of
the determinant computations~(\ref{Eloewner}) is of particular note;
see Theorem~\ref{Phorn} and its consequence, Theorem~\ref{Tsymm}.

\subsection{Schoenberg redux: moment sequences and Hankel
matrices}\label{S33}

In this section, we outline another approach to proving Schoenberg's
theorem~\ref{Tschoenberg}, which yields a stronger version parallel to
the strengthening by Rudin of Theorem~\ref{Trudin}. The present
section reveals connections between positivity preservers, totally
non-negative Hankel matrices, moment sequences of positive measures on
the real line, and also a connection to semi-algebraic geometry.

We begin with Rudin's Theorem~\ref{Trudin} and the
family~(\ref{Erudin}). Notice that the positive definite sequences
in~(\ref{Erudin}) give rise to the Toeplitz matrices
$A(n, \alpha, \beta, \theta)$ with $(j,k)$ entry equal to
$\alpha + \beta \cos\bigl( ( j - k ) \theta \bigr)$. From the
elementary identity
\[
\cos( p - q ) = \cos p \cos q + \sin p \sin q %
\qquad ( p, q \in \R ),
\]
it follows that these Toeplitz matrices have rank at most three:
\begin{equation}\label{EAnabt}
A( n, \alpha, \beta, \theta ) = %
\alpha \bone_{n \times n} + \beta \bu \bu^T + \beta \bv \bv^T,
\end{equation}
where
\[
\bu := \bigl( %
\cos \theta, \cos( 2 \theta ), \ldots, \cos( n \theta) %
\bigr)^T \ \text{ and } \ %
\bv := \bigl( %
\sin \theta, \sin( 2 \theta ), \ldots, \sin( n \theta ) %
\bigr)^T.
\]
In particular, Rudin's work (see Theorem~\ref{Trudin} and the
subsequent discussion) implies the following result.

\begin{proposition}\label{Prudin}
Let $\theta \in \R$ such that $\theta / \pi$ is irrational.
An entrywise map $f : \R \to \R$ preserves positivity on the set of
Toeplitz matrices
\[
\{ A( n,\alpha,\beta,\theta) : n \geq 1, \ \alpha, \beta > 0 \}
\]
if and only if $f( x ) = \sum_{k = 0}^\infty c_k x^k$ is a convergent
power series on $\R$, with $c_k \geq 0$ for all $k \geq 0$.
\end{proposition}

Thus, one can significantly reduce the set of test matrices.

\begin{proof}
Given $0 < \rho < \infty$, let the restriction
$f_\rho := f|_{( -\rho, \rho )}$. Observe from the discussion
following Theorem~\ref{Trudin} that Rudin's work explicitly shows the
result for $f_1$, whence for any $f_\rho$ by a change of variables.
Thus,
\[
f_\rho( x ) = \sum_{k = 0}^\infty c_{k, \rho} x^k, \qquad %
c_{k, \rho} \geq 0 \text{ for all } k \geq 0 \text{ and } \rho > 0.
\]
Given $0 < \rho < \rho' < \infty$, it follows by the identity theorem
that $c_{k,\rho} = c_{k,\rho'}$ for all $k$. Hence
$f(x) = \sum_{k \geq 0} c_{k,1} x^k$ (which was Rudin's $f_1(x)$), now
on all of $\R$.
\end{proof}

In a parallel vein to Rudin's results and Proposition~\ref{Prudin},
the following strengthening of Schoenberg's result can be shown, using
a different (and perhaps more elementary) approach than those of
Schoenberg and Rudin.

\begin{theorem}[Belton--Guillot--Khare--Putinar
\cite{BGKP-hankel}]\label{Thankel}
Suppose $0 < \rho \leq \infty$ and $I = (-\rho,\rho)$. Then the following
are equivalent for a function $f : I \to \R$.
\begin{enumerate}
\item The entrywise map $f[-]$ preserves positivity on $\cP_n(I)$, for
all $n \geq 1$.

\item The entrywise map $f[-]$ preserves positivity on the Hankel
matrices in $\cP_n(I)$ of rank at most $3$, for all $n \geq 1$.

\item The function $f$ is real analytic on $I$ and absolutely monotonic
on $(0,\rho)$. In other words, $f(x) = \sum_{k \geq 0} c_k x^k$ on $I$,
with $c_k \geq 0\ \forall k$.
\end{enumerate}
\end{theorem}

\begin{remark}
Recall the alternate notion of positive definite functions discussed
in Remark~\ref{Rpinkus}. In~\cite{pinkus04} and related works, Pinkus
and other authors study this alternate notion of positive definite
functions on $H$. Notice that such matrices form precisely the set of
positive semidefinite symmetric matrices of rank at most $\dim H$. In
particular, Theorem~\ref{Thankel} and the far earlier 1959
paper~\cite{Rudin-Duke59} of Rudin both provide a characterization of
these functions, on every Hilbert space of dimension $3$ or more.
\end{remark}

Parallel to the discussions of the proofs of Schoenberg's and Rudin's
results (see the previous chapter), we now explain how to prove
Theorem~\ref{Thankel}. Clearly, $(3) \implies (1) \implies (2)$ in the
theorem. We first outline how to weaken the condition $(2)$ even
further and still imply $(3)$. The key idea is to consider
\emph{moment sequences} of certain non-negative measures on the real
line. This parallels Rudin's considerations of Fourier--Stieltjes
coefficients of non-negative measures on the circle.

\begin{definition}\label{Dadmissible}
A measure $\mu$ with support in $\R$ is said to be \emph{admissible}
if $\mu \geq 0$ on $\R$, and all moments of $\mu$ exist and are
finite:
\[
s_k( \mu ) := \int_\R x^k \std\mu( x ) < \infty \qquad ( k \geq 0 ).
\]
The sequence $\bs( \mu ) := \bigl( s_k( \mu ) \bigr)_{k=0}^\infty$ is
termed the \emph{moment sequence} of $\mu$. Corresponding to $\mu$ and
this moment sequence is the \emph{moment matrix} of $\mu$:
\[
H_\mu := \begin{bmatrix}
s_0( \mu ) & s_1( \mu ) & s_2( \mu ) & \cdots \\
s_1( \mu ) & s_2( \mu ) & s_3( \mu ) & \cdots \\
s_2( \mu ) & s_3( \mu ) & s_4( \mu ) & \cdots \\
\vdots & \vdots & \vdots & \ddots
\end{bmatrix};
\]
note that $H_\mu = [ s_{i + j}( \mu )] _{i,j \geq 0}$ is a
semi-infinite Hankel matrix. Finally, a function $f : \R \to \R$ acts
entrywise on moment sequences, to yield real sequences:
\[
f[ \bs( \mu ) ] := ( f\bigl( s_0( \mu ) \bigr), \ldots, %
f\bigl( s_k( \mu ) \bigr), \ldots ).
\]
\end{definition}

We are interested in understanding which entrywise functions preserve
the space of moment sequences of admissible measures. The connection
to positive semidefinite matrices is made through Hamburger's theorem,
which says that a real sequence $( s_0, s_1, \ldots )$ is the moment
sequence of an admissible measure on $\R$ if and only if every
(finite) principal minor of the moment matrix $H_\mu$ is positive
semidefinite. For simplicity, this last will be reformulated below to
saying that $H_\mu$ is positive semidefinite.

The weakening of Theorem~\ref{Thankel}(2) is now explained: it suffices
to consider the reduced test set of those Hankel matrices, which arise as
the moment matrices of admissible measures supported at three points.
Henceforth, let~$\delta_x$ denote the Dirac probability measure
supported at $x \in \R$. It is not hard to verify that the $m$-point
measure $\mu = \sum_{j = 1}^m c_j \delta_{x_j}$ has Hankel
matrix~$H_\mu$ with rank no more than~$m$:
\begin{align}\label{Emoments}
\begin{aligned}
s_k(\mu) & = \sum_{j=1}^m c_j x_j^k \qquad ( k \geq 0 ) \\
\implies \quad
H_\mu & = \sum_{j=1}^m c_j \bu_j \bu_j^T, \text{ where } %
\bu_j := (1, x_j, x_j^2, \ldots )^T.
\end{aligned}
\end{align}
Thus, a further strengthening of Schoenberg's result is as follows.

\begin{theorem}[Belton--Guillot--Khare--Putinar
\cite{BGKP-hankel}]\label{Thankel2}
In the setting of Theorem~\ref{Thankel}, the three assertions
contained therein are also equivalent to
\begin{enumerate}
\setcounter{enumi}{3}
\item For each measure
\begin{equation}\label{EsSchoenberg}
\mu = a \delta_1 + b \delta_{u_0} + c \delta_{-1},
\quad \text{with } u_0 \in ( 0, 1 ), \ a, b, c \geq 0, \ %
a + b + c \in ( 0, \rho ),
\end{equation}
there exists an admissible measure $\sigma_\mu$ on $\R$ such that
$f\bigl( s_k( \mu ) \bigr) = s_k( \sigma_\mu )$ for all $k \geq 0$.
\end{enumerate}
\end{theorem}

In fact, we will see in Section~\ref{Sinttrick} below that this
assertion~(4) can be simplified to just assert that $f[ H_\mu ]$ is
positive semidefinite, and so completely avoid the use of Hamburger's
theorem.

We now discuss the proof of these results, working with
$\rho = \infty$ for ease of exposition. The first observation is that
the strengthening of the Horn--Loewner theorem~\ref{Thorn2}, together
with the use of Bernstein's theorem (see remark~(2) following
Theorem~\ref{Thorn}), implies the following ``stronger'' form of
Vasudeva's theorem~\ref{Tvasudeva}:

\begin{theorem}[see \cite{BGKP-hankel}]\label{Tvasudeva2}
Suppose $I = (0,\infty)$ and $f : I \to \R$. Also fix $u_0 \in (0,1)$.
The following are equivalent:
\begin{enumerate}
\item The entrywise map $f[-]$ preserves positivity on $\cP_n(I)$ for all
$n \geq 1$.

\item The entrywise map $f[-]$ preserves positivity on all moment
matrices $H_\mu$ for $\mu = a \delta_1 + b \delta_{u_0}, \ a,b > 0$.

\item The function $f$ equals a convergent power series
$\sum_{k=0}^\infty c_k x^k$ for all $x \in I$,
with the Maclaurin coefficients $c_k \geq 0$ for all $k \geq 0$.
\end{enumerate}
\end{theorem}

Notice that the test matrices in assertion~(2) are all Hankel, and of
rank at most two. This severely weakens Vasudeva's original hypotheses.

Now suppose the assertion in Theorem~\ref{Thankel2}(4) holds. By the
preceding result, $f(x)$ is given on $(0,\infty)$ by an absolutely
monotonic function $\sum_{k \geq 0} c_k x^k$. The next step is to show
that $f$ is continuous. For this, we will crucially use the following
``integration trick''. Suppose for each admissible measure $\mu$ as
in~(\ref{EsSchoenberg}), there is a non-negative measure $\sigma_\mu$
supported on $[ -1, 1 ]$ such that
$f\bigl( s_k( \mu ) \bigr) = s_k( \sigma_\mu )$ for all $k \geq 0$.
(Note here that it is not immediate that the support
is contained in $[-1,1]$.)

Now let $p( t ) = \sum_{k \geq 0} b_k t^k$ be a polynomial that takes
non-negative values on $[ -1, 1 ]$. Then,
\begin{equation}\label{Einttrick}
0 \leq \int_{-1}^1 p( t ) \std\sigma_\mu( t ) = %
\sum_{k = 0}^\infty \int_{-1}^1 b_k t^k \std\sigma_\mu( t ) = %
\sum_{k = 0}^\infty b_k s_k( \sigma_\mu ) = %
\sum_{k = 0}^\infty b_k f\bigl( s_k( \mu ) \bigr).
\end{equation}

\begin{remark}
For example, suppose $p( t ) = 1 - t^d$ for some $d \geq 1$. If
$\mu = a \delta_1 + b \delta_{u_0} + c \delta_{-1}$, where
$u_0 \in ( 0, 1 )$ and $a$, $b$, $c > 0$, then the
inequality~(\ref{Einttrick}) gives that
\[
0 \leq f\bigl( s_0( \mu ) \bigr) - f\bigl( s_d( \mu ) \bigr) = %
f( a + b + c) - f( a + b u_0^d + c (-1)^d ).
\]
It is not clear \emph{a priori} how to deduce this inequality using
the fact that $f[-]$ preserves matrix positivity and the Hankel moment
matrix of $\mu$. The explanation, which we provide in
Section~\ref{Sinttrick} below, connects moment problems, matrix
positivity, and real algebraic geometry.
\end{remark}

We now outline how (\ref{Einttrick}) can be used to prove of the
continuity of $f$.  First note that $| s_k( \mu ) | \leq s_0( \mu )$
for $\mu$ as above and all $k \geq 0$. This fact and the easy
observation that $f$ is bounded on compact subsets of $\R$ together
imply that all moments of $\sigma_\mu$ are uniformly bounded. From
this we deduce that $\sigma_\mu$ is necessarily supported
on~$[ -1, 1 ]$.

The inequality~(\ref{Einttrick}) now gives the left-continuity of $f$
at $-\beta$, for every $\beta \geq 0$. Fix $u_0 \in ( 0, 1)$, and let
\[
\mu_b := ( \beta + b u_0 ) \delta_{-1} + b \delta_{u_0} %
\qquad ( b > 0).
\]
Applying~(\ref{Einttrick}) to the polynomials
$p_{\pm,1}( t ) := ( 1 \pm t )( 1 - t^2 )$, we deduce that
\[
f\bigl( \beta + b ( 1 + u_0 ) \bigr) - %
f\bigl( \beta + b ( u_0 + u_0^2 ) \bigr) \geq %
| f( -\beta ) - f\bigl( -\beta - b u_0 ( 1 - u_0^2 ) \bigr) |.
\]
Letting $b \to 0^+$, the left continuity of $f$ at $-\beta$ follows.
Similarly, to show that $f$ is right continuous at $-\beta$, we apply
the integral trick to $p_{\pm,1}( t )$ and to
$\mu'_b := ( \beta + b u_0^3 ) \delta_{-1} + b \delta_{u_0}$ instead
of $\mu_b$.

Having shown continuity, to prove the stronger Schoenberg theorem, we
next assume that $f$ is smooth on $\R$. For all $a \in \R$, define the
function
\[
H_a : \R \to \R; \ x \mapsto f( a + e^x ).
\]
The function $H_a$ satisfies the estimates
\begin{equation}\label{Eestimates}
| H_a^{( n )}( x ) | \leq H_{|a|}^{(n)}( x ) \qquad %
( a, x \in \R, \ n \in \Z_+ ).
\end{equation}
This is shown by another use of the integration
trick~(\ref{Einttrick}), this time for the polynomials
$p_{\pm,n}( t ) := ( 1 \pm t ) ( 1 - t^2 )^n$ for all $n \geq 0$. In
turn, the estimates~(\ref{Eestimates}) lead to showing that $H_a$ is
real analytic on $\R$, for all $a \in \R$. Now composing $H_{-a}$ for
$a > | x |$ with the function $L_a( y ) := \log( a + y )$ shows that
$f( x )$ is real analytic on $\R$ and agrees with
$\sum_{k \geq 0} a_k x^k$ on $( 0, \infty )$. This concludes the proof
for smooth functions.

Finally, to pass from smooth functions to continuous functions, we again
use a mollified family $f_\delta \to f$ as $\delta \to 0^+$. Each
$f_\delta$ is the restriction of an entire function, say
$\widetilde{f}_\delta$, and the family
$\{ \widetilde{f}_{1/n} : n \geq 1 \}$ forms a normal family on each
open disc $D( 0, r )$. It follows from results by Montel and Morera
that $\widetilde{f}_{1/n}(z)$ converges uniformly to a function
$g_r$ on each closed disc $\overline{D( 0, r )}$, and $g_r$ is
analytic. Since $g_r$ restricts to $f$ on $( -r, r )$, it follows that
$f$ is necessarily also real analytic on $\R$, and we are done.

\subsection{The integration trick, and positivity
certificates}\label{Sinttrick}

Observe that the inequality~(\ref{Einttrick}) can be
written more generally as follows.

\emph{Given a polynomial
$p( t ) = \sum_{k \geq 0} b_k t^k$ which takes non-negative values on
$[ -1, 1 ]$, as well as a positive semidefinite Hankel matrix
$H = ( s_{i + j} )_{i,j \geq 0}$, we have that
\begin{equation}\label{Etrick}
\sum_{k \geq 0} b_k s_k \geq 0.
\end{equation}}

As shown in~(\ref{Einttrick}), this assertion is clear via an
application of Hamburger's theorem. We now demonstrate how the
assertion can instead be derived from first principles, with
interesting connections to positivity certificates.

First note that the inequality~(\ref{Etrick}) holds if
$p( t )$ is the square of a polynomial. For instance, if
$p( t ) = (1-3t)^2 = 1 - 6 t + 9 t^2$ on $[ -1, 1 ]$, then
\begin{equation}\label{Einttrick2}
s_0 - 6 s_1 + 9s _2 = ( e_0 - 3 e_1 )^T H ( e_0 - 3 e_1 ),
\end{equation}
where $e_0 = ( 1, 0, 0, \ldots )$ and $e_1 = ( 0, 1, 0, 0, \ldots )$.
The non-negativity of~(\ref{Einttrick2}) now follows immediately from
the positivity of the matrix~$H$. The same reasoning applies if
$p( t )$ is a sum of squares of polynomials, or even the limit of a
sequence of sums of squares. Thus, one approach to showing the
inequality~(\ref{Etrick}) for an arbitrary polynomial $p( t )$ which is
non-negative on $[ -1, 1 ]$ is to seek a \emph{limiting sum-of-squares
representation}, which is also known as a \emph{positivity
certificate}, for~$p$.

If a $d$-variate real polynomial is a sum of squares of real polynomials,
then it is clearly non-negative on $\R^d$, but the converse is not true
for $d > 1$.%
\footnote{This is connected to semi-algebraic geometry and to
Hilbert's seventeenth problem: recall the famous result of Motzkin
that there are non-negative polynomials on $\R^d$ that are not sums of
squares, such as $x^4 y^2 + x^2 y^4 - 3 x^2 y^2 + 1$.
Such phenomena have been studied in several settings, including
polytopes (by Farkas, Handelman, and P{\'o}lya) and more general
semi-algebraic sets (by Putinar, Schm{\"u}dgen, Stengel, Vasilescu,
and others).}
Even when $d = 1$, while a sum-of-squares representation is an
equivalent characterization for one-variable polynomials that are
non-negative on $\R$, here we are working on the compact
semi-algebraic set $[ -1, 1 ]$. We now give three proofs of the
existence of such a positivity certificate in the setting used above.

\begin{proof}[Proof 1]
A result of Berg, Christensen, and Ressel (see the end of~\cite{BCR})
shows more generally that, for every dimension $d \geq 1$, any
non-negative polynomial on $[ -1, 1 ]^d$ has a limiting
sum-of-squares representation.
\end{proof}

\begin{proof}[Proof 2]
The only polynomials used in proving the stronger form of Schoenberg's
theorem, Theorems~\ref{Thankel} and~\ref{Thankel2}, appear
following~(\ref{Eestimates}):
\[
p_{\pm, n}(t) := ( 1 \pm t ) ( 1 - t^2 )^n \qquad ( n \geq 0 ).
\]
Each of these polynomials is composed of factors of the form
$p_{\pm, 0}( t ) = 1 \pm t$, so it suffices to produce a limiting
sum-of-squares representation for these two polynomials on
$[ -1, 1 ]$. Note that
\begin{align*}
\frac{1}{2}(1 \pm t)^2 = &\ \frac{1}{2} \pm t + \frac{t^2}{2},\\
\frac{1}{4}(1 - t^2)^2 = &\ \frac{1}{4} - \frac{t^2}{2} +
\frac{t^4}{4},\\
\frac{1}{8}(1 - t^4)^2 = &\ \frac{1}{8} - \frac{t^4}{4} +
\frac{t^8}{8},
\end{align*}
and so on. Adding the first $n$ equations shows that
$( 1 \pm t ) + 2^{-n}( t^{2^n} - 1 )$ is a sum-of-squares polynomial
for all $n$. Taking $n \to \infty$ finishes the proof.
\end{proof}

\begin{proof}[Proof 3]
In fact, for any $d \geq 1$ and any compact set $K \subset \R^d$, if
$f$ is a non-negative continuous function on $K$, then $f$ has a
positivity certificate. The Stone--Weierstrass theorem gives a
sequence of polynomials which converges to~$\sqrt{f}$, and the squares
of these polynomials then provide the desired limiting representation
for $f$. This is a simpler proof than Proof 1 from~\cite{BCR}, but the
convergence here is uniform, whereas the convergence in~\cite{BCR} is
stronger.
\end{proof}

\begin{remark}
In (\ref{Einttrick}), we used $H = H_{\sigma_\mu}$, which was positive
semidefinite by assumption. The previous discussion shows that
Theorem~\ref{Thankel2}(4) can be further weakened, by requiring only that
$f[ H_\mu ]$ is positive semidefinite, as opposed to being equal to
$H_{\sigma}$ for some admissible measure $\sigma$. Hence we do not
require Hamburger's theorem in order to prove the strengthening of
Schoenberg's theorem that uses the test set of low-rank Hankel matrices.
\end{remark}


\subsection{Variants of moment-sequence transforms}

We now present a trio of results on functions which
preserve moment sequences.

For $K \subset \R$, let $\moment( K )$ denote the set of moment
sequences corresponding to admissible measures with support in $K$.
We say that $F$ maps $\moment( K )$ into $\moment( L )$, where $K$,
$L \subset \R$, if for every admissible measure $\mu$ with support in $K$
there exists an admissible measure $\sigma$ with support in $L$ such that
\[
F( s_k( \mu ) ) = s_k( \sigma ) \quad \text{for all } k \in \Z_+,
\]
where $s_k( \mu )$ is the $k$th-power moment of $\mu$, as in
Definition~\ref{Dadmissible}.

\begin{theorem}\label{T2sided}
A function $F : \R \to \R$ maps $\moment([-1,1])$ into itself if and
only if $F$ is the restriction to~$\R$ of an absolutely monotonic
entire function.
\end{theorem}

\begin{theorem}\label{T1sided}
A function $F : \R_+ \to \R$ maps $\moment( [ 0, 1 ] )$ into itself if
and only if $F$ is absolutely monotonic on $( 0, \infty )$ and
$0 \leq F( 0 ) \leq \lim_{\epsilon \to 0^+} F( \epsilon )$.
\end{theorem}

\begin{theorem}\label{Tminus}
A function $F : \R \to \R$ maps $\moment( [ -1, 0 ] )$ into
$\moment( ( -\infty, 0 ] )$ if and only if there exists an absolutely
monotonic entire function $\widetilde{F} : \C \to \C$ such that
\[
F( x ) = \left\{\begin{array}{ll}
\widetilde{F}( x ) & \text{if } x \in ( 0, \infty ), \\
0 & \text{if } x = 0, \\
-\widetilde{F}( -x ) \qquad & \text{if } x \in ( -\infty, 0 ).
\end{array}\right.
\]
\end{theorem}

It is striking to observe the possibility of a discontinuity at the
origin which may occur in the latter two of these three theorems.

We will content ourselves here with sketching the proof of the second
result. For the others, see \cite{BGKP-hankel}, noting that the first of
the results follows from Theorems~\ref{Thankel} and~\ref{Thankel2} for
$\rho = \infty$.

\begin{proof}[Proof of Theorem~\ref{T1sided}]

Note that the moment matrix corresponding to an element of
$\moment( [ 0, 1 ] )$ has a zero entry if and only if
$\mu = a \delta_0$ for some $a \geq 0$. This and the Schur product
theorem give one implication.

For the converse, suppose $F$ preserves $\moment( [ 0, 1 ] )$. Fix
finitely many scalars $c_j$, $t_j > 0$ and an integer $n \geq 0$, and
set
\begin{equation}\label{Echoice}
p( t ) = ( 1 - t )^n \quad \text{and} \quad
\mu = \sum_j e^{- t_j \alpha} c_j \delta_{e^{-t_j h}},
\end{equation}
where $\alpha > 0$ and $h>0$. If $g( x ) := \sum_j c_j e^{-t_j x}$
then the integration trick (\ref{Einttrick}), but working on
$[ 0, 1 ]$, shows that the forward finite differences of $F \circ g$
alternate in sign:
\[
\sum_{k=0}^n (-1)^k \binom{n}{k} F\Bigl(
\sum_j c_j e^{- t_j ( \alpha + k h )}\Bigr) \geq 0,
\]
so $(-1)^n \Delta^n_h(F \circ g)(\alpha) \geq 0$. As this holds
for all $\alpha$, $h > 0$ and all $n \geq 0$, it follows that
$F \circ g : (0,\infty) \to (0,\infty)$ is completely monotonic. The
weak density of measures of the form $\mu$, together with  Bernstein's
theorem (\ref{Ebernstein}), gives that $F \circ g$ is completely
monotonic on $(0,\infty)$ for every completely monotonic function
$g : (0,\infty) \to (0,\infty)$. Finally, a theorem of Lorch and
Newman \cite[Theorem 5]{Lorch-Newman} now gives that
$F : (0,\infty) \to ( 0, \infty )$ is absolutely monotonic.
\end{proof}

\subsection{Multivariable positivity preservers and moment families}

We now turn to the multivariable case, and begin with two results of
FitzGerald, Micchelli, and Pinkus \cite{fitzgerald}. We first introduce
some notation and a piece of terminology.

Fix $I \subset \C$ and an integer $m \geq 1$, and let
\[
A^k = ( a^k_{i j} )_{i, j = 1}^N \in I^{N \times N} \quad %
\text{for } k = 1, \ldots, m.
\]
For any function $f : I^m \to \C$, we have the $N \times N$ matrix
\[
f( A^1, \ldots, A^m ) := %
\bigl( f( a^1_{i j}, \ldots, a^m_{i j} ) \bigr)_{i, j = 1}^N \in %
\C^{N \times N}.
\]
We say that $f : \R^m \to \R$ is \emph{real positivity preserving} if
\[
f( A^1, \ldots, A^m ) \in \cP_N( \R ) %
\text{ for all } A^1, \ldots, A^m \in \cP_N( \R ) %
\text{ and all } N \geq 1,
\]
where, as above $\cP_N( \R )$ is the collection of $N \times N$ positive
semidefinite matrices with real entries. Similarly, we say that
$f : \C^m \to \C$ is positivity preserving if
\[
f( A^1, \ldots, A^m ) \in \cP_N %
\qquad \text{for all } A^1, \ldots, A^m \in \cP_N %
\text{ and all } N \geq 1,
\]
where $\cP_N$ is the collection of $N \times N$ positive semidefinite
matrices with complex entries. Finally, recall that a function
$f : \R^m \to \R$ is said to be \emph{real entire} if there exists an
entire function $F : \C^m \to \C$ such that $F |_{\R^m} = f$. We will
also use the multi-index notation
\[
\bx^\alpha := x_1^{\alpha_1} \cdots x_m^{\alpha_m} \quad \text{if }
\bx = ( x_1, \ldots, x_m ) \text{ and } %
\alpha = ( \alpha_1, \ldots, \alpha_m ).
\]

The following theorems are natural extensions of Schoenberg's theorem
and Herz's theorem, respectively.

\begin{theorem}[{\cite[Theorem~2.1]{fitzgerald}}]\label{TrealFMP}
Let $f : \R^m \to \R$, where $m \geq 1$. Then $f$ is real
positivity preserving if and only $f$ is real entire of the form
\[
f( \bx ) = \sum_{\alpha \in \Z_+^m} c_{\alpha} \bx^\alpha \qquad %
( \bx \in \R^m ),
\]
where $c_{\alpha} \geq 0$ for all $\alpha \in \Z_+^m$.
\end{theorem}

\begin{theorem}[{\cite[Theorem~3.1]{fitzgerald}}]
Let $f : \C^m \to \C$, where $m \geq 1$. Then $f$ is positivity
preserving if and only $f$ is of the form
\[
f( \bz ) = \sum_{\alpha, \beta \in \Z_+^m} c_{\alpha \beta} %
\bz^\alpha \overline{\bz}^\beta \qquad ( \bz \in \C^m ),
\]
where $c_{\alpha \beta} \geq 0$ for all $\alpha$, $\beta \in \Z_+^m$
and the power series converges absolutely for all $\bz \in \C$.
\end{theorem}

We now consider the notion of moment family for measures on $\R^d$. As
above, a measure on $\R^d$ is said to be \emph{admissible} if it is
non-negative and has moments of all orders. Given such a measure
$\mu$, we define the \emph{moment family}
\[
s_\alpha( \mu ) := \int \bx^\alpha \std\mu( \bx ) \qquad %
\text{ for all } \alpha \in \Z_+^m.
\]
In line with the above, we let $\moment(K)$ denote the set of all
moment families of admissible measures supported on $K \subset \R^d$.

Note that a measure $\mu$ is supported in  $[ -1, 1 ]^d$ if and only
if its moment family is uniformly bounded:
\[
\sup\bigl\{  | s_\alpha( \mu ) | : \alpha \in \Z_+^m \bigr\} < \infty.
\]

\begin{theorem}[{\cite[Theorem~8.1]{BGKP-hankel}}]\label{TEuclidean}
A function $F : \R \to \R$ maps $\moment\bigl( [ -1, 1 ]^d \bigr)$ to
itself if and only if~$F$ is absolutely monotonic and entire.
\end{theorem}
\begin{proof}
Since $[ -1, 1 ]$ can be identified with
$[ -1, 1 ] \times \{ 0 \}\ ^{d-1} \subset [ -1, 1 ]^d$, the forward
implication follows from the one-dimensional result,
Theorem~\ref{T2sided}.

For the converse, we use the fact \cite{Putinar} that a collection of
real numbers $(s_\alpha)_{\alpha \in \Z_+^d}$ is an element of
$\moment\bigl( [ -1, 1 ]^d \bigr)$ if and only if the weighted
Hankel-type kernels on $\Z_+^d \times \Z_+^d$
\[
( \alpha, \beta ) \mapsto s_{\alpha + \beta} \quad \text{and} \quad %
( \alpha, \beta ) \mapsto %
s_{\alpha + \beta} - s_{\alpha + \beta + 2 \bone_j} \quad %
(1 \leq j \leq d)
\]
are positive semidefinite, where
\[
\bone_j := ( 0, \ldots, 0, 1 ,0, \ldots, 0 ) \in \Z_+^d
\]
with $1$ in the $j$th position. Now suppose $F$ is absolutely
monotonic and entire; given a family
$( s_\alpha )_{\alpha \in \Z_+^d}$ subject to these positivity
constraints, we have to verify that the family
$( F(s_\alpha) )_{\alpha \in \Z_+^d}$ satisfies them as well.

Theorem~\ref{Thankel} gives that
$( \alpha, \beta ) \mapsto F( s_{\alpha + \beta} )$
and $( \alpha, \beta ) \mapsto F( s_{\alpha + \beta + 2\bone_j } )$
are positive semidefinite, so we must show that
\[
( \alpha, \beta ) \mapsto %
F( s_{\alpha + \beta} ) - F( s_{\alpha + \beta + 2\bone_j} )
\]
is positive semidefinite for $j = 1$, \ldots, $d$. As $F$ is
absolutely monotonic and entire, it suffices to show that
\[
( \alpha, \beta ) \mapsto ( s_{\alpha + \beta} )^{\circ n} - %
( s_{\alpha + \beta + 2\bone_j} )^{\circ n}
\]
is positive semidefinite for any $n \geq 0$, but this follows from the
Schur product theorem: if $A \geq B \geq 0$, then
\[
A^{\circ n} \geq A^{\circ (n - 1)} \circ B \geq A^{\circ (n-2)} \circ %
B^{\circ 2} \geq \cdots \geq B^{\circ n}.
\qedhere
\]
\end{proof}

We next consider characterizations of real-valued multivariable
functions which map tuples of moment sequences to moment sequences.

Let $K_1$, \ldots, $K_m \subset \R$. A function $F : \R^m \to \R$ acts
on tuples of moment sequences of (admissible) measures
$\moment( K_1 ) \times \cdots \times \moment( K_m )$ as follows:
\begin{equation}
F[ \bs( \mu_1) , \ldots, \bs( \mu_m ) ]_k := %
F\bigl( s_k( \mu_1 ), \ldots, s_k( \mu_m ) \bigr) %
\quad \text{for all } k \geq 0.
\end{equation}

Given $I \subset \R^m$, a function $F : I \to \R$ is
\emph{absolutely monotonic} if $F$ is continuous on~$I$, and for all
interior points $\bx \in I$ and $\alpha \in \Z_+^m$, the mixed partial
derivative $D^\alpha F( \bx )$ exists and is non-negative, where
\[
D^\alpha F( \bx ) := \frac{\partial^{| \alpha |}}%
{\partial x_1^{\alpha_1} \cdots \partial x_m^{\alpha_m}} %
F( x_1, \ldots, x_m ) \quad %
\text{and } | \alpha | := \alpha_1 + \cdots + \alpha_m.
\]
With this definition, the multivariable analogue of Bernstein's
theorem is as one would expect; see \cite[Theorem~4.2.2]{Bochner-book}.

To proceed further, it is necessary to introduce the notion of a
\emph{facewise absolutely monotonic function} on~$\R_+^m$. Observe that
the orthant~$\R_+^m$ is a convex polyhedron, and is therefore the
disjoint union of the relative interiors of its faces. These faces are in
one-to-one correspondence with subsets of
$[m] := \{ 1, \ldots, m \}$:
\begin{equation}\label{Epolyhedral}
J \mapsto \R_+^J := %
\{ ( x_1, \ldots, x_m ) \in \R_+^m : %
x_i = 0 \text{ if } i \not\in J \};
\end{equation}
note that this face has relative interior
$\R_{>0}^J := ( 0, \infty )^J \times \{ 0 \}^{[m] \setminus J}$.

\begin{definition}
A function $F : \R_+^m \to \R$ is \emph{facewise absolutely monotonic}
if, for every $J \subset [m]$, there exists an
absolutely monotonic function $g_J$ on~$\R_+^J$ which agrees
with $F$ on~$\R_{>0}^J$.
\end{definition}

Thus a facewise absolutely monotonic function is piecewise absolutely
monotonic, with the pieces being the relative interiors of the faces
of the orthant $\R_+^m$. See \cite[Example~8.4]{BGKP-hankel} for
further discussion. In the special case $m=1$, this broader class of
functions (than absolutely monotonic functions on $\R_+$) coincides
precisely with the maps which are absolutely monotonic on $(0,\infty)$
and have a possible discontinuity at the origin, as in
Theorem~\ref{T1sided} above.

This definition allows us to characterize the preservers of $m$-tuples
of elements of~$\moment\bigl( [ 0, 1 ] \bigr)$; the preceding observation
shows that Theorem~\ref{T1sided} is precisely the $m=1$ case.

\begin{theorem}[{\cite[Theorem~8.5]{BGKP-hankel}}]\label{T1sided-multi}
Let $F : \R_+^m \to \R$, where the integer $m \geq 1$. The
following are equivalent.
\begin{enumerate}
\item $F$ maps $\moment( [ 0, 1 ] )^m$ into $\moment( [ 0, 1 ] )$.

\item $F$ is facewise absolutely monotonic, and the functions
$\{ g_J : J \subset [m] \}$ are such that
$0 \leq g_J \leq g_K$ on $\R_+^J$ whenever $J \subset K \subset [m]$.

\item $F$ is such that
\[
F\bigl( \sqrt{x_1 y_1}, \ldots, \sqrt{x_m y_m} \bigr)^2 \leq
F( x_1, \ldots, x_m ) F( y_1, \ldots, y_m )
\]
for all $\bx$, $\by \in \R_+^m$ and there exists some
$\bz \in ( 0, 1 )^m$ such that the products
$\bz^\alpha := z_1^{\alpha_1} \cdots z_m^{\alpha_m}$ are distinct for
all $\alpha \in \Z_+^m$ and $F$ maps
$\moment\bigl( \{ 1, z_1 \} \bigr) \times \cdots \times %
\moment\bigl( \{ 1, z_m \} \bigr) \cup \moment(\{ 0, 1 \})^m$
to $\moment( \R )$.
\end{enumerate}
\end{theorem}

The heart of Theorem~\ref{T1sided-multi} can be deduced from the
following result on positivity preservation on tuples of low-rank Hankel
matrices. In a sense, it is the multi-dimensional generalization of the
`stronger Vasudeva theorem'~\ref{Tvasudeva2}.

Fix $\rho \in ( 0, \infty ]$, an integer $m \geq 1$ and a point
$\bz \in ( 0, 1 )^m$ with distinct products, as in
Theorem~\ref{T1sided-multi}(3). For all $N \geq 1$, let
\[
\cH_N := \{ a \bone_{N \times N} + b \bu_{l, N} \bu_{l, N}^T : %
a \in ( 0, \rho ), \ b \in [ 0, \rho - a ), \ 1 \leq l \leq m \},
\]
where $\bu_{l, N} := ( 1, z_l, \ldots, z_l^{N - 1} )^T$.

\begin{theorem}[{\cite[Theorem~8.6]{BGKP-hankel}}]\label{Thorn-hankel2}
If $F : ( 0, \rho )^m \to \R$ preserves positivity
on~$\cP_2\bigl( ( 0, \rho ) \bigr)^m$ and $\cH_N^m$ for all~$N
\geq 1$, then $F$ is absolutely monotonic and is the restriction of an
analytic function on the polydisc $D( 0, \rho )^m$.
\end{theorem}

The notion of facewise absolute monotonicity emerges from the study of
positivity preservers of tuples of moment sequences. If one focuses
instead on maps preserving positivity of tuples of all positive
semidefinite matrices, or even all Hankel matrices, then this richer
class of maps does not appear.

\begin{proposition}\label{P1sided-multi}
Suppose $\rho \in ( 0, \infty ]$ and $F : [ 0, \rho )^m \to \R$. The
following are equivalent.
\begin{enumerate}
\item $F[-]$ preserves positivity on the space of $m$-tuples of
Hankel matrices with entries in $[ 0, \rho )$.
\item $F$ is absolutely monotonic on $[ 0, \rho )^m$.
\item $F[-]$ preserves positivity on the space of $m$-tuples of all
matrices with entries in $[0,\rho)$.
\end{enumerate}
\end{proposition}

\begin{proof}
Clearly $(2) \implies (3) \implies (1)$, so suppose (1) holds. It
follows from Theorem~\ref{Thorn-hankel2} that $F$ is absolutely
monotonic on the domain~$(0,\rho)^m$ and agrees there with an
analytic function $g : D(0,\rho)^m \to \C$. To see that $F \equiv g$
on $[ 0, \rho )^m$, we use induction on $m$, with the $m = 1$ case
being left as an exercise (see
\cite[Proof of Proposition~7.3]{BGKP-hankel}).

Now suppose $m > 1$, let
$\bc = ( c_1, \ldots, c_m ) \in %
[ 0, \rho )^m \setminus ( 0, \rho )^m$ and define
\[
H := \begin{bmatrix}
1 & 0 & 1\\
0 & 1 & 1\\
1 & 1 & 2
\end{bmatrix} \qquad \text{and} \qquad
A_i := \begin{cases}
\bone_{3 \times 3} & \text{if } c_i > 0,\\
H & \text{if } c_i = 0.
\end{cases}
\]
Choosing $\bu_n = ( u_{1,n}, \ldots, u_{m,n} ) \in ( 0, \rho )^m$ such
that $\bu_n \to \bc$, it follows that
\[
\lim_{n \to \infty} F[ u_{1, n} A_1, \ldots, u_{m, n} A_m ] = %
\begin{bmatrix}
g(\bc) & F(\bc) & g(\bc)\\
F(\bc) & g(\bc) & g(\bc)\\
g(\bc) & g(\bc) & g(\bc)
\end{bmatrix}
\in \cP_3,
\]
where the $( 1, 2 )$ and $( 2, 1 )$ entries are as claimed by the
induction hypothesis. The determinants of the first and last principal
minors now give that
\[
g(\bc) \geq 0 \qquad \text{and} \qquad %
-g( \bc ) \bigl( g( \bc ) - F( \bc ) \bigr)^2 \geq 0,
\]
whence $F( \bc ) = g( \bc )$.
\end{proof}

Having considered functions defined on the positive orthant, we now
look at the situation for functions defined over the whole of $\R^m$.

\begin{theorem}[{\cite[Theorem~8.9]{BGKP-hankel}}]\label{T2sided-multi}
Suppose $F : \R^m \to \R$ for some integer $m \geq 1$. The following
are equivalent.
\begin{enumerate}
\item $F$ maps $\moment\bigl( [ -1, 1 ] \bigr)^m$ into
$\moment( \R )$.

\item The function $F$ is real positivity preserving.

\item The function $F$ is absolutely monotonic on $\R_+^m$ and agrees
with an entire function on~$\R^m$.
\end{enumerate}
\end{theorem}

As before, the proof reveals that verifying positivity preservation
for tuples of low-rank Hankel matrices suffices. The following notation
and corollary make this precise.

For all $u \in ( 0, \infty )$, let
$\moment_u := \moment\bigl( \{ -1, u, 1 \} \bigr)$ and
\[
\moment_{[ u ]} := %
\bigcup\bigl\{ \moment\bigl( \{ s_1, s_2 \} \bigr) : %
s_1 \in \{ -1, 0, 1 \}, \ s_2 \in \{ -u, 0, u \} \bigr\}.
\]

\begin{corollary}[{\cite[Theorem~8.10]{BGKP-hankel}}]\label{C2sided-multi}
The hypotheses in Theorem \ref{T2sided-multi} are also equivalent to
the following.
\begin{enumerate}
\setcounter{enumi}{3}
\item There exist $u_0 \in ( 0, 1 )$ and $\epsilon > 0$ such that $F$
maps 
\[
\moment_{[u_0]}^m \cup \bigcup \bigl\{ %
\moment_{v_1} \times \cdots \times \moment_{v_m} : %
v_1, \ldots, v_m \in ( 0, 1 + \epsilon ) \bigr\}
\]
into $\moment( \R )$.
\end{enumerate}
\end{corollary}


\section{Entrywise polynomials preserving positivity in fixed
dimension}\label{Spolynomial}

Having discussed at length the dimension-free setting, we now turn our
attention to functions that preserve positivity in a fixed dimension
$N \geq 2$. This is a natural question from the standpoint of both
theory as well as applications. This latter connection to applied
fields and to high-dimensional covariance estimation will be explained
below in Chapter~\ref{Sstats}.

Mathematically, understanding the functions $f$ such that
$f[-] : \cP_N \to \cP_N$ for fixed $N \geq 2$, is a non-trivial and
challenging refinement of Schoenberg's 1942 theorem. A complete
characterization was found for $N=2$ by Vasudeva~\cite{vasudeva79}:

\begin{theorem}[Vasudeva \cite{vasudeva79}]\label{TK2}
Given a function $f : ( 0, \infty ) \to \R$, the entrywise map $f[-]$
preserves positivity on $\cP_2\bigl( ( 0, \infty ) \bigr)$ if and only
$f$ is non-negative, non-decreasing, and multiplicatively mid-convex:
\[
f( x ) f( y ) \geq f\bigl( \sqrt{x y} \bigr)^2 %
\qquad \text{for all } x, y > 0.
\]
In particular, $f$ is either identically zero or never zero on
$( 0, \infty )$, and $f$ is also continuous.
\end{theorem}

On the other hand, if $N \geq 3$, then such a characterization remains
open to date. As mentioned above, perhaps the only known result for
general entrywise preservers is the Horn--Loewner theorem~\ref{Thorn} (or
its more general variants such as Theorem~\ref{Tmaster}).

In light of this challenging scarcity of results in fixed dimension, a
strategy adopted in the literature has been to further refine the
problem, in one of several ways:
\begin{enumerate}
\item Restrict the class of functions, while operating entrywise on all
of $\cP_N$ (over some given domain $I$, say $(0,\rho)$ or $(-\rho,\rho)$
for $0 < \rho \leq \infty$). For example, in this survey we consider
possibly non-integer power functions, polynomials and power series, and
even linear combinations of real powers.

\item Restrict the class of matrices and study entrywise functions over
this class in a fixed dimension. For instance, popular sub-classes of
matrices include positive matrices with rank bounded above, or with a
given sparsity pattern (zero entries), or classes such as Hankel or
Toeplitz matrices; or intersections of these classes. For instance, in
discussing the Horn--Loewner and Schoenberg--Rudin results, we
encountered Toeplitz and Hankel matrices of low rank.

\item Study the problem under both of the above restrictions.
\end{enumerate}

In this chapter we begin with the first of these restrictions.
Specifically, we will study polynomial maps that preserve positivity,
when applied entrywise to $\cP_N$. Recall from the Schur product
theorem that if the polynomial $f$ has only non-negative coefficients
then $f[ - ]$ preserves positivity on $\cP_N$ for every dimension
$N \geq 1$. It is natural to expect that if one reduces the test set,
from all dimensions to a fixed dimension, then the class of polynomial
preservers should be larger. Remarkably, until 2016 not a single
example was known of a polynomial positivity preserver with a negative
coefficient. Then, in quick succession, the two papers
\cite{BGKP-fixeddim, Khare-Tao} provided a complete
understanding of the sign patterns of entrywise polynomial preservers
of $\cP_N$. The goal of this chapter is to discuss some of the results
in these works.

\subsection{Characterizations of sign patterns}

Until further notice, we work with entrywise polynomial
or power-series maps of the form
\begin{equation}\label{Epoly}
f( x ) = c_0 x^{n_0} + c_1 x^{n_1} + \cdots,
\quad \text{with } 0 \leq n_0 < n_1 < \cdots,
\end{equation}
and $c_j \in \R$ typically non-zero, which preserve $\cP_N( I )$ for
various $I$. Our goal is to try and understand their sign patterns,
that is, which $c_j$ can be negative. The first observation is that as
soon as $I$ contains the interval $( 0, \rho )$ for any $\rho > 0$, by
the Horn--Loewner type necessary conditions in Lemma~\ref{Lhorn}, the
lowest $N$ non-zero coefficients of $f( x )$ must be positive.

The next observation is that if $I \not\subset \R_+$, then, in
general, there is no structured classification of the sign patterns of
the power series preservers on $\cP_N( I )$. For example, let $k$ be a
non-negative integer; the polynomials
\[
f_{k,t}( x ) := t ( 1 + x^2 + \cdots + x^{2 k} ) - x^{2 k + 1} %
\qquad ( t > 0 )
\]
do not preserve positivity entrywise on
$\cP_N\bigl( (-\rho,\rho) \bigr)$ for any $N \geq 2$. This may be seen
by taking  $\bu := ( 1, -1, 0, \ldots, 0 )^T$ and
$A := \eta \bu \bu^T$ for some $0 < \eta < \rho$, and noting that
\[
\bu^T f_{k, t}[ A ] \bu = -4 \eta^{2 k + 1} < 0.
\]

Similarly, if one allows complex entries and uses higher-order roots
of unity, such negative results (vis-a-vis Lemma~\ref{Lhorn}) are
obtained for complex matrices.

Given this, in the rest of the chapter we will focus on $I = (0,\rho)$
for $0 < \rho \leq \infty$.\footnote{That said, we also briefly discuss
the one situation in which our results do apply more generally, even to
$I = D(0,\rho) \subset \C$ (an open complex disc).}
As mentioned above, if $f$ as in~(\ref{Epoly}) entrywise preserves
positivity even on rank-one matrices in
$\cP_N\bigl( ( 0, \rho ) \bigr)$ then its first $N$ non-zero Maclaurin
coefficients are positive. Our goal is to understand if any other
coefficient can be negative (and if so, which of them). This has at
least two ramifications:
\begin{enumerate}
\item It would yield the first example of a polynomial entrywise map (for
a fixed dimension) with at least one negative Maclaurin coefficient.
Recall the contrast to Schoenberg's theorem in the dimension-free
setting.

\item This also yields the first example of a polynomial (or power
series) that entrywise preserves positivity on $\cP_N(I)$ but not
$\cP_{N+1}(I)$. In particular it would imply that the Horn--Loewner type
necessary condition in Lemma~\ref{Lhorn}(1) is ``sharp''.
\end{enumerate}

These goals are indeed achieved in the particular case $n_0= 0$,
\ldots, $n_{N-1} = N-1$ in~\cite{BGKP-fixeddim}, and subsequently, for
arbitrary $n_0 < \cdots < n_{N-1}$ in~\cite{Khare-Tao}. (In fact, in
the latter work the $n_j$ need not even be integers; this is discussed
below.)  Here is a `first' result along these lines. Henceforth we
assume that $\rho < \infty$; we will relax this assumption midway
through Section~\ref{Sunbdd} below.

\begin{theorem}[Belton--Guillot--Khare--Putinar \cite{BGKP-fixeddim} and
Khare--Tao \cite{Khare-Tao}]\label{Tpowerseries}
Suppose $N \geq 2$ and $n_0 < \cdots < n_{N - 1}$ are non-negative
integers, and $\rho$, $c_0$, \ldots, $c_{N-1}$ are positive
scalars. Given $\epsilon_M \in \{ 0, \pm1 \}$ for all $M > n_{N-1}$,
there exists a power series
\[
f( x ) = c_0 x^{n_0} + \cdots + c_{N-1} x^{n_{N-1}} + %
\sum_{M > n_{N-1}} d_M x^M
\]
such that $f$ is convergent on $( 0, \rho )$, the entrywise map
$f[-]$ preserves positivity on $\cP_N\bigl( ( 0,\rho ) \bigr)$ and
$d_M$ has the same sign (positive, negative or zero) as $\epsilon_M$
for all $M > n_{N-1}$.
\end{theorem}

\begin{proof}[Outline of proof]
The claim is such that it suffices to show the result for exactly one 
$\epsilon_M = -1$. Indeed, given the claim, for each $M > n_{N-1}$ there
exists $\delta_M \in ( 0, 1 / M! )$ such that
$\sum_{j = 0}^{N-1} c_j x^{n_j} + d x^M$ preserves positivity
entrywise on $\cP_N\bigl( ( 0, \rho ) \bigr)$ whenever
$| d | \leq \delta_M$. Now let $d_M := \epsilon_M \delta_M$ for
all $M > n_{N-1}$, and define
\[
f_M( x ) := \sum_{j = 0}^{N_1} c_j x^{n_j} + d_M x^M \quad %
\text{and} \quad %
f( x ) := \sum_{M > n_{N-1}} 2^{n_{N-1} - M} f_M( x ).
\]
Then it may be verified that
$| f( x ) | \leq %
\sum_{j = 0}^{N - 1} c_j x^{n_j} + 2^{n_{N - 1}} e^{x / 2}$,
and hence $f$ has the desired properties.
\end{proof}

Thus it suffices to show the existence of a polynomial positivity
preserver on $\cP_N\bigl( (0,\rho) \bigr)$ with precisely one negative
Maclaurin coefficient, the leading term. In the next few sections we
explain how to achieve this goal. In fact, one can show a more general
result, for real powers as well.

\begin{theorem}[Khare--Tao \cite{Khare-Tao}]\label{Treal1}
Fix an integer $N \geq 2$ and real exponents
$n_0 < \cdots < n_{N-1} < M$ in the set
$\Z_+ \cup [ N - 2, \infty )$. Suppose $\rho$, $c_0$, \ldots,
$c_{N - 1} > 0$ as above. Then there exists $c' < 0$  such that the
function
\[
f( x ) = c_0 x^{n_0} + \cdots + c_{N-1} x^{n_{N-1}} + c' x^M
\qquad \bigl( x \in ( 0,\rho ) \bigr)
\]
preserves positivity entrywise on $\cP_N\bigl( ( 0, \rho ) \bigr)$.
[Here and below, we set $0^0 := 1$.]
\end{theorem}

The restriction of the $n_j$ lying in $\Z_+ \cup [ N - 2, \infty )$ is
a technical one that is explained in a later chapter on the study of
entrywise powers preserving positivity on
$\cP_N\bigl( ( 0, \infty ) \bigr)$; see Theorem~\ref{Tfitzhorn}.

\begin{remark}
A stronger result, Theorem~\ref{Treal2}, which also applies to real
powers, is stated below. We mention numerous ramifications of the results
in this chapter following that result.
\end{remark}

The proofs of the preceding two theorems crucially use type-$A$
representation theory (specifically, a family of symmetric
functions) that naturally emerges here via generalized Vandermonde
determinants. These symmetric homogeneous polynomials are introduced
and used in the next section.

For now, we explain how Theorem~\ref{Treal1} helps achieve a complete
classification of the sign patterns of a family of generalised power
series, of the form
\[
f( x ) = \sum_{j = 0}^\infty c_j x^{n_j}, \qquad %
n_j \in \Z_+ \cup [ N - 2, \infty ) \text{ for all } j \geq 0,
\]
but without the requirement that that exponents are non-decreasing. In
this generality, one first notes that the Horn--Loewner-type
Lemma~\ref{Lhorn} still applies: if some coefficient $c_{j_0} < 0$,
then there must be at least $N$ indices $j$ such that $n_j < n_{j_0}$
and $c_j > 0$. The following result shows that once again, this
necessary condition is best possible.

\begin{theorem}[Classification of sign patterns for real-power series
preservers, Khare--Tao \cite{Khare-Tao}]\label{Tclassify}
Fix an integer $N \geq 2$, and distinct real exponents
$n_0$, $n_1$, \ldots in $\Z_+ \cup [ N - 2, \infty )$. Suppose
$\epsilon_j \in \{ 0, \pm1 \}$ is a choice of sign for each
$j \geq 0$, such that if $\epsilon_{j_0} = -1$ then $\epsilon_j = +1$
for at least $N$ choices of $j$ such that $n_j < n_{j_0}$. Given
any $\rho > 0$, there exists a choice of coefficients
$c_j$ with sign $\epsilon_j$ such that 
\[
f( x ) := \sum_{j = 0}^\infty c_j x^{n_j}
\]
is convergent on $( 0, \rho )$ and preserves positivity entrywise
on $\cP_N\bigl( ( 0, \rho ) \bigr)$.
\end{theorem}

Notice this result is strictly more general than
Theorem~\ref{Tpowerseries}, because the sequence $n_0$, $n_1$, $\ldots$ 
can contain an infinite decreasing sequence of positive non-integer
powers, for example, all rational elements of $[ N - 2, \infty )$.
Thus Theorem~\ref{Tclassify} covers a larger class of functions than
even Hahn or Puiseux series.

Theorem~\ref{Tclassify} is derived from Theorem~\ref{Treal1} in a similar
fashion to the proof of Theorem~\ref{Tpowerseries}, and we refer the
reader to \cite[Section 1]{Khare-Tao} for the details.

\subsection{Schur polynomials; the sharp threshold bound for a single
matrix}\label{Sonematrix}

We now explain how to prove Theorem~\ref{Treal1}. The present section
will discuss the case of integer powers, and end by proving the
theorem for a single `generic' rank-one matrix. In the following
section we show how to extend the results to all rank-one matrices for
integer powers.  The subsequent section will complete the proof for
real powers, and then for matrices of all ranks.

The key new tool that is indispensable to the following analysis is
that of \emph{Schur polynomials}. These can be defined in a number of
equivalent ways; we refer the reader to \cite{Bressoud-Wei} for more
details, including the equivalence of these definitions shown using ideas
of Karlin--Macgregor, Lindstr\"om, and Gessel--Viennot. For our purposes
the definition of Cauchy is the most useful:

\begin{definition}
Given non-negative integers $N \geq 1$ and $n_0 < \cdots < n_{N-1}$,
let
\[
\bn := ( n_0, \ldots, n_{N - 1} )^T, \qquad \text{and} \quad %
\bn_{\min} := ( 0, 1, \ldots, N - 1 )^T,
\]
and define $V( \bn ) := \prod_{0 \leq i < j \leq N - 1} ( n_j - n_i )$.

Given a vector $\bu = ( u_1, \ldots, u_N )^T$ and a non-negative
integer $k$, let
$\bu^{\circ k} := ( u_1^k, \ldots, u_N^k )^T$, and let
$\bu^{\circ \bn}$ be the $N \times N$ matrix with $( j, k )$ entry
$\bu_j^{n_{k - 1}}$.

The \emph{Schur polynomial} in variables $u_1$, \ldots, $u_N$ of
\emph{degree} $\bn$ is given by
\begin{equation}\label{Ecauchy}
s_\bn( \bu ) := \frac{\det \bu^{\circ \bn}}{\det \bu^{\circ \bn_{\min}} }.
\end{equation}
\end{definition}

Notice that the numerator is a generalized Vandermonde determinant,
so a homogeneous and alternating polynomial, while the denominator is
the usual Vandermonde determinant in the indeterminates $u_j$. Hence
their ratio $s_\bn( \bu )$ is a homogeneous symmetric polynomial in
$\Z[u_1, \ldots, u_N]$. It follows that Schur polynomials are well
defined when working over any commutative unital ring.

Schur polynomials are an extremely well-studied family of symmetric
functions. Their appeal lies in the important observation that they
are the characters of all irreducible (finite-dimensional) polynomial
representations of the complex Lie group $GL_n( \C )$ (or of the Lie
algebra $\mathfrak{sl}_{n + 1}( \C )$). In this setting, the
definition of Cauchy is a special case of the Weyl character
formula. Thus, its specialization yields the corresponding Weyl
dimension formula, which will be of use below:
\begin{equation}\label{Eweyldim}
s_\bn( ( 1, \ldots, 1 )^T ) = %
\prod_{0 \leq i < j \leq N - 1} \frac{n_j - n_i}{j - i} = %
\frac{V( \bn )}{V( \bn_{\min} )}.
\end{equation}

An alternate proof of~(\ref{Eweyldim}) comes from the
\emph{principal specialization formula}: for a variable $q$, one has
that
\begin{equation}\label{Especialize}
s_\bn\bigl( ( 1, q, \ldots, q^{N - 1} )^T ) = %
\prod_{0 \leq i < j \leq N - 1} \frac{q^{n_j} - q^{n_i}}{q^j - q^i};
\end{equation}
this follows from~(\ref{Ecauchy}) because now the numerator is also a
standard Vandermonde determinant. We also refer the reader to
\cite{Macdonald} for many more results and properties of Schur
polynomials.

Returning to polynomial positivity preservers, we wish to consider
functions of the form
\[
f( x ) = %
c_0 x^{n_0} + \cdots + c_{N - 1} x^{n_{N - 1}} + c' x^M,
\]
with non-negative integers $n_0 < \cdots < n_{N-1} < M$ and positive
coefficients $c_0$, \ldots, $c_{N-1}$. We are interested in
characterizing those $c' \in \R$ for which the entrywise map $f[-]$
preserve positivity on $\cP_N\bigl( ( 0, \rho ) \bigr)$. By the Schur
product theorem, this is equivalent to finding the smallest $c'$ such
that $f[-]$ is a preserver. We may assume that $c' < 0$, so we
rescale by $t := | c' |^{-1}$ and define
\begin{equation} 
p_t( x ) := t \sum_{j = 0}^{N-1} c_j x^{n_j} - x^M.
\end{equation}
The goal now is to find the smallest $t > 0$ such that $p_t[-]$
preserves positivity on $\cP_N\bigl( ( 0, \rho ) \bigr)$. We next
achieve this goal for a single rank-one matrix.

\begin{proposition}\label{Prank1}
With notation as above, define
\[
\bn_j = %
( n_0, \ldots, n_{j - 1}, \widehat{n_j}, n_{j + 1}, \ldots, n_{N - 1}, M)^T
\]
for $0 \leq j \leq N - 1$. Given a vector $\bu \in ( 0, \infty )^N$
with distinct coordinates, the following are equivalent.
\begin{enumerate}
\item The matrix $p_t[ \bu \bu^T ]$ is positive semidefinite.

\item $\det p_t[ \bu \bu^T] \geq 0$.

\item $t \geq \displaystyle %
\sum_{j = 0}^{N - 1} \frac{s_{\bn_j}( \bu )^2}{c_j s_\bn( \bu )^2}$.
\end{enumerate}
\end{proposition}

In particular, this shows that for a generic rank-one matrix in
$\cP_N\bigl( (0,\rho) \bigr)$, there does exist a
positivity-preserving polynomial with a negative leading term.

In essence, the equivalences in Proposition~\ref{Prank1} hold more
generally; this is distilled into the following lemma.

\begin{lemma}[Khare--Tao \cite{KT-fpsac18}\footnote{The
work~\cite{KT-fpsac18} is an extended abstract of the
paper~\cite{Khare-Tao}, but some of the results in it have different
proofs from~\cite{Khare-Tao}.}]\label{Lktfpsac}
Fix $\bw \in \R^N$ and a positive-definite matrix $H$. Fix $t > 0$ and
define $P_t := t H - \bw \bw^T$. The following are equivalent.
\begin{enumerate}
\item $P_t$ is positive semidefinite.
\item $\det P_t \geq 0$.
\item $\displaystyle t \geq \bw^T H^{-1} \bw = %
1 - \frac{\det( H - \bw \bw^T)}{\det H}$.
\end{enumerate} 
\end{lemma}

We refer the reader to~\cite{KT-fpsac18} for the detailed proof of
Lemma~\ref{Lktfpsac}, remarking only that the equality in assertion (3)
follows by using Schur complements in two different ways to expand the
determinant of the matrix
$\begin{bmatrix} H & \bw \\ \bw^T & 1 \end{bmatrix}$.

Now Proposition~\ref{Prank1} follows directly from Lemma~\ref{Lktfpsac},
by setting
\[
H = \sum_{j = 0}^{N - 1} c_j \bu^{\circ n_j} ( \bu^{\circ n_j} )^T %
\quad \text{and} \quad \bw = \bu^{\circ M},
\]
where $H$ is positive definite because of the following general matrix
factorization (which is also used below).

\begin{proposition}
Let $f( x ) = \sum_{k = 0}^M f_k x^k$ be a polynomial with
coefficients in a commutative ring $R$. For any integer $N \geq 1$
and any vectors $\bu = ( u_1, \ldots, u_N )^T$ and
$\bv = ( v_1, \ldots, v_N )^T \in R^N$, it holds that
\begin{align}\label{Efactor}
& f[ t \bu \bv^T ] = %
\sum_{k = 0}^M f_k t^k \bu^{\circ k} ( \bv^{\circ k})^T\\
 & = \begin{bmatrix}
 1 & u_1 & \cdots & u_1^M \\
 1 & u_2 & \cdots & u_2^M \\
 \vdots & \vdots& \ddots & \vdots \\
 1 & u_N & \cdots & u_N^M
\end{bmatrix}
\begin{bmatrix}
 f_0 & 0 & \cdots & 0 \\
 0 & f_1 t & \cdots & 0 \\
 \vdots & \vdots & \ddots & \vdots \\
 0 & 0 & \cdots & f_M t^M
\end{bmatrix}
\begin{bmatrix}
 1 & v_1 & \cdots & v_1^M \\
 1 & v_2 & \cdots & v_2^M \\
 \vdots & \vdots& \ddots & \vdots \\
 1 & v_N & \cdots & v_N^M
\end{bmatrix}^T,\nonumber
\end{align}
where $1$ is a multiplicative identity which is adjoined to $R$ if
necessary.
\end{proposition}

Now to adopt Lemma~\ref{Lktfpsac}(3), this same equation and the
Cauchy--Binet formula allow one to compute $\det( H - \bw \bw^T )$
in the present situation, and this yields precisely that
$t \geq \displaystyle \sum_{j = 0}^{N - 1} %
\frac{s_{\bn_j}( \bu )^2}{c_j s_\bn( \bu )^2}$,
as desired.

\subsection{The threshold for all rank-one matrices: a Schur positivity
result}

We continue toward a proof of Theorem~\ref{Treal1}. The next step is
to use Proposition~\ref{Prank1} to achieve an intermediate goal: a
threshold bound for $c'$ that works for all rank-one matrices in
$\cP_N\bigl( ( 0, \rho ) \bigr)$, still working with integer
powers. Clearly, to do so one has to understand the supremum of each
ratio $R_j := s_{\bn_j}( \bu )^2 / s_\bn( \bu )^2$, as $\bu$ runs over
vectors in $( 0, \sqrt{\rho} )^N$ with distinct coordinates. More
precisely, one has to understand the supremum of the weighted sum
$\sum_j R_j / c_j$.

This observation was first made in the work~\cite{BGKP-fixeddim} for
the case $n_j = j$, that is, $\bn = \bn_{\min}$. It led to the first
proof of Theorem~\ref{Treal1}, with all of the denominators being the
same: $s_{\bn_{\min}}( \bu ) = 1$. We now use another equivalent
definition of Schur polynomials, by Littlewood, realizing them as sums
of monomials corresponding to certain Young tableaux. Every monomial
has a non-negative integer coefficient. It follows by the continuity and
homogeneity of $s_{\bn_j}$ and the Weyl Dimension
Formula~(\ref{Eweyldim}), that the supremum in the previous paragraph
equals the value at $(\sqrt{\rho}, \ldots, \sqrt{\rho})^T$, namely
\[
\sup_{\bu \in (0,\sqrt{\rho})^N} s_{\bn_j}(\bu)^2 =
\frac{V(\bn_j)^2}{V(\bn_{\min})^2} \rho^{M - n_j}.
\]
Since all of these suprema are attained at the same point
$\sqrt{\rho} ( 1, \ldots, 1 )^T$, the weighted sum in
Proposition~\ref{Prank1}(3) also attains its supremum at the same
point. Thus, we conclude using Proposition~\ref{Prank1} that
\[
f( x ) = \sum_{j = 0}^{N - 1} c_j x^{n_j} + c' x^M
\]
preserves positivity entrywise on all rank-one matrices
$\bu \bu^T \in \cP_N\bigl( ( 0, \rho ) \bigr)$ if and only if
\[
c' \geq -\biggl( \sum_{j = 0}^{N - 1} %
\frac{V( \bn_j )^2}{c_j V( \bn_{\min} )^2} \rho^{M - n_j} %
\biggr)^{-1}.
\]

In fact, if $\bn = \bn_{\min}$ then the entire argument above goes
through even when one changes the domain to the open complex disc
$D( 0, \rho )$, or any intermediate domain
$( 0, \rho ) \subset D \subset D( 0, \rho )$. This is precisely the
content of the main result in~\cite{BGKP-fixeddim}.

\begin{theorem}[Belton--Guillot--Khare--Putinar
\cite{BGKP-fixeddim}]\label{Tbreakthrough}
Fix $\rho > 0$ and integers $M \geq N \geq 2$. Let
\[
f( z ) = \sum_{j = 0}^{N - 1} c_j z^j + c' z^M, \qquad %
\text{where } c_0, \ldots, c_{N - 1},  c' \in \R,
\]
and let $I := \overline{D}( 0, \rho )$ be the closed disc in the
complex plane with centre $0$ and radius $\rho$. The following are
equivalent.
\begin{enumerate}
\item The entrywise map $f[-]$ preserves positivity on
$\cP_N(I)$.

\item The entrywise map $f[-]$ preserves positivity on rank-one matrices
in $\cP_N\bigl( ( 0, \rho ) \bigr)$.

\item Either $c_0$, \ldots, $c_{N - 1}$, $c'$ are all non-negative,
or $c_0$, \ldots, $c_{N - 1}$ are positive and
\[
c' \geq -\biggl( \sum_{j = 0}^{N - 1} %
\frac{V( \bn_j )^2}{c_j V( \bn_{\min})^2 } \rho^{M - j} %
\biggr)^{-1},
\]
where
$\bn_j := %
( 0, 1, \ldots, j - 1, \widehat{j}, j + 1, \ldots, N - 1, M )^T$
for $0 \leq j \leq N - 1$.
\end{enumerate}
\end{theorem}

This theorem provides a complete understanding of which polynomials of
degree at most $N$ preserve positivity entrywise on
$\cP_N\bigl( ( 0, \rho ) \bigr)$ and, more generally, on any subset of
$\cP_N\bigl( \overline{D}( 0, \rho ) \bigr)$ that contains the
rank-one matrices in $\cP_N\bigl( ( 0,\rho ) \bigr)$.

\begin{remark}
Clearly $(1) \implies (2)$ here, and the proof of $(2)
\Longleftrightarrow (3)$ was outlined above via Proposition~\ref{Prank1}.
We defer mentioning the proof strategy for $(2) \implies (1)$, because we
will later see a similar theorem over $I = (0,\rho)$ for more general
powers $n_j$. The proof of that result, Theorem~\ref{Treal2}, will be
outlined in some detail.
\end{remark}

Having dealt with the base case of $\bn = \bn_{\min}$, as well as
$\bn = ( k, k + 1, \ldots, k + N - 1 )$ for any $k \in \Z_+$, which
holds by the Schur product theorem, we now turn to the general case.
In general, $s_\bn( \bu )$ is no longer a monomial, and so it is no
longer clear if and where the supremum of each ratio
$s_{\bn_j}( \bu )^2 / s_\bn( \bu )^2$, or of their weighted sum, is
attained for $\bu \in (0,\sqrt{\rho})^N$. The threshold bound for all
rank-one matrices itself is not apparent, and the bound for all
matrices in $\cP_N\bigl( ( 0, \rho ) \bigr)$ is even more
inaccessible.

By a mathematical miracle, it turns out that the same phenomena as in
the base case hold in general. Namely, the ratio of each $s_{\bn_j}$
and $s_\bn$ attains its supremum at $\sqrt{\rho} ( 1, \ldots, 1 )^T$.
Hence one can proceed as above to obtain a uniform threshold for $c'$,
which works for all rank-one matrices in
$\cP_N\bigl( ( 0, \rho ) \bigr)$.

\begin{example}
To explain the ideas of the preceding paragraph, we present an example.
Suppose
\[
N = 3, \qquad \bn = ( 0, 2, 3 ), \qquad M = 4, \quad \text{and} \quad %
\bu = ( u_1, u_2, u_3 )^T.
\]
Then
\begin{align*}
\bn_3 & = ( 0, 2, 4 ), \\
s_\bn( \bu ) & = u_1 u_2 + u_2 u_3 + u_3 u_1, \\
\text{and} \quad s_{\bn_3}( \bu ) & = %
( u_1 + u_2 ) ( u_2 + u_3 )( u_3 + u_1 ).
\end{align*}
The claim is that $s_{\bn_3}( \bu ) / s_\bn( \bu )$ is coordinatewise
non-decreasing for $\bu \in ( 0, \infty )^3$; the assertion about its
supremum on $( 0, \sqrt{\rho} )^N$ immediately follows from this. It
suffices by symmetry to show the claim only for one variable,
say~$u_3$. By the quotient rule,
\[
s_\bn( \bu ) \partial_{u_3} s_\bm( \bu ) - %
s_\bm( \bu ) \partial_{u_3} s_\bn( \bu ) = %
( u_1 + u_2 ) ( u_1 u_3 + 2 u_1 u_2 + u_2 u_3 ) u_3,
\]
and this is clearly non-negative on the positive orthant, proving the
claim. As we see, the above expression is, in fact, monomial
positive, from which numerical positivity follows immediately.

In fact, an even stronger fact holds. Viewed as a polynomial in $u_3$,
every coefficient in the above expression is in fact
\emph{Schur positive}. In other words, the coefficient of each $u_3^j$
is a non-negative combination of Schur polynomials in $u_1$ and $u_2$:
\[
( u_1 + u_2 ) ( u_1 u_3 + 2 u_1 u_2 + u_2 u_3 ) u_3 = %
\sum_{j \geq 0} p_j( u_1, u_2 ) u_3^j,
\]
where
\[
p_j( u_1, u_2 ) = \begin{cases}
 2 s_{( 1, 3 )}( u_1, u_2 ) & \text{if } j = 1, \\
 s_{( 0, 3 )}( u_1, u_2 ) + s_{( 1, 2 )}( u_1, u_2 ) & %
\text{if } j =  2, \\ 
0 & \text{otherwise}.
\end{cases}
\]
In particular, this implies that each coefficient is monomial
positive, whence numerically positive. We recall here that the
monomial positivity of Schur polynomials follows from the definition
of $s_\bn( \bu )$ using Young tableaux.
\end{example}

The miracle to which we alluded above, is that the Schur positivity in
the preceding example in fact holds in general.

\begin{theorem}[Khare--Tao \cite{Khare-Tao}]\label{Pschur-ratio}
If $n_0 < \cdots < n_{N - 1}$ and $m_0 < \cdots < m_{N - 1}$ are
$N$-tuples of non-negative integers such that
$m_j \geq n_j$ for $j = 0$, \ldots, $N - 1$, then the function
\[
f_{\bm, \bn} : ( 0,\infty )^N \to \R; \ \bu \mapsto %
\frac{s_\bm( \bu )}{s_\bn( \bu )}
\]
is non-decreasing in each coordinate. Furthermore, if
\begin{equation}\label{Equotient}
s_\bn( \bu ) \partial_{u_N} s_\bm( \bu ) - %
s_\bm( \bu ) \partial_{u_N} s_\bn( \bu )
\end{equation}
is considered as a polynomial in $u_N$, then the coefficient of every
monomial~$u_N^j$ is a Schur-positive polynomial in
$u_1$,\ldots, $u_{N - 1}$.
\end{theorem}

The second, stronger part of Theorem~\ref{Pschur-ratio} follows from a
deep and highly non-trivial result in symmetric function theory (or
type-$A$ representation theory) by Lam, Postnikov, and
Pylyavskyy~\cite{LPP}, following earlier results by Skandera. We refer
the reader to this paper and to~\cite{Khare-Tao} for more details.
Notice also that the first assertion in Theorem~\ref{Pschur-ratio}
only requires the numerical positivity of the
expression~(\ref{Equotient}). This is given a separate proof
in~\cite{Khare-Tao}, using the method of condensation due to Charles
Lutwidge Dodgson~\cite{Dod}.%
\footnote{This article by Dodgson immediately follows his
better-known 1865 publication, \emph{Alice's Adventures in
Wonderland}.}
In this context, we add for completeness that in~\cite{Khare-Tao} the
authors also show a log-supermodularity (or FKG, or $MTP_2$)
phenomenon for determinants of totally positive matrices.

\subsection{Real powers; the threshold works for all matrices}

We now return to the proof of Theorem~\ref{Treal1}, which holds for real
powers. Our next step is to observe that the first part of
Theorem~\ref{Pschur-ratio} now holds for all real powers. Since one can
no longer define Schur polynomials in this case, we work with generalized
Vandermonde determinants instead:

\begin{corollary}\label{Cschur-ratio}
Fix $N$-tuples of real powers $\bn = ( n_0 < \cdots < n_{N - 1} )$ and
$\bm = ( m_0 < \cdots < m_{N - 1} )$, such that $n_j \leq m_j$ for all
$j$. Letting $\bu^{\circ \bn} := [ u_j^{n_{k - 1}} ]_{j, k = 1}^N$ as
above, the function
\[
f : \{ \bu \in ( 0,\infty )^N : u_i \neq u_j \text{ if } i \neq j \} %
\to \R; \ %
\bu \mapsto \frac{\det \bu^{\circ \bm}}{\det \bu^{\circ \bn}}
\]
is non-decreasing in each coordinate.
\end{corollary}

We sketch here one proof. The version for integer powers,
Theorem~\ref{Pschur-ratio}, gives the version for rational powers, by
taking a ``common denominator'' $L \in \Z$ such that $L m_j$ and
$L n_j$ are all integers, and using a change of variables
$y_j := u_j^{1 / L}$. The general version for real powers then follows
by considering rational approximations and taking limits.

Corollary~\ref{Cschur-ratio} helps prove the real-power version of
Theorem~\ref{Treal1}, just as Theorem~\ref{Pschur-ratio} would have
shown the integer powers case of Theorem~\ref{Treal1}. Namely, first
note that Proposition~\ref{Prank1} holds even when the $n_j$ are real
powers; the only changes are (a) to assume that the coordinates of
$\bu$ are distinct, and (b) to rephrase the last assertion~(3) to the
following:
\[
t \geq \sum_{j = 0}^{N - 1} %
\frac{( \det \bu^{\circ \bn_j} )^2}{c_j ( \det \bu^{\circ \bn} )^2}.
\]
These arguments help prove the first part of the following result,
which is the culmination of these ideas.

\begin{theorem}[Khare--Tao \cite{Khare-Tao}]\label{Treal2}
Fix an integer $N \geq 1$ and real exponents
$n_0 < \cdots < n_{N - 1} < M$, as well as scalars $\rho > 0$ and
$c_0$, \ldots, $c_{N - 1}$, $c'$. Let
\[
f( x ) := \sum_{j = 0}^{N - 1} c_j x^{n_j} + c' x^M.
\]
The following are equivalent.
\begin{enumerate}
\item The function $f$ preserves positivity entrywise on all
rank-one matrices in $\cP_N\bigl( ( 0, \rho ) \bigr)$.

\item The function $f$ preserves positivity entrywise on all Hankel
rank-one matrices in $\cP_N\bigl( ( 0, \rho ) \bigr)$.

\item Either the coefficients $c_0$, \ldots, $c_{N - 1}$ and $c'$ are
non-negative, or $c_0$, \ldots, $c_{N - 1}$ are positive and
\[
c' \geq -\biggl( \sum_{j = 0}^{N - 1} %
\frac{V( \bn_j )^2}{c_j V( \bn )^2} \rho^{M - n_j} \biggr)^{-1},
\]
where $V( \bn )$ and $\bn_j$ are as defined above.
\end{enumerate}

If, moreover, the exponents $n_j$ all lie in
$\Z_+ \cup [ N - 2, \infty )$, then these assertions are also
 equivalent to the following.
\begin{enumerate}
\setcounter{enumi}{3}
\item The function $f$ preserves positivity entrywise on
$\cP_N\bigl( ( 0, \rho ) \bigr)$.
\end{enumerate}
\end{theorem}

Before sketching the proof, we note several ramifications of this
result.
\begin{enumerate}
\item The theorem completely characterizes linear combinations of up
to $N + 1$ powers that entrywise preserve positivity on
$\cP_N\bigl( ( 0, \rho ) \bigr)$. The same is true for any subset of
$\cP_N\bigl( ( 0, \rho ) \bigr)$ that contains all rank-one positive
semidefinite  Hankel matrices.

\item As discussed above, Theorem~\ref{Treal2} implies
Theorem~\ref{Tclassify}, which helps in understanding which sign
patterns correspond to countable sums of real powers that preserve
positivity entrywise on $\cP_N\bigl( ( 0, \rho ) \bigr)$ (or on the
subset of rank-one matrices). In particular, the existence of
sign patterns which are not all non-negative shows the existence of
functions which preserve positivity on $\cP_N$ but not on
$\cP_{N + 1}$.

\item Theorem~\ref{Treal2} bounds $A^{\circ M}$ in terms of a multiple
of $\sum_{j = 0}^{N - 1} c_j A^{\circ n_j}$. More generally, one can
do this  for an arbitrary convergent power series instead of a
monomial, in the spirit of Theorem~\ref{Tpowerseries}. Even more
generally, one may work with Laplace transforms of measures; see
Corollary~\ref{Claplace} below.
\end{enumerate}

For completeness, we also mention two developments related (somewhat
more distantly) to the above results.
\begin{itemize}
\item A refinement of a conjecture of Cuttler, Greene, and Skandera
(2011) and its proof; see~\cite{Khare-Tao} for more details. In
particular, this approach assists with a novel characterization of weak
majorization, using Schur polynomials.

\item A related ``Schubert cell-type'' stratification of the cone
$\cP_N( \C )$; see \cite{BGKP-fixeddim} for further details.
\end{itemize}

We conclude this section by outlining the proof of Theorem~\ref{Treal2}.

\begin{proof}
Clearly, $(4) \implies (1) \implies (2)$. If $(2)$ holds, then, by
Corollary~\ref{Cmaster} at $a = 0$, either all the $c_j$ and $c'$ are
non-negative, or $c_j$ is positive for all $j$. Thus, we suppose that
$c_j > 0 > c'$.

Note that if 
$\bu( u_0 ) := ( 1, u_0, \ldots, u_0^{N - 1} )^T$ for some
$u_0 \in ( 0, 1 )$, then
\[
A( u_0 ) := \rho u_0^2 \bu( u_0 ) \bu( u_0 )^T
\]
is a rank-one Hankel matrix and hence in our test set. Repeating the
analysis in Section~\ref{Sonematrix}, using generalized Vandermonde
determinants instead of Schur polynomials and rank-one Hankel
matrices of the form $A( u_0 )$,
\begin{align*}
| c'|^{-1} & \geq \sup_{u_0 \in ( 0, 1 )} \sum_{j = 0}^{N - 1} %
\frac{( \det[ \sqrt{\rho} u_0 \bu( u_0 ) ]^{\circ \bn_j} )^2}%
{c_j ( \det[ \sqrt{\rho} u_0 \bu( u_0 ) ]^{\circ \bn} )^2} \\
 & = \sum_{j = 0}^{N - 1} \lim_{u_0 \to 1^-} \sum_{j = 0}^{N - 1} %
\frac{( \det \bu( u_0 )^{\circ \bn_j} )^2}%
{c_j ( \det \bu( u_0 )^{\circ \bn} )^2} ( \rho u_0^2 )^{M - n_j},
\end{align*}
where the equality follows from Corollary~\ref{Cschur-ratio}
above. The real-exponent version of~(\ref{Especialize}) holds if
$q \in ( 0, \infty ) \setminus \{ 1 \}$ and the exponents $n_j$ are real and
non-decreasing:
\[
\det \bu( q )^{\circ \bn} = %
\prod_{0 \leq i < k \leq N - 1} ( q^{n_k} - q^{n_i} ) = %
V( q^{\circ \bn} ).
\]
Applying this identity, the above computation yields
\[
| c' |^{-1} \geq \lim_{u_0 \to 1^-} \sum_{j = 0}^{N - 1} %
\frac{V( u_0^{\circ \bn_j} )^2}{V( u_0^{\circ \bn} )^2} %
\frac{( \rho u_0^2 )^{M - n_j}}{c_j} = %
\sum_{j = 0}^{N - 1} \frac{V( \bn_j )^2}{c_j V( \bn )^2} %
\rho^{M - n_j}.
\]
Thus $(2) \implies (3)$. Conversely, that $(3) \implies (1)$ follows
by a similar analysis to that given above, using
Corollary~\ref{Cschur-ratio} and the density of matrices $\bu \bu^T$,
where $\bu \in \bigl( 0, \sqrt{\rho} \bigr)^N$ has distinct entries,
in the set of all rank-one matrices in
$\cP_N\bigl( ( 0, \rho )\bigr)$.

It remains to show that $(1) \implies (4)$ if all the exponents
$n_j \in \Z_+ \cup [ N - 2, \infty )$. We proceed by induction on
$N$. The case $N = 1$ is immediate. For the inductive step, we apply
the extension principle of the following Proposition~\ref{Pextension}
with $h = f$, which requires verification that $f'[-]$ preserves
positivity on $\cP_{N - 1}$. This is a straightforward calculation via
the induction hypothesis.
\end{proof}

The following extension principle was inspired by work of FitzGerald
and Horn~\cite{FitzHorn}.

\begin{proposition}[Khare--Tao \cite{Khare-Tao}]\label{Pextension}
Suppose $0 < \rho \leq \infty$, and $I = ( 0, \rho )$,
$( -\rho, \rho )$ or the closure of one of these sets. Let
$h : I \to \R$ be a continuously differentiable function
on the interior of $I$. If $h'[-]$ preserves positivity entrywise on
$\cP_{N - 1}( I )$ and $h[-]$ does so on the rank-one matrices in
$\cP_N( I )$, then $h[-]$ in fact preserves positivity on all of
$\cP_N( I )$.%
\footnote{An analogous version of this results holds for
$I = D( 0, \rho)$ or its closure in $\C$, with $h : I \to \C$
analytic. This is used to prove the corresponding implication in
Theorem~\ref{Tbreakthrough} above.}
\end{proposition}

Proposition~\ref{Pextension} relies on two arguments found
in~\cite{FitzHorn}:
(a) every matrix in $\cP_N$ may be written as the sum of a rank-one
matrix in $\cP_N$, and a matrix in $\cP_{N - 1}$ with its last row and
column both zero, and
(b) applying the integral identity
\[
h( x ) - h( y ) = \int_x^y h'( t ) \std t = %
\int_0^1 ( x - y ) h'( \lambda x + ( 1 - \lambda ) y ) \std\lambda
\]
entrywise to this decomposition. See \cite[Section 3]{Khare-Tao} for
more details. The original use of these arguments was when $h$ is a
power function; this is explained in Chapter~\ref{Spower} below.

\subsection{Power series preservers and beyond; unbounded
domains}\label{Sunbdd}

In the remainder of this chapter, we use Theorem~\ref{Treal2} to derive
several corollaries; thus, we retain and use the notation of
that theorem. As discussed following Theorem~\ref{Treal2}, the first
consequence extends the theorem from bounding monomials
$A^{\circ M} = ( x^M )[ A ]$ by a multiple of
$\sum_{j = 0}^{N - 1} c_j A^{\circ n_j}$, to bounding $f[ A ]$ for
more general power series. Even more generally, one can work with
Laplace transforms of real measures on $\R$.

\begin{corollary}[Khare--Tao \cite{Khare-Tao}]\label{Claplace}
Let the notation be as for Theorem~\ref{Treal2}, with $c_j > 0$ for
all $j$.  Suppose $\mu$ is a real measure supported on
$[ n_{N - 1} + \epsilon, \infty )$ for some $\epsilon > 0$, and let
\begin{equation}
g_\mu( x ) := \int_{n_{N-1} + \epsilon}^\infty x^t \std\mu( t ).
\end{equation}
If $g_\mu$ is absolutely convergent at $\rho$, then there exists a
finite threshold $t_\mu > 0$ such that, for all
$A \in \cP_N\bigl( ( 0, \rho ) \bigr)$, the matrix
\[
t_\mu \sum_{j = 0}^{N - 1} c_j A^{\circ n_j} - g_\mu[ A ]
\]
is positive semidefinite.
\end{corollary}

\begin{proof}
By Theorem~\ref{Treal2} and the fact that $\cP_N( \R )$ is a closed
convex cone, it suffices to show the finiteness of the quantity
\[
\int_{n_{N-1} + \epsilon}^\infty \sum_{j = 0}^{N - 1} %
\frac{V( \bn_j )^2}{c_j V( \bn )^2} \rho^{M - n_j} \std\mu_+( M ),
\]
where $\mu_+$ is the positive part of $\mu$. This follows from the
hypotheses.
\end{proof}

We now turn to the $\rho = \infty$ case, which was briefly alluded to
above. In other words, the domain is now unbounded:
$I = ( 0, \infty )$. As in the bounded-domain case, the question of
interest is to classify all possible sign patterns of polynomial or
power-series preservers on $\cP_N( I )$ for a fixed integer $N$.

Similar to the above discussion for bounded $I$, the crucial step in
classifying sign patterns of power series (or more general functions,
as in Theorem~\ref{Tclassify}) is to work with integer powers and
precisely one coefficient that can be negative. Thus, one first
observes that Lemma~\ref{Lhorn}(2) holds in the unbounded-domain case
$I = ( 0, \infty )$. Hence given a polynomial
\[
f( x ) = \sum_{j = 0}^{2 N - 1} c_j x^{n_j} + c' x^M,
\]
where
\[
0 \leq n_0 < \cdots < n_{N-1} < M < n_N < n_{N + 1} \cdots < %
n_{2 N - 1},
\]
if $f[-]$ preserves positivity on
$\cP_N\bigl( ( 0, \infty ) \bigr)$, then either all the coefficients
$c_0$, \ldots, $c_{2 N - 1}$, $c'$ are non-negative, or $c_0$, \ldots,
$c_{2 N-1}$ are positive and $c'$ can be negative. In this case, an
explicit threshold is not known as it is in
Theorem~\ref{Treal2}, but we now explain why such a threshold exists.

We start from~(\ref{Efactor}) and repeat the subsequent analysis via
the Cauchy--Binet formula. To find a uniform threshold for $c'$ that
works for all rank-one matrices in $\cP_N\bigl( ( 0, \infty ) \bigr)$,
it suffices to bound, uniformly from above, certain ratios of sums of
squares of Schur polynomials. This may be done because of the
following tight bounds.

\begin{proposition}[Khare--Tao \cite{Khare-Tao}]\label{Pleading}
If $\bn := ( n_0, \ldots, n_{N - 1} )$ and
$\bu := ( u_1, \ldots, u_N )$, where $n_0 < \cdots < n_{N - 1}$ are
non-negative integers and  $u_1 \leq \cdots \leq u_N$ are non-negative
real numbers, then
\begin{equation}\label{Eleading}
\bu^{\bn - \bn_{\min}} \leq s_\bn( \bu ) \leq %
\frac{V( \bn )}{V( \bn_{\min} )} \bu^{\bn - \bn_{\min}},
\end{equation}
where $\bn_{\min} := ( 0, \ldots, n_{N - 1} )$. The constants $1$ and
$V( \bn ) / V( \bn_{\min} )$ on each side of~(\ref{Eleading}) cannot
be improved.
\end{proposition}

We refer the reader to \cite[Section 4]{Khare-Tao} for further
details, including how Proposition~\ref{Pleading} implies the
existence of preservers $f$ as above for rank-one matrices with
$c' < 0$. The extension from rank-one matrices to all
of~$\cP_N\bigl( ( 0, \infty ) \bigr)$ is carried out using the
extension principle in Proposition~\ref{Pextension}.

In a sense, Proposition~\ref{Pleading} isolates the `leading term' of
every Schur polynomial. This calculation can be generalized to the case
of non-integer powers,\footnote{We refer the reader again to
\cite[Section 5]{Khare-Tao} for the details, which use additional
concepts from type-$A$ representation theory: the
Harish-Chandra--Itzykson--Zuber integral and Gelfand--Tsetlin
patterns.}%
which helps extend the above results for the unbounded
domain $I = ( 0, \infty )$ to real powers. This yields the desired
classification, similar to Theorem~\ref{Tclassify} in the bounded-domain 
case.

\begin{theorem}[Khare--Tao \cite{Khare-Tao}]\label{Tpowerseries2}
Let $N \geq 2$, and let
$\{ \alpha_j : j \geq 0 \} \subset \Z_+ \cup [ N - 2, \infty )$ be a
set of distinct real numbers. For each $j \geq 0$, let
$\epsilon_j \in \{ 0, \pm1 \}$ be a sign and suppose that, whenever
$\epsilon_{j_0} = -1$, then $\epsilon_j = +1$ for at least $N$ choices
of $j$ such that $\alpha_j < \alpha_{i_0}$ and also for at least $N$
choices of $j$ such that $\alpha_j > \alpha_{i_0}$. There exists a
series with real coefficients,
\[
f( x ) = \sum_{j = 0}^\infty c_j x^{\alpha_j}
\]
which converges on $( 0, \infty )$, preserves positivity entrywise on
$\cP_N\bigl( ( 0, \infty ) \bigr)$, and is such that $c_j$ has the
same sign as $\epsilon_j$ for all $j \geq 0$.
\end{theorem}

Note that, in particular, Theorem~\ref{Tpowerseries2} reaffirms that
the Horn--Loewner-type conditions in Lemma~\ref{Lhorn}(2) are sharp.

\subsection{Digression: Schur polynomials from smooth functions, and new
symmetric function identities}\label{Ssymm}

Before proceeding to additional applications of Theorem~\ref{Treal2}
and related results, we take a brief detour to explain how Schur
polynomials arise naturally from any sufficiently differentiable
function.

\begin{theorem}[Khare \cite{Khare}]\label{Phorn}
Fix non-negative integers $m_0 < m_1 < \cdots < m_{N-1}$, as
well as scalars $\epsilon > 0$ and $a \in \R$. Let
$M := m_0 + \cdots + m_{N - 1}$ and suppose the function
$f : [ a, a + \epsilon ) \to \R$ is $M$-times differentiable at $a$.
Given vectors $\bu$, $\bv \in \R^N$, define
$\Delta : [ 0,\epsilon' ) \to \R$ for a sufficiently small
$\epsilon' \in ( 0, \epsilon )$ by setting
\[
\Delta( t ) := \det f[ a \bone_{N \times N} + t \bu \bv^T ].
\]
Then,
\begin{equation}\label{Eloewner2}
\Delta^{( M )}( 0 ) = \sum_{\bm \vdash M} %
\binom{M}{m_0, m_1, \ldots, m_{N-1}} %
V( \bu ) V( \bv ) s_\bm( \bu ) s_\bm( \bv ) %
\prod_{k = 0}^{N - 1} f^{(m_k)}( a ),
\end{equation}
where the first factor in the summand is a multinomial
coefficient, and we sum over all partitions
$\bm = ( m_0, \ldots, m_{N - 1}) $ of $M$ with unequal parts, that is,
$M = m_0 + \cdots + m_{N - 1}$ and $0 \leq m_0 < \cdots < m_{N-1}$.

In particular,
$\Delta( 0 ) = \Delta'( 0 ) = \cdots = %
\Delta^{( \binom{N}{2} - 1 )}(0) = 0$.
\end{theorem}

\begin{remark}
As a special case, if $f : \R \to \R$ is smooth at $a$, and $\bu$,
$\bv \in \R^N$, then defining
$\Delta( t ) := \det f [ a \bone_{N \times N} + t \bu \bv^T ]$ gives a
function $\Delta$ which is smooth at $0$, and
Theorem~\ref{Phorn} gives all of these derivatives via the
formula~(\ref{Eloewner2}). The general version of Theorem~\ref{Phorn}
is a key ingredient in showing Theorem~\ref{Tmaster}, which subsumes
all known variants of Horn--Loewner-type necessary conditions in fixed
dimension.
\end{remark}

The key determinant computation required to prove the original
Horn--Loewner necessary condition in fixed dimension (see
Theorem~\ref{Thorn}) is the special case of Theorem~\ref{Phorn} where
$\bu = \bv$ and $m_j = j$
for all $j$. In this situation, $s_\bm( \bu ) = s_\bm( \bv ) = 1$, so
Schur polynomials do not appear. The general version of
Theorem~\ref{Phorn} decouples the vectors $\bu$ and $\bv$, and holds
for all $M > 0$ if $f$ is smooth (as in Loewner's setting). Moreover,
it reveals the presence of Schur polynomials in every other case
than the ones studied by Loewner, that is, when $M > \binom{N}{2}$.

While Theorem~\ref{Phorn} involves derivatives of a smooth function,
the result and its proof are, in fact, completely algebraic, and valid
over any commutative ring. To show this, an algebraic analogue of the
differential operator is required, with more structure than is given
by a derivation. The precise statement and its proof may be found in
\cite[Section~2]{Khare}.

We conclude this section by applying Theorem~\ref{Phorn} and its
algebraic avatar to symmetric function theory. We begin by recalling
the famous \emph{Cauchy summation identity}
\cite[Example~I.4.6]{Macdonald}: if
$f_0( x ) := 1 + x + x^2 + \cdots$ is the geometric series, viewed as
a formal power series over a commutative unital ring $R$, and $u_1$,
\ldots, $u_N$, $v_1$, \ldots, $v_N$ are commuting variables, then
\begin{equation}
\det f_0[ \bu \bv^T ] = V( \bu ) V( \bv ) %
\sum_\bm s_\bm( \bu ) s_\bm( \bv ),
\end{equation}
where the sum runs over all partitions $\bm$ with at most $N$
parts.%
\footnote{Usually one uses infinitely many indeterminates in symmetric
function theory, but given the connection to the entrywise calculus
in a fixed dimension, we will restrict our attention to $u_j$ and
$v_j$ for $1 \leq j \leq N$.}

A natural question is whether similar formulae hold when $f_0$ is
replaced by other formal power series. Very few such results were
known; this includes one due to Frobenius~\cite{Fr}, for the function
$f_c( x ) := ( 1 - c x ) / ( 1 - x)$ with $c$ an scalar. (This is also
connected to theta functions and elliptic
Frobenius--Stickelberger--Cauchy determinant identities.) For this
function,
\begin{align}
\det f_c[ \bu \bv^T ] & = %
\det\Bigl[ \frac{1 - c u_j v_k}{1 - u_j v_k} \Bigr]_{j, k = 1}^N
\nonumber \\
& = V( \bu ) V( \bv ) ( 1 - c )^{N - 1} \nonumber \\
& \qquad \times \Bigl( %
\sum_{\bm : m_0 = 0} s_\bm( \bu ) s_\bm( \bv ) + ( 1 - c ) %
\sum_{\bm : m_0 > 0} s_\bm( \bu ) s_\bm( \bv ) \Bigr).
\end{align}

A third, obvious identity is if $f$ is a `fewnomial' with at most
$N - 1$ terms. In this case, $f[ \bu \bv^T ]$ is a sum of at most
$N - 1$ rank-one matrices, and so its determinant vanishes.

The following result extends all three of these cases to an arbitrary
formal power series over an arbitrary commutative ring $R$, and with
an additional $\Z_+$-grading.

\begin{theorem}[Khare \cite{Khare}]\label{Tsymm}
Fix a commutative unital ring $R$ and let $t$ be an indeterminate. Let
$f( t ) := \sum_{M \geq 0} f_M t^M \in R[[ t ]]$ be an arbitrary
formal power series. Given vectors $\bu$, $\bv \in R^N$, where $N \geq
1$, we have that
\begin{equation}\label{Esymm}
\det f[ t \bu \bv^T ] = V( \bu ) V( \bv ) \sum_{M \geq \binom{N}{2}} %
t^M \sum_{\bm = (m_{N-1}, \ldots, m_0) \; \vdash M} %
s_\bm( \bu ) s_\bm( \bv ) \prod_{k = 0}^{N - 1} f_{m_k}.
\end{equation}
\end{theorem}

The heart of the proof involves first computing, for each $M \geq 0$,
the coefficient of $t^M$ in $\det f[ t \bu \bv^T ]$, over the
``universal ring''
\[
R' := \Q[ u_1, \ldots, u_N, v_1, \ldots, v_N, f_0, f_1, \ldots ],
\]
where $u_j$, $v_k$ and $f_m$ are algebraically independent over $\Q$.
These coefficients are seen to equal $\Delta^{( M )}(0) / M!$, by the
algebraic version of Theorem~\ref{Phorn}. Thus, (\ref{Esymm}) holds
over $R'$. Then note that both sides of~(\ref{Esymm}) lie in the
subring
$R_0 := \Z[ u_1, \ldots, u_N, v_1, \ldots, v_N, f_0, f_1, \ldots ]$,
so the identity holds in~$R_0$. Finally, it holds as claimed by
specializing from~$R_0$ to~$R$.

An alternate approach to proving Theorem~\ref{Tsymm} is also provided
in~\cite{Khare}. The identity~(\ref{Efactor}) is applied, along with
the Cauchy--Binet formula, to each truncated Taylor--Maclaurin
polynomial $f_{\leq M}$ of $f( x )$. The result follows by taking
limits in the $t$-adic topology, using the $t$-adic continuity of the
determinant function.

\subsection{Further applications: linear matrix inequalities, Rayleigh
quotients, and the cube problem}

This chapter ends with further ramifications and applications of the
above results. First, notice that Theorem~\ref{Treal2} implies the
following linear matrix inequality version that is `sharp' in more
than one sense:

\begin{corollary}\label{Clmi}
Fix $\rho > 0$, real exponents $n_0 < \cdots < n_{N-1} < M$ for some
integer $N \geq 1$, and scalars $c_j > 0$ for all $j$. Then,
\begin{align*}
A^{\circ M} & \leq \cC \Bigl( c_0 A^{\circ n_0} + \cdots + %
c_{N-1} A^{\circ n_{N - 1}} \Bigr), \\
& \qquad \text{where } \cC = \sum_{j = 0}^{N - 1} %
\frac{V( \bn_j )^2}{c_j V( \bn )^2} \rho^{M - n_j},
\end{align*}
for all $A \in \cP_N \bigl( ( 0, \rho ) \bigr)$ of rank one, or of all
ranks if $n_0$, \ldots, $n_{N - 1} \in \Z_+ \cup [ N - 2, \infty )$.
Moreover, the constant $\cC$ is the smallest possible, as is the
number of terms $N$ on the right-hand side.
\end{corollary}

Seeking a uniform threshold such as $\cC$ in the preceding
inequality can also be achieved (as explained above) by first working
with a single positive matrix, then optimizing over all matrices. The
first step here can be recast as an extremal problem that involves
Rayleigh quotients:

\begin{proposition}[see~\cite{BGKP-fixeddim,Khare-Tao}]\label{Prayleigh}
Fix an integer $N \geq 2$ and real exponents
$n_0 < \cdots < n_{N-1} < M$, where each
$n_j \in \Z_+ \cup [ N - 2, \infty )$. Given positive scalars
$c_0$, \ldots, $c_{N-1}$, let
\[
h( x ) := \sum_{j = 0}^{N - 1} c_j x^{n_j} \qquad %
\bigl( x \in ( 0, \infty ) \bigr).
\]
Then, for $0 < \rho < \infty$ and
$A \in \cP_N\bigl( [ 0, \rho ] \bigr)$,
\begin{equation}
t \, h[ A ] \succeq A^{\circ M} \qquad \text{if and only if} \quad %
t \geq %
\varrho( h[ A ]^{\dagger / 2} A^{\circ M} h[ A ]^{\dagger / 2} ),
\end{equation}
where $\varrho[ B ]$ and $B^\dagger$ denote the spectral radius and
the Moore--Penrose pseudo-inverse of a square matrix $B$, respectively.
Moreover, for every non-zero matrix
$A \in \cP_N\bigl( [ 0, \rho] \bigr)$, the following variational
formula holds:
\[
\varrho( h[ A ]^{\dagger / 2} A^{\circ M} h[ A ]^{\dagger / 2} ) = %
\sup_{\bu \in (\ker h[A])^\perp \setminus \{ \bzero \}}
\frac{\bu^T A^{\circ M} \bu}{\bu^T h[\bu \bu^T] \bu} \leq
\sum_{j = 0}^{N - 1} \frac{V( \bn_j )^2}{V( \bn )^2} %
\frac{\rho^{M - n_j}}{c_j}.
\]
\end{proposition}

Proposition~\ref{Prayleigh} is shown using the Kronecker normal form for
matrix pencils; see the treatment in~\cite[Section~X.6]{Gantmacher}.
When the matrix $A$ is a generic rank-one matrix, the above
generalized Rayleigh quotient has a closed-form expression, which
features Schur polynomials for integer powers. This reveals
connections between Rayleigh quotients, spectral radii, and symmetric
functions.

\begin{proposition}
Notation as in Proposition \ref{Prayleigh}; but now with $n_j$ not
necessarily in $\Z_+ \cup [ N - 2, \infty )$. If $A = \bu \bu^T$,
where $\bu \in ( 0, \infty )^N$ has distinct coordinates, then
$h[ A ]$ is invertible, and the threshold bound
\begin{equation}\label{Erankone}
\varrho( h[ A ]^{\dagger / 2} A^{\circ M} h[ A ]^{\dagger / 2} )
 = ( \bu^{\circ M} )^T h[ \bu \bu^T]^{-1} \bu^{\circ M} %
 = \sum_{j = 0}^{N - 1} \frac{( \det \bu^{\circ \bn_j} )^2}%
{c_j ( \det \bu^{\circ \bn} )^2}.
\end{equation}
\end{proposition}

In fact, the proof of the final equality in~(\ref{Erankone}) is
completely algebraic, and reveals new determinantal identities that
hold over any field $\F$ with at least $N$ elements.

\begin{proposition}[Khare--Tao \cite{Khare-Tao}]
Suppose $N \geq 1$ and $0 \leq n_0 < \cdots < n_{N-1} < M$ are integers,
and $\bu, \bv \in \F^N$ each have distinct coordinates.
Let $c_j \in \F^\times$ and define $h(t) := \sum_{j=0}^{N-1}
c_j t^{n_j}$. Then $h[\bu \bv^T]$ is invertible, and
\[
( \bv^{\circ M} )^T h[ \bu \bv^T ]^{-1} \bu^{\circ M} = %
\sum_{j = 0}^{N - 1} 
\frac{\det \bu^{\circ \bn_j} \det \bv^{\circ \bn_j}}%
{c_j \det \bu^{\circ \bn} \det \bv^{\circ \bn}}.
\]
\end{proposition}

The final result is a variant of the matrix-cube
problem~\cite{Nemirovski}, and connects to
spectrahedra~\cite{BPT,Vinzant} and modern optimization theory. Given
two or more real symmetric $N \times N$ matrices
$A_0$, \ldots, $A_{M + 1}$ for  the corresponding matrix cube of size
$2 \eta > 0$ is
\[
\cU[ \eta ] := \bigl\{ A_0 + %
\sum_{m = 1}^{M + 1} u_m A_m : u_m \in [ {-\eta }, \eta ] \bigr\}.
\]
The matrix-cube problem is to find the largest $\eta > 0$ such that
$\cU[ \eta ] \subset \cP_N( \R )$. In the present setting of the
entrywise calculus, the above results imply asymptotically matching
upper and lower bounds for the size of the matrix cube.

\begin{theorem}[see~\cite{BGKP-fixeddim,Khare-Tao}]
Suppose $M \geq 0$ and $0 \leq n_0 < n_1 < \cdots$ are integers. Fix
positive scalars $\rho > 0$, $0 < \alpha_1 < \cdots < \alpha_{M+1}$, and
$c_j > 0\ \forall j \geq 0$, and define for each $N \geq 1$ and each
matrix $A \in \cP_N\bigl( [ 0, \rho ] \bigr)$, the cube
\begin{equation}
\mathcal{U}_A[\eta] := \left\{ \sum_{j=0}^{N-1} c_j A^{\circ n_j} +
\sum_{m=1}^{M+1} u_m A^{\circ (n_{N-1} + \alpha_m)} : u_m \in [-\eta,
\eta] \right\}.
\end{equation}
Also define for $N \geq 1$ and $\alpha > 0$:
\begin{equation}\label{Ecube1}
\mathcal{K}_\alpha(N) := \sum_{j=0}^{N-1}
\frac{V(\bn_j(\alpha,N))^2}{V(\bn(N))^2} \frac{\rho^{\alpha -
n_j}}{c_j},
\end{equation}
where $\bn( N ) := ( n_0, \ldots, n_{N - 1} )^T$, and
\[
\bn_j(\alpha, N) := %
( n_0, \ldots, n_{j - 1}, n_{j + 1}, \ldots, n_{N - 1}, %
n_{N - 1} + \alpha ).
\]
Then for each fixed $N \geq 1$, we have the uniform upper and lower
bounds:
\begin{align}\label{Ecube2}
\begin{aligned}
\eta \leq \bigl( \cK_{\alpha_1}( N ) + \cdots + %
\cK_{\alpha_{M + 1}}( N ) \bigr)^{-1} & \implies %
\cU_A[ \eta ] \subset \cP_N \text{ for all } %
A \in \cP_N\bigl( [ 0, \rho ] \bigr) \\
& \implies \eta \leq \cK_{\alpha_{M + 1}}( N )^{-1}.
\end{aligned}
\end{align}
Moreover, if the $n_j$ grow linearly, in that
\[
\alpha_{M + 1} - \alpha_M \geq n_{j + 1} - n_j %
\qquad \text{for all } j \geq 0,
\]
then the lower and upper bounds for $\eta = \eta_N$ in (\ref{Ecube2})
are asymptotically equal as $N \to \infty$:
\[
\lim_{N \to \infty} \cK_{\alpha_{M + 1}}( N )^{-1} %
\sum_{m = 1}^{M + 1} \cK_{\alpha_m}( N ) = 1.
\]
\end{theorem}


\section{Totally non-negative matrices and positivity preservers}

In this chapter, we discuss variant notions of matrix positivity that
are well studied in the literature, \emph{total positivity} and
\emph{total non-negativity}, and characterize the maps which preserve
these properties.

\begin{definition}
A real matrix $A$ is said to be \emph{totally non-negative} or
\emph{totally positive} if every minor of $A$ is non-negative or
positive, respectively. We will denote these matrices, as well as the
property, by TN and TP.

In older texts, such matrices were called \textit{totally positive}
and \textit{strictly totally positive}, respectively.
\end{definition}



To introduce the theory of total positivity, we can do no better than
quote from the preface of Karlin's magisterial book \cite{Karlin}:
``Total positivity is a concept of considerable power that plays an
important role in various domains of mathematics, statistics and
mechanics''. Karlin goes on to list ``problems involving convexity,
moment spaces, eigenvalues of integral operators, ... oscillation
properties of solutions of linear differential equations ... the
theory of approximations ... statistical decision procedures
... discerning uniformly most powerful tests for hypotheses
... ascertaining optimal policy for inventory and production processes
... analysis of diffusion-type stochastic processes, and ... coupled
mechanical systems.''

Perhaps the earliest result on total positivity is due to Fekete, in
correspondence with P\'olya \cite{Fekete} published in 1912 (see
Lemma~\ref{Lfekete}). Schoenberg observed the variation-diminishing
properties of TP matrices in 1930 \cite{Schoenberg30}, and published a
series of papers on P\'olya frequency functions, which are defined in
terms of total positivity, in the 1950s
\cite{SchoenbergPFF1,SchoenbergPFF2,SchoenbergPFF3}. Independently of
Schoenberg, Krein's investigation of ordinary differential equations
led him to the total positivity of Green's functions for certain
differential operators, and in the mid-1930s his works with
Gantmacher looked at spectral and other properties of totally positive
matrices and kernels; see \cite{GK-book} and \cite[Section~10.6]{Karlin}.

For more on these four authors, one may consult the afterwork of
Pinkus's book on total positivity \cite{Pinkus-book}, which also
contains a wealth of results on totally positive and totally
non-negative matrices. For a modern collection of applications of the
theory of total positivity, see the book edited by Gasca and Micchelli
\cite{GM-book}.

More recently, total positivity has had a major impact on Lie
theory. Lusztig extended the theory of total positivity to the setting
of linear algebraic groups; see \cite{Lusztig-intro} for an exposition
of this work. This led Fomin and Zelevinsky to investigate the
combinatorics of Lusztig's theory \cite{Fomin-Zelevinksy} and resulted
in the invention of cluster algebras \cite{FZ02}. These objects have
generated an enormous amount of activity in a short period of time,
with connections across a wide range of areas within representation
theory, combinatorics, geometry, and mathematical physics. For
the latter, we will mention only the totally non-negative Grassmannian
\cite{PSW}, its connections with scattering amplitudes for quantum
field theories \cite{A-HBCGPJ}, and the work by Kodama and Williams on
regular soliton solutions of the Kadomtsev--Petviashvili equation
\cite{KW14}.

\begin{example}
Perhaps the most well-known class of totally positive matrices
consists of the \emph{(generalized) Vandermonde matrices}: for real
numbers $0 < x_1 < \cdots < x_m$ and $\alpha_1 < \cdots < \alpha_n$,
the $m \times n$ matrix
\[
A := [ x_j^{\alpha_k} ]_{1 \leq j \leq m, \ 1 \leq k \leq n}
\]
is totally positive. Indeed, it suffices to show the positivity of any
such matrix determinant $\det A$ when $m = n$. That $\det A$ is
non-zero follows from Laguerre's extension of Descartes' rule of signs
(see~\cite{Jameson}) and by fixing the $x_j$ and considering a linear
homotopy from $(0, 1, \ldots, n - 1)$ to $( \alpha_1, \ldots, \alpha_n
)$, one obtains a continuous non-vanishing function from the usual
Vandermonde determinant $\prod_{1 \leq j < k \leq n} ( x_k - x_j )$
(which is positive) to $\det A$.
\end{example}

\begin{example}
Another prominent class of symmetric totally positive matrices
consists of the Hankel moment matrices
$H_\mu := [ s_{j + k}( \mu ) ]_{j, k \geq 0}$ corresponding to
admissible measures $\mu$; see Definition~\ref{Dadmissible}.
\end{example}

\subsection{Totally non-negative and totally positive kernels}

An important generalization of TN and TP matrices is given by the
following functional form.

\begin{definition}
Let $X$ and $Y$ be totally ordered sets, and let
$K : X \times Y \to \R$ be a kernel.
\begin{enumerate}
\item The kernel $K$ is \emph{totally positive of order $r$}, denoted
$TP_r$, if, for any $n$-tuples of points
$x_1 < \cdots < x_n$ in $X$ and $y_1 < \cdots < y_n$ in $Y$, where
$1 \leq n \leq r$, the matrix
\[
[ K( x_j, y_k ) ]_{j, k = 1}^n
\]
has positive determinant.

\item The  kernel $K$ is \emph{totally positive} if $K$ is $TP_r$ for
all $r \geq 1$. 

\item Similarly, one defines $TN_r$ kernels and totally non-negative
kernels by replacing the word ``positive'' in the above by
``non-negative.''
\end{enumerate}
\end{definition}

If $X = \{ 1, \ldots, m \}$ and $Y = \{ 1, \ldots, n \}$, we recover
the earlier notions of totally positive and totally non-negative
matrices.
When $X$ and $Y$ are taken to be real intervals, TN and TP kernels can
be thought of as continuous analogues of TN and TP matrices. In fact,
one has a continuous analogue of the Cauchy--Binet formula, which
generalizes its traditional version.

\begin{theorem}[Basic Composition Lemma, see
e.g.~\cite{Karlin,KR}]\label{Tkr}
Suppose $X$, $Y$, $Z \subset \R$ and let $\mu$ be a non-negative Borel
measure on $Y$. Suppose $K : X \times Y \to \R$ and
$L : Y \times Z \to \R$ are pointwise Borel measurable with respect to
$Y$, and let
\[
M : X \times Z \to \R; \ %
( x, z ) \mapsto \int_Y K( x, y ) L( y, z ) \std\mu( y ).
\]
If $M$ is well defined on the whole of $X \times Z$, then
\begin{align*}
\begin{aligned}
 &\det\begin{bmatrix} M( x_1, z_1 ) & \ldots & M( x_1, z_m ) \\
  \vdots & \ddots & \vdots \\
M( x_m, z_1 ) & \ldots & M( x_m, z_m ) \end{bmatrix} \\
  & = \idotsint\limits_{y_1 < y_2 < \cdots < y_m \in Y} %
\det[ K( x_i, y_j ) ]_{i, j = 1}^m  %
\det[ L( y_j, z_k ) ]_{j, k =1}^m \prod_{j = 1}^m \std\mu( y_j ).
\end{aligned}
\end{align*}
\end{theorem}

As an immediate consequence, we have the following corollary.

\begin{corollary}\label{Ckr}
In the setting of Theorem~\ref{Tkr}, if the kernels $K$ and $L$ are
both $TN_r$ or $TP_r$ for some $r \geq 1$, then $M$ has the same
property. In particular, if $K$ and $L$ are both TN or TP, then so is
$M$.
\end{corollary}

We conclude this part with an observation of P\'olya that connects to
a class of well-studied functions, and also implies the positive
definiteness of the Gaussian kernel. Recall from the proof of
Theorem~\ref{Tsch-pos-def} above that this latter property was
crucially used by Schoenberg in characterizing metric space embeddings
into Hilbert space; however, its proof above was only outlined (via
the more sophisticated machinery of Fourier analysis and Bochner's
theorem).

\begin{lemma}[P\'olya]\label{Lpolya}
The Gaussian kernel $K : \R \times \R \to \R$ given by
$K( x, y ) := \exp( -( x - y )^2 )$ is totally positive.
\end{lemma}

\begin{proof}
It suffices to show that every square matrix generated from the kernel
has positive determinant. Given real numbers
$x_1 < \cdots < x_n$ and $y_1 < \cdots < y_n$, we observe the
following factorization:
\begin{multline*}
[ \exp( -( x_j - y_k )^2 ) ]_{j, k = 1}^n \\
= \diag[\exp( -x_j^2 ) ]_{j = 1}^n %
[ \exp( 2 x_j y_k ) ]_{j, k = 1}^n %
\diag[ \exp( -y_k^2 ) ]_{k = 1}^n.
\end{multline*}
The proof concludes by observing that all three matrices on the
right-hand side have positive determinants, the second because it is a
Vandermonde matrix $[ p_j^{\alpha_k} ]$ with
$p_j = \exp( 2 x_j )$ and $\alpha_k = y_k$.
\end{proof}

\begin{example}\label{EPFF}
The Gaussian function $f( x ) = \exp( -x^2 )$ is thus an example of a
\emph{P\'olya frequency function}, that is, one for which $f( x - y )$
is a TP kernel on $\R \times \R$. As noted above, these functions were
intensively studied by Schoenberg, and continue to be much studied in
mathematics and statistics; two of the classic references are
\cite{Brenti,Efron}.
\end{example}

The case of the multivariate Gaussian kernel follows immediately from
the one-dimensional version.

\begin{corollary}
For all $d \geq 1$, the Gaussian kernel
\[
\R^d \times \R^d \to ( 0, \infty ); \ %
( \bx, \by ) \mapsto K( \bx, \by ) := \exp( -\| \bx - \by \|^2 )
\]
is positive semidefinite on $\R^d \times \R^d$. In other words, the
matrix $[ \exp( -\| \bx_j - \bx_k \|^2 ) ]_{j, k = 1}^n$ is positive
semidefinite for all $\bx_1$, \ldots, $\bx_n \in \R^d$.
\end{corollary}

\begin{proof}
The $d = 1$ case is a direct consequence of Lemma~\ref{Lpolya}, and
the case of general $d$ follows from this by using the Schur product
theorem.
\end{proof}

\subsection{Entrywise preservers of totally non-negative Hankel
matrices}\label{Shtn}

In the recent article~\cite{FJS} by Fallat, Johnson, and Sokal, the
authors study when various classes of totally non-negative (TN)
matrices are closed under taking sums or Schur products. As they
observe, the set of all TN matrices is not closed under these
operations; for example, the $3 \times 3$ identity matrix and the
all-ones matrix $\bone_{3 \times 3}$ are both TN but their sum is not.

It is of interest to isolate a class of TN matrices that is a closed
convex cone, and is furthermore closed under taking Schur products.
Indeed, it is under these conditions that the observation of
P\'olya--Szeg\"o (see Section~\ref{Shistory}) holds, leading to large
classes of TN preservers.

Such a class of matrices has been identified in both the
dimension-free as well as fixed-dimension settings. It consists of the
\emph{TN Hankel matrices}. In a fixed dimension, there is the
following classical result from 1912.

\begin{lemma}[Fekete \cite{Fekete}]\label{Lfekete}
Let $A$ be a possibly rectangular real Hankel matrix such that all of
its contiguous minors are positive. Then $A$ is totally positive.
\end{lemma}

Recall that a minor is said to be \emph{contiguous} if it is obtained
from successive rows and successive columns of $A$.

If $A$ is a square Hankel matrix, let $A^{(1)}$ be the square
submatrix of $A$ obtained by removing the first row and the last
column. Notice that every contiguous minor of $A$ is a principal minor
of either $A$ or $A^{(1)}$.  Combined with Fekete's lemma, these
observations help show another folklore result.

\begin{theorem}\label{FHTN}
Let $A$ be a square real Hankel matrix. Then $A$ is TN or TP if and
only if both $A$ and $A^{(1)}$ are positive semidefinite or positive
definite, respectively.
\end{theorem}

Theorem~\ref{FHTN} is a very useful bridge between matrix positivity
and total non-negativity. A related dimension-free variant (see
\cite{Akhiezer,GK}) concerns the Stieltjes moment problem: a sequence
$( s_0, s_1, \ldots,)$ is the moment sequence of an admissible measure
on $\R_+$ (see Definition~\ref{Dadmissible}) if and only if the Hankel
matrices $H := ( s_{j + k} )_{j, k \geq 0}$ and $H^{(1)}$ (obtained by
excising the first row of $H$, or equivalently, the first column) are
both positive semidefinite. By Theorem~\ref{FHTN}, this is equivalent to
saying that $H$ is totally non-negative.

With Theorem~\ref{FHTN} in hand, one can easily show several basic facts
about Hankel TN matrices; we collect these in the following result for
convenience.

\begin{lemma}\label{Lhtn}
For an integer $N \geq 1$ and a set $I \subset \R_+$, let
$\HTN_N( I )$ denote the set of $N \times N$ TN Hankel matrices with
entries in $I$. For brevity, we let $\HTN_N := \HTN_N\bigl( \R_+ )$.
\begin{enumerate}
\item The family $\HTN_N$ is closed under taking sums and non-negative
scalar multiples, or more generally, integrals against non-negative
measures (as long as these exist).

\item In particular, if $\mu$ is an admissible measure supported on
$\R_+$, then its moment matrix
$H_\mu := \bigl( s_{j + k}( \mu ) \bigr)_{j, k = 0}^\infty$ is totally
non-negative.
 
\item $\HTN_N$ is closed under taking entrywise products.

\item If the power series
$f( x ) = \sum_{k \geq 0} c_k x^k$ is convergent on $I \subset\R_+$,
with $c_k \geq 0$ for all $k \geq 0$, then the entrywise map
$f[-]$ preserves total non-negativity on $\HTN_N( I )$, for all
$N \geq 1$.
\end{enumerate}
\end{lemma}

Given Lemma~\ref{Lhtn}(4), which is identical to the start of the story
for positivity preservers, it is natural to expect parallels between the
two settings. For example, one can ask if a Schoenberg-type phenomenon
also holds for preservers of total non-negativity on
$\bigcup_{N \geq 1} \HTN_N\bigl( [ 0, \rho ) \bigr)$ with
$0 < \rho \leq \infty$. As we now explain, this is indeed the case;
we will set $\rho = \infty$ for ease of exposition. From
Theorem~\ref{Thankel} and the subsequent discussion, it follows via
Hamburger's theorem that the class of functions
$\sum_{k \geq 0} c_k x^k$ with all $c_k \geq 0$ characterizes the
entrywise maps preserving the set of moment sequences of admissible
measures supported on $[ -1, 1 ]$. By the above discussion, in
considering the family of matrices $\HTN_N$ for all $N \geq 1$, we are
studying moment sequences of admissible measures supported on
$I = \R_+$, or the related Hausdorff moment problem for
$I = [ 0, 1 ]$. In this case, one also has a Schoenberg-like
characterization, outside of the origin.

\begin{theorem}[Belton--Guillot--Khare--Putinar \cite{BGKP-hankel}]
Let $f : \R_+ \to \R$. The following are equivalent.
\begin{enumerate}
\item Applied entrywise, the map $f$ preserves the set $\HTN_N$ for
all $N \geq 1$.

\item Applied entrywise, the map $f$ preserves positive
semidefiniteness on $\HTN_N$ for all $N \geq 1$.

\item Applied entrywise, the map $f$ preserves the set of moment
sequences of admissible measures supported on $\R_+$.

\item Applied entrywise, the map $f$ preserves the set of moment
sequences of admissible measures supported on $[ 0, 1 ]$.

\item The function $f$ agrees on $( 0, \infty )$ with an absolutely
monotonic entire function, hence is non-decreasing, and
$0 \leq f( 0 ) \leq \lim_{\epsilon \to 0^+} f( \epsilon )$.
\end{enumerate}
\end{theorem}

\begin{remark}
If we work only with $f : ( 0, \infty ) \to \R$, then we are
interested in matrices in $\HTN_N$ with positive entries. Since the
only matrices in $\HTN_N$ with a zero entry are scalar multiples of
the elementary square matrices $E_{1 1}$ or $E_{N N}$ (equivalently,
the only admissible measures supported in $\R_+$ with a zero moment
are of the form $c \delta_0$), the test set does not really reduce,
and hence the preceding theorem still holds in essence: we must
replace $\HTN_N$ by $\HTN_N\bigl( ( 0, \infty ) \bigr)$ in (1) and
(2), reduce the class of admissible measures to those that are not of
the form $c \delta_0$ in (3) and (4), and end (5) at `entire
function'. These five modified statements are, once again, equivalent,
and provide further equivalent conditions to those of Vasudeva
(Theorems~\ref{Tvasudeva} and~\ref{Tvasudeva2}).
\end{remark}

In a similar vein, we now present the classification of sign patterns
of polynomial or power-series functions that preserve TN entrywise in
a fixed dimension on Hankel matrices. This too turns out to be exactly
the same as for positivity preservers.

\begin{theorem}[Khare--Tao \cite{Khare-Tao}]
Fix $\rho > 0$ and real exponents $n_0 < \cdots < n_{N - 1} < M$. For
any real coefficients $c_0$, \ldots, $c_{N-1}$, $c'$, let
\begin{equation}
f( x ) :=  \sum_{j = 0}^{N - 1} c_j x^{n_j} + c' x^M.
\end{equation}
The following are equivalent.
\begin{enumerate}
\item The entrywise map $f[-]$ preserves TN on the rank-one matrices
in~$\HTN_N\bigl( ( 0, \rho ) \bigr)$.

\item The entrywise map $f[-]$ preserves positivity on the rank-one
matrices in $\HTN_N\bigl( ( 0, \rho ) \bigr)$. 

\item Either all the coefficients $c_0$, \ldots, $c_{N - 1}$, $c'$ are
non-negative, or $c_0$, \ldots, $c_{N - 1}$ are positive and
$c' \geq -\cC^{-1}$, where
\begin{equation}
\cC = \sum_{j = 0}^{N - 1} \frac{V( \bn_j )^2}{V( \bn )^2} %
\frac{\rho^{M - n_j}}{c_j}.
\end{equation}
\end{enumerate}
If $n_j \in \Z_+ \cup [ N - 2, \infty )$ for $j = 0$, \ldots, $N - 1$,
then conditions (1), (2) and (3) are further equivalent to the
following.
\begin{enumerate}
\setcounter{enumi}{3}
\item The entrywise map $f[-]$ preserves TN on
$\HTN_N\bigl( [ 0, \rho ] \bigr)$.
\end{enumerate}
\end{theorem}

In particular, this produces further equivalent conditions to
Theorem~\ref{Treal2}. Notice that assertion (2) here is valid because
the rank-one matrices used in proving Theorem~\ref{Treal2} are of the
form $c \bu \bu^T$, where $\bu = ( 1, u_0, \ldots, u_0^{N - 1})^T$,
$u_0 \in ( 0, 1 )$, and $c \in ( 0,\rho )$, so that
$c \bu \bu^T \in \HTN_N\bigl( ( 0, \rho ) \bigr)$.

The consequences of Theorem~\ref{Treal2} also carry over for TN
preservers. For instance, one can bound Laplace transforms analogously
to Corollary~\ref{Claplace}, by replacing the words ``positive
semidefinite'' by ``totally non-negative'' and the set $\cP_N \bigl( (0,
\rho ) \bigr)$ by $\HTN_N\bigl( ( 0, \rho ) \bigr)$. Similarly, one can
completely classify the sign patterns of power series that preserve TN
entrywise on Hankel matrices of a fixed size:

\begin{theorem}[Khare--Tao \cite{Khare-Tao}]
Theorems~\ref{Tclassify} and~\ref{Tpowerseries2} hold upon replacing the
phrase ``preserves positivity entrywise on
$\cP_N\bigl( ( 0, \rho ) \bigr)$'' with
``preserves TN entrywise on $\HTN_N\bigl( ( 0, \rho ) \bigr)$'',
for both $\rho < \infty$ and for $\rho = \infty$.
\end{theorem}

We point the reader to~\cite[End of Section~9]{Khare-Tao} for details.

To conclude, it is natural to seek a general result that relates the
positivity preservers on $\cP_N( I )$ and TN preservers on the set
$\HTN_N( I )$ for domains $I \subset \R_+$. Here is one variant which
helps prove the above theorems, and which essentially follows from
Theorem~\ref{FHTN}.

\begin{proposition}[Khare--Tao \cite{Khare-Tao}]
Fix integers $1 \leq k \leq N$ and a scalar $0 < \rho \leq \infty$.
Suppose $f : [ 0, \rho ) \to \R$ is such that the entrywise map $f[-]$
preserves positivity on $\cP_N^k\bigl( [ 0, \rho ) \bigr)$, the set of
matrices in $\cP_N\bigl( [ 0, \rho ) \bigr)$ with rank no more than
$k$. Then $f[-]$ preserves total non-negativity on
$\HTN_N\bigl( [ 0, \rho ) \bigr) \cap %
\cP_N^k\bigl( [ 0, \rho ) \bigr)$.
\end{proposition}


\subsection{Entrywise preservers of totally non-negative matrices}

The TN property is very rigid when it comes to entrywise operations,
as the following result makes clear.

\begin{theorem}[{\cite[Theorem~2.1]{BGKP-TPP}}]\label{Tfixeddim}
Let $F : \R_+ \to \R$ be a function and let $d := \min( m, n )$,
where~$m$ and~$n$ are positive integers. The following are equivalent.
\begin{enumerate}
\item $F$ preserves TN entrywise on $m \times n$ matrices. 
\item $F$ preserves TN entrywise on $d \times d$ matrices. 
\item $F$ is either a non-negative constant or
\begin{enumerate}
\item[(a)] $(d = 1)$ $F( x ) \geq 0$;
\item[(b)] $(d = 2)$ $F( x ) = c x^\alpha$ for some $c > 0$ and some
$\alpha \geq 0$;
\item[(c)] $(d = 3)$ $F( x ) = c x^\alpha$ for some $c > 0$ and some
$\alpha \geq 1$;
\item[(d)] $(d \geq 4)$ $F( x ) = c x$ for some $c > 0$.
\end{enumerate}
\end{enumerate}
\end{theorem}

\begin{proof}
That $(1) \iff (2)$ is immediate, as is the equivalence of $(2)$ and
$(3)$ when $d = 1$. For larger values of $d$, we sketch the
implication $(2) \implies (3)$. 

For $d = 2$, let the totally non-negative matrices
\begin{equation}\label{E2matrices}
A( x, y ) := \begin{bmatrix}
x & x y \\ 1 & y
\end{bmatrix} \qquad \text{and} \qquad
B( x, y ) := \begin{bmatrix}
x y & x \\ y & 1
\end{bmatrix} \qquad ( x, y \geq 0 ).
\end{equation}
If the non-constant function $F$ preserves TN entrywise for
$2 \times 2$ matrices, then the non-negativity of the determinants of
$F[ A( x, y ) ]$ and $F[ B( x, y ) ]$ gives that
\begin{equation}\label{Emult}
F( x y ) F( 1 ) = F( x ) F( y ) %
\qquad \text{for all } x, y \geq 0.
\end{equation}
It follows that $F$ is strictly positive. Applying Vasudeva's
argument, as set out before Proposition~\ref{Pconvex}, now implies
that $F$ is continuous on $( 0, \infty )$. Since the identity
(\ref{Emult}) shows that $x \mapsto F( x ) / F( 1 )$ is
multiplicative, there exists an exponent $\alpha \in \R_+$ such that
$F( x ) = F( 1 ) x^\alpha$ for all $x > 0$. The final details are left
as an exercise.

For $d = 3$, note that the $3 \times 3$ matrix
$A \oplus 0$ is totally non-negative if and only if the
$2 \times 2$ matrix $A$ is. Hence the previous working gives that
$F( x ) = c x^\alpha$ for some $c > 0$ and $\alpha \geq 0$.
Looking at $\det F[ C ]$ for the totally non-negative matrix
\begin{equation}\label{E3matrix}
C := \begin{bmatrix}
1 & 1 / \sqrt{2} & 0 \\
1 / \sqrt{2} & 1 & 1 / \sqrt{2} \\
0 & 1 / \sqrt{2} & 1
\end{bmatrix}
\end{equation}
shows that we must have $\alpha \geq 1$.

The argument to rule out the possibility that $\alpha \in [ 1, 2 )$
when $d \geq 4$ is more involved, but makes use of an example of
Fallat, Johnson and Sokal  \cite[Example~5.8]{FJS}. Full details are
provided in \cite{BGKP-TPP}.
\end{proof}

If our totally non-negative matrices are also required to be
symmetric, and so positive semidefinite, then the classes of
preservers are enlarged somewhat, but still fairly restrictive.

\begin{theorem}[{\cite[Theorem~2.3]{BGKP-TPP}}]\label{Tsymmetric}
Let $F : \R_+ \to \R$ and let $d$ be a positive integer. The following
are equivalent.
\begin{enumerate}
\item $F$ preserves TN entrywise on symmetric $d \times d$ matrices.
\item $F$ is either a non-negative constant or
\begin{enumerate}
\item[(a)] $(d = 1)$ $F \geq 0$;
\item[(b)] $(d = 2)$ $F$ is non-negative, non-decreasing,
and multiplicatively mid-convex, that is,
$F( \sqrt{x y} )^2 \leq F( x ) F( y )$ for all
$x$, $y \in [ 0,\infty )$, so continuous;
\item[(c)] $(d = 3)$ $F( x ) = c x^\alpha$ for some $c > 0$ and some
$\alpha \geq 1$;
\item[(d)] $(d = 4)$ $F( x ) = c x^\alpha$ for some $c > 0$ and some
$\alpha \in \{1\} \cup [ 2, \infty )$;
\item[(e)] $(d \geq 5$) $F( x ) = c x$ for some $c > 0$.
\end{enumerate}
\end{enumerate} 
\end{theorem}

\subsection{Entrywise preservers of totally positive matrices}

In moving from total non-negativity to total positivity, we face two
significant technical challenges. Firstly, the idea of
realizing totally non-negative $d \times d$ matrices as
submatrices of totally non-negative $( d + 1 ) \times ( d + 1 )$
matrices, by padding with zeros, does not transfer to the TP setting.
Secondly, it is no longer possible to use Vasudeva's idea to establish
multiplicative mid-point convexity, since the test matrices used for this
are not always totally positive.

The first issue leads us into the domain of \emph{totally positive
completion problems} \cite{FJSm}. It is possible to do this
generality, using parametrizations of TP matrices
\cite{Fomin-Zelevinksy} or exterior bordering \cite[Chapter 9]{FJ},
but the following result has the advantage of providing an explicit
embedding into a well-known class of matrices.

\begin{lemma}[{\cite[Lemma~3.2]{BGKP-TPP}}]\label{L2x2tp}
Any totally positive $2 \times 2$ matrix may be realized as the
leading principal submatrix of a positive multiple of a rectangular
totally positive generalized Vandermonde matrix of any larger size.
\end{lemma}

\begin{remark}[{\cite[Remark~3.4]{BGKP-TPP}}]\label{Rstudents}
Lemma~\ref{L2x2tp} can be strengthened to the following completion
result: given integers $m$, $n \geq 2$, an arbitrary $2 \times 2$
matrix $A$ occurs as a minor in a totally positive $m \times n$ matrix
at any given position (that is, in a specified pair of rows and pair
of columns) if and only if $A$ is totally positive.
\end{remark}

The other tool which will be vital to our deliberations is the
following result of Whitney.

\begin{theorem}[{\cite[Theorem 1]{Whitney}}]\label{Twhitney}
The set of totally positive $m \times n$ matrices is dense in
the set of totally non-negative $m \times n$ matrices.
\end{theorem}

With these tools in hand, we are able to provide a complete
classification of the entrywise TP preservers of each fixed size,
akin to the results in the preceding section.

\begin{theorem}[{\cite[Theorem~3.1]{BGKP-TPP}}]\label{Ttp}
Let $F : ( 0, \infty ) \to \R$ be a function and let
$d := \min( m, n )$, where $m$ and~$n$ are positive
integers. The following are equivalent.
\begin{enumerate}
\item $F$ preserves total positivity entrywise on $m \times n$
matrices.
\item $F$ preserves total positivity entrywise on $d \times d$
matrices.
\item The function $F$ satisfies
\begin{enumerate}
\item[(a)] $(d = 1)$ $F( x ) > 0$;
\item[(b)] $(d = 2)$ $F( x ) = c x^\alpha$ for some $c > 0$ and some
$\alpha > 0$;
\item[(c)] $(d = 3)$ $F( x ) = c x^\alpha$ for some $c > 0$ and some
$\alpha \geq 1$;
\item[(d)] $(d \geq 4)$ $F(x) = c x$ for some $c > 0$.
\end{enumerate}
\end{enumerate}
\end{theorem}

\begin{proof}
We sketch the proof that $(2) \implies (3)$ when $d = 2$ and
$d \geq 3$. For the first case, working with the matrix
\[
\begin{bmatrix}
y & x \\
x & y
\end{bmatrix} \qquad ( y > x > 0 )
\]
shows that $F$ takes positive values and is increasing, so is Borel
measurable and continuous except on a countable set. We now fix a
point of continuity~$a$ and use the totally positive matrices
\[
A( x, y, \epsilon ) := \begin{bmatrix} a x & a x y \\
a - \epsilon & a y \end{bmatrix}
\quad \text{and} \quad
B( x, y, \epsilon ) := \begin{bmatrix} a x y & a x \\
a y & a + \epsilon
\end{bmatrix}
\]
to show that
\begin{align*}
0 & \leq \lim_{\epsilon \to 0^+} \det F[ A( x, y, \epsilon ) ] = %
F( a x ) F( a y ) - F( a x y ) F( a ) \\
\text{and} \quad 0 & \leq %
\lim_{\epsilon \to 0^+} \det F[ B( x, y , \epsilon ) ] = %
F( a ) F( a x y ) - F( a x ) F( a y )
\end{align*}
for all $x$, $y > 0$. Hence
$G : x \mapsto F( a x ) / F( a )$ is such that
\[
G( x y ) = G( x ) G( y ) \qquad \text{for all } x, y > 0,
\]
so $G$ is a measurable solution of the Cauchy functional equation.
It follows that $G( x ) = x^\alpha$ for some $\alpha \in \R$.
As $F$, and so $G$, is increasing, we must have $\alpha > 0$.

Finally, if $d \geq 3$, then the embedding of Lemma~\ref{L2x2tp} and
the previous working give positive constants $c$ and $\alpha$ such
that $F( x ) = c x^\alpha$. In particular, the function $F$ admits a
continuous extension $\tilde{F}$ to $\R_+$. The density of TP in TN,
that is, Theorem~\ref{Twhitney}, implies that $\tilde{F}$ preserves TN
entrywise on $d \times d$ matrices. Theorem~\ref{Tfixeddim} now
establishes the form of $\tilde{F}$, and so of~$F$.
\end{proof}

We may consider a version of the previous theorem which restricts to
the case of totally positive matrices which are symmetric. A moment's
thought leads to the consideration of a symmetric version of the
matrix completion problem.

\begin{lemma}[{\cite[Lemma~3.7]{BGKP-TPP}}]\label{L2symtp}
Any symmetric totally positive $2 \times 2$ matrix occurs as the
leading principal submatrix of a totally positive $d \times d$ Hankel
matrix, where $d \geq 2$ can be taken arbitrary large.
\end{lemma}

\begin{proof}
It suffices to embed the matrix
\[
\begin{bmatrix}
1 & a \\
a & b
\end{bmatrix} \qquad ( 0 < a < \sqrt{b} )
\]
into such a Hankel matrix. It is an exercise to prove the existence of
a continuous function $f : [ 0, 1 ] \to \R_+; \ x \mapsto c x^s$ such
that
\[
\int_0^1 f( x ) \std x = a \qquad \text{and} \qquad %
\int_0^1 f( x )^2 \std x = b,
\]
and then setting
\[
a_{j k} := \int_0^1 f( x )^{j + k} \std x \qquad ( j, k \geq 0 )
\]
gives a Hankel matrix $A$ as required. The verification of total
positivity may be made with the help of Andr\'eief's identity,
\begin{multline*}
\det\left[ \int \phi_i( x ) \psi_j( x ) \std x \right]_{i, j = 1}^k \\
 = \frac{1}{k!} \int \cdots \int \det ( \phi_i( x_j ) )_{i, j = 1}^k
\det( \psi_i( x_j ) )_{i, j = 1}^k \std x_1 \cdots \std x_k,
\end{multline*}
where $\phi_i( x ) = f( x )^{\alpha_i - 1}$ and
$\psi_j( x ) = f( x )^{\beta_j - 1}$, with
\[
1 \leq \alpha_1 < \cdots < \alpha_k \leq d \qquad \text{and} \qquad %
1 \leq \beta_1 < \cdots < \beta_k \leq d,
\]
together with the total positivity of generalized Vandermonde
matrices.
\end{proof}

We remark here that the preceding result can be further strengthened to
have the symmetric TP $2 \times 2$ matrix occur in any
``symmetric'' position inside a larger square symmetric TP
Hankel matrix, in the spirit of Remark~\ref{Rstudents}.
See~\cite[Theorem~3.9]{BGKP-TPP} for details.

We now state the symmetric version of Theorem~\ref{Ttp}.

\begin{theorem}[{\cite[Theorem~3.6]{BGKP-TPP}}]\label{TsymmetricTP}
Let $F : ( 0, \infty ) \to \R$ and let $d$ be a positive integer.
The following are equivalent.
\begin{enumerate}
\item $F$ preserves total positivity entrywise on
symmetric $d \times d$ matrices.
\item The function $F$ satisfies
\begin{enumerate}
\item[(a)] $(d = 1)$ $F( x ) > 0$;
\item[(b)] $(d = 2)$ $F$ is positive, increasing,
and multiplicatively mid-convex, that is,
$F( \sqrt{x y} )^2 \leq F( x ) F( y )$ for all $x$,
$y \in ( 0, \infty )$, so continuous;
\item[(c)] $(d = 3)$ $F( x ) = c x^\alpha$ for some $c > 0$ and some
$\alpha \geq 1$;
\item[(d)] $(d = 4)$ $F( x ) = c x^\alpha$ for some $c > 0$ and some
$\alpha \in \{1\} \cup [ 2, \infty )$.
\item[(e)] $(d \geq 5)$ $F( x ) = c x$ for some $c > 0$.
\end{enumerate}
\end{enumerate} 
\end{theorem}

Although we have developed the key ingredients to prove this theorem,
we content ourselves with referring the interested reader to
\cite{BGKP-TPP}.


\section{Power functions}\label{Spower}


A natural approach to tackle the problem of characterizing entrywise
preservers in fixed dimension is to examine if some natural simple
functions preserve positivity. One such family is the collection of power
functions, $f( x ) = x^\alpha$ for $\alpha > 0$. Characterizing which
fractional powers preserve positivity entrywise has recently received
much attention in the literature. One of the first results in this area
reads as follows.

\begin{theorem}%
[FitzGerald and Horn {\cite[Theorem~2.2]{FitzHorn}}]\label{Tfitzhorn}
Let $N \geq 2$ and let $A = [ a_{j k} ] \in \cP_N\bigl( \R_+ \bigr)$.
For any real number $\alpha \geq N - 2$, the matrix
$A^{\circ \alpha} := [ a_{j k}^{\alpha} ]$ is positive semidefinite.
If $0 < \alpha < N- 2 $ and $\alpha$ is not an integer, then there
exists a matrix $A \in \cP_N\bigl( ( 0, \infty ) \bigr)$ such that
$A^{\circ \alpha}$ is not positive semidefinite.
\end{theorem}

Theorem \ref{Tfitzhorn} shows that every real power $\alpha \geq N-2$
entrywise preserves positivity, while no non-integers in $(0,N-2)$ do so.
This surprising ``phase transition'' phenomenon at the integer $N-2$ is
referred to as the ``critical exponent'' for
preserving positivity. Studying which powers entrywise preserve
positivity is a very natural and interesting problem. It also often
provides insights to determine which general functions preserve
positivity. For example, Theorem~\ref{Tfitzhorn} suggests that functions
that entrywise preserve positivity on $\cP_N$ should have a certain
number of non-negative derivatives, which is indeed the case by
Theorem~\ref{Thorn}.

\begin{proof}[Outline of the proof]
The first part of Theorem~\ref{Tfitzhorn} relies on an ingenious idea
that we now sketch. The result is obvious for $N = 2$. Let us assume
it holds for some $N - 1 \geq 2$, let $A \in \cP_N( \R_+ )$, and let
$\alpha \geq N - 2$. Write $A$ in block form,
\[
A = \begin{bmatrix}
B & \xi \\
\xi^T & a_{N N}
\end{bmatrix}, 
\]
where $B$ has dimension $( N - 1 )\times ( N - 1 )$ and
$\xi \in \R^{N - 1}$. Assume without loss of generality that
$a_{N N} \neq 0$ (as the case where $a_{N  N} = 0$ follows from the
induction hypothesis) and let
$\zeta := ( \xi^T, a_{N N} )^T / \sqrt{a_{N N}}$. Then
$A - \zeta \zeta^T = ( B - \xi \xi^T ) / a_{N N} \oplus 0$, where
$( B - \xi \xi^T ) / a_{N N}$ is the Schur complement of $a_{N N}$
in~$A$. Hence $A - \zeta \zeta^T$ is positive semidefinite. By the
fundamental theorem of calculus, for any  $x$, $y \in \R$,
\[
x^\alpha  = y^\alpha + \alpha \int_0^1 ( x - y ) %
( \lambda x + ( 1 - \lambda ) y )^{\alpha - 1 } \std\lambda.
\]
Using the above expression entrywise, we obtain 
\[
A^{\circ \alpha} = \zeta^{\circ \alpha} (  \zeta^{\circ \alpha})^T +
\int_0^1 ( A - \zeta \zeta^T ) \circ %
( \lambda A + ( 1 - \lambda )\zeta \zeta^T )^{\circ (\alpha-1)} %
\std\lambda.
\]
Observe that the entries of the last row and column of the matrix
$A - \zeta \zeta^T$ are all zero. Using the induction hypothesis and
the Schur product theorem, it follows that the integrand is positive
semidefinite, and therefore so is $A^{\circ \alpha}$.

The converse implication in Theorem~\ref{Tfitzhorn} is shown by
considering a matrix of the form $a \bone_{N \times N} + t \bu \bu^T$,
where $a$, $t > 0$, the coordinates of $\bu$ are distinct, and $t1$ is
small. Recall this is the exact same class of matrices that was useful
in proving the Horn--Loewner theorem~\ref{Thorn} as well as its
strengthening in Theorem~\ref{Tmaster}.  The original proof, by
FitzGerald and Horn \cite{FitzHorn}, used
$\bu = ( 1, 2, \ldots, N )^T$, while a later proof by Fallat, Johnson
and Sokal \cite{FJS} used the same argument, now with
$\bu = ( 1, u_0, \ldots, u_0^{N - 1} )^T$; the motivation in
\cite{FJS} was to work with Hankel matrices, and the matrix
$a \bone_{N \times N} + t \bu \bu^T$ is indeed Hankel. That said, the
argument of FitzGerald and Horn works more generally than both of these
proofs, to show that, for any non-integral power $\alpha \in (0, N - 2)$,
$a > 0$, and vector $\bu \in ( 0, \infty )^N$ with distinct coordinates,
there exists $t > 0$ such that $( a \bone_{N \times N} + t \bu
\bu^T)^{\circ \alpha}$ is not positive semidefinite.
\end{proof}

In her 2017 paper~\cite{Jain}, Jain provided a remarkable
strengthening of the result mentioned at the end of the previous
proof, which removes the dependence on $t$ entirely.

\begin{theorem}[Jain~\cite{Jain}]\label{Tjain}
Let
\[
A := [ 1 + u_j u_k ]_{j, k = 1}^N = \bone_{N \times N} + \bu \bu^T,
\]
where $N \geq 2$ and
$\bu = ( u_1, \ldots, u_N )^T \in ( 0, \infty )^N$ has distinct
entries. Then $A^{\circ \alpha}$ is positive semidefinite for
$\alpha \in \R$ if and only if
$\alpha \in \Z_+ \cup [ N - 2, \infty )$.
\end{theorem}

Jain's result identifies a family of rank-two positive semidefinite
matrices, every one of which encodes the classification of powers
preserving positivity over all of $\cP_N\bigl( ( 0, \infty ) \bigr)$.
In a sense, her rank-two family is the culmination of previous work on
positivity preserving powers for $\cP_N\bigl( ( 0, \infty ) \bigr)$,
since for rank-one matrices, every entrywise power preserves
positivity:
$( \bu \bu^T )^{\circ \alpha} = %
\bu^{\circ \alpha} ( \bu^{\circ \alpha} )^T$.

An immediate consequence of these results is the classification of the
entrywise powers preserving positivity on the $N \times N$ Hankel TN
matrices. Recall from the results in Section~\ref{Shtn} (including
Lemma~\ref{Lhtn}(4)) that there is to be expected a strong correlation
between this classification and the one in Theorem~\ref{Tfitzhorn}.

\begin{corollary}
Given $N \geq 2$, the following are equivalent for an exponent
$\alpha \in \R$.
\begin{enumerate}
\item The entrywise power function $x \mapsto x^\alpha$ preserves total
non-negativity on $\HTN_N$ (see Lemma~\ref{Lhtn}).

\item The entrywise map $x \mapsto x^\alpha$ preserves positivity on
$\HTN_N$.

\item The entrywise map $x \mapsto x^\alpha$ preserves positivity on the
matrices in $\HTN_N\bigl( ( 0, \infty ) \bigr)$ of rank at most two.
 
\item The exponent $\alpha \in \Z_+ \cup [ N - 2, \infty )$.
\end{enumerate}
\end{corollary}

\begin{proof}
That $(4) \implies (2)$ and $(2) \implies (1)$ follow from
Theorems~\ref{Tfitzhorn} and~\ref{FHTN}, respectively. That
$(1) \implies (2)$ and $(2) \implies (3)$ are obvious, and
Jain's theorem~\ref{Tjain} shows that $(3) \implies (4)$.
\end{proof}

A problem related to the above study of entrywise powers preserving
positivity, is to characterize infinitely divisible matrices. This
problem was also considered by Horn in \cite{Horn67}. Recall that a
complex $N \times N$ matrix is said to be \emph{infinitely divisible}
if $A^{\circ \alpha} \in \cP_N$ for all $\alpha \in \R_+$. Denote the
\emph{incidence matrix} of $A$ by $M( A )$:
\[
M( A )_{j k} = m_{j k} := \begin{cases}
0 & \text{if } a_{j k} = 0 \\
1 & \text{otherwise}.
\end{cases}
\]
Also, let
\[
L( A ) := \{ \bx \in \C^N : %
\sum_{j, k = 1}^N m_{j k} x_j \overline{x_k} = 0 \},
\]
and note that $L( A )$ is the kernel of $M( A )$ if $M( A )$ is
positive semidefinite.

Assuming the arguments of the entries are chosen in a
\emph{consistent way} \cite{Horn67}, we let
\[
\log^\# A := M( A ) \circ \log[ A ] = %
[ \mu_{j k} \log a_{j k} ]_{j, k = 1}^N,
\]
with the usual convention $0 \log 0 = 0$.

\begin{theorem}[{Horn \cite[Theorem~1.4]{Horn67}}]
An $N \times N$ matrix $A$ is infinitely divisible if and
only if (a) $A$ is Hermitian, with $a_{j j} \geq 0$ for all $j$,
(b) $M( A ) \in \cP_N$, and (c) $\log^\# A$ is positive semidefinite
on $L( A )$.
\end{theorem}

\subsection{Sparsity constraints}\label{Ssparsity}
Theorem~\ref{Tfitzhorn} was recently extended to more structured
matrices.  Given $I \subset \R$ and a graph $G = ( V, E )$ on the
finite vertex set $V = \{ 1, \ldots, N \}$, we define the cone of
positive-semidefinite matrices with \emph{zeros according to $G$}:
\begin{equation}\label{EPG}
\cP_G( I ) := \{ A = [ a_{j k}] \in \cP_N( I ) : a_{j k} = 0 %
\text{ if } ( j, k ) \not\in E \text{ and } i \neq j \}.
\end{equation}
Note that if $( j, k ) \in E$, then the entry $a_{j k}$ is
unconstrained; in particular, it is allowed to be $0$. Consequently,
the cone $\cP_G := \cP_G( \R )$ is a closed subset of $\cP_N$.

A natural refinement of Theorem~\ref{Tfitzhorn} involves studying powers
that entrywise preserve positivity on $\cP_G$. In that case, the flavor
of the problem changes significantly, with the discrete structure of the
graph playing a prominent role.

\begin{definition}[Guillot--Khare--Rajaratnam \cite{GKR-critG}]\label{DcriticalExp}
Given a simple graph $G = (V,E)$, let
\begin{equation}
\cH_G := \{ \alpha \in \R : A^{\circ \alpha} \in \cP_G %
\text{ for all } A \in \cP_G( \R_+ ) \}.
\end{equation}
Define  the \emph{Hadamard critical exponent of $G$} to be
\begin{equation}\label{EcriticalExp}
CE( G ) := %
\min\{\alpha \in \R : [ \alpha, \infty ) \subset \cH_G \}.
\end{equation}
\end{definition}

Notice that, by Theorem \ref{Tfitzhorn}, for every graph
$G = ( V, E )$, the critical exponent $CE( G )$ exists, and lies in
$[ \omega( G ) - 2, | V | - 2 ]$, where $\omega( G )$ is the size of
the largest complete subgraph of $G$, that is, the clique number. To
compute such critical exponents is natural and highly non-trivial.

FitzGerald and Horn proved that $CE( K_n ) = n - 2$ for all $n \geq 2$
(Theorem~\ref{Tfitzhorn}), while it follows from
\cite[Proposition~4.2]{GKR-sparse} that $CE( T ) = 1$ for every
tree~$T$. For a general graph, it is not \textit{a priori} clear what
the critical exponent is or how to compute it. A natural family of
graphs that encompasses both complete graphs and trees is that of
chordal graphs. Recall that a graph is \emph{chordal} if it does not
contain an induced cycle of length $4$ or more. Chordal graphs feature
extensively in many areas, such as the theory of graphical models
\cite{lauritzen}, and in problems involving positive-definite
completions (see~\cite{smith2008positive}). Examples of important
chordal graphs include trees, complete graphs, Apollonian graphs, band
graphs, and split graphs.

Recently, Guillot, Khare, and Rajaratnam \cite{GKR-critG} were able to
compute the complete set of entrywise powers preserving positivity on
$\cP_G$ for all chordal graphs $G$. Here, the critical exponent can be
described purely combinatorially.

\begin{theorem}[{Guillot--Khare--Rajaratnam \cite{GKR-critG}}]\label{TcritGchordal}
Let $K_r^{(1)}$ denote the complete graph with one edge removed, and
let $G$ be a finite simple connected chordal graph. The critical
exponent for entrywise powers preserving positivity on $\cP_G$ is
$r - 2$, where $r$ is the largest integer such that $K_r$ or $K_r^{(1)}$
is an induced subgraph of $G$. More precisely, the set of entrywise
powers preserving $\cP_G$ is $\cH_G = \Z_+ \cup [ r - 2, \infty )$, with
$r$ as before.
\end{theorem} 

The set of entrywise powers preserving positivity was also computed in
\cite{GKR-critG} for cycles and bipartite graphs.

\begin{theorem}[Guillot--Khare--Rajaratnam \cite{GKR-critG}]\label{Tnonchordal}
The critical exponent of cycles and bipartite graphs is $1$.
\end{theorem}

Surprisingly, the critical exponent does not depend on the size of the
graph for cycles and bipartite graphs. In particular, it is striking
that any power greater than $1$ preserves positivity for families of
dense graphs such as bipartite graphs. Such a result is in sharp
contrast to the general case, where there is no underlying structure
of zeros. That small powers can preserve positivity is important for
applications, since such entrywise procedures are often used to
regularize positive definite matrices, such as covariance or
correlation matrices, where the goal is to minimally modify the
entries of the original matrix (see~\cite{Li_Horvath, Zhang_Horvath}
and Chapter~\ref{Sstats} below).

For a general graph, the problem of computing the set $\cH_G$ or the
critical exponent $CE( G )$ remains open. We now outline some other
natural open problems in the area.

\textbf{Problems.}
\begin{enumerate}
\item  In every currently known case (Theorems \ref{TcritGchordal},
\ref{Tnonchordal}), $CE( G )$ is equal to $r - 2$, where $r$ is the
largest integer such that $K_r$ or $K_r^{( 1 )}$ is an induced subgraph
of $G$. Is the same true for every graph $G$?

\item Is $CE( G )$ always an integer? Can this be proved without
computing $CE( G )$ explicitly?

\item Recall that every chordal graph is perfect. Can the critical
 exponent be calculated for other broad families of graphs such as
the family of perfect graphs?
\end{enumerate}

\subsection{Rank constraints and other Loewner properties}

Another approach to generalize Theorem \ref{Tfitzhorn} is to examine
other properties of entrywise functions such as monotonicity,
convexity, and super-additivity (with respect to the Loewner
semidefinite ordering) \cite{Hiai2009, GKR-crit-2sided}. Given a set
$V \subset \cP_N( I )$, recall that a function $f: I \to \R$ is
\begin{itemize}
\item \emph{positive on $V$} with respect to the Loewner ordering if
$f[ A ] \geq 0$ for all $0 \leq A \in V$;

\item \emph{monotone on $V$} with respect to the Loewner ordering if
$f[ A ] \geq f[ B ]$ for all $A$, $B \in V$ such that
$A \geq B \geq 0$; 

\item \emph{convex on $V$} with respect to the Loewner ordering if
$f[ \lambda A + ( 1 - \lambda ) B ] \leq %
\lambda f[ A ] + ( 1 - \lambda )  f[ B ]$
for all $\lambda [ 0, 1 ]$ and all $A$, $B \in V$ such that
$A \geq B \geq 0$;

\item \emph{super-additive on $V$} with respect to the Loewner
ordering if $f[A + B] \geq f[ A ] + f[ B ]$ for all $A$, $B \in V$ for
which $f[ A + B ]$ is defined.
\end{itemize}

The following relations between the first three notions were obtained
by Hiai.

\begin{theorem}[{Hiai \cite[Theorem 3.2]{Hiai2009}}]
Let $I = ( -\rho, \rho )$ for some $\rho > 0$.
\begin{enumerate}
\item For each $N \geq 3$, the function $f$ is monotone on
$\cP_N( I )$ if and only if $f$ is differentiable on $I$ and $f'$ is
positive on $\cP_N( I )$.
\item For each $N \geq 2$, the function $f$ is convex on
$\cP_N( I )$ if and only if $f$ is differentiable on $I$ and $f'$ is
monotone on $\cP_N( I )$.
\end{enumerate}
\end{theorem}

Power functions satisfying any of the above four properties have been
characterized by various authors. In recent work, Hiai \cite{Hiai2009}
has extended Theorem~\ref{Tfitzhorn} by considering the odd and even
extensions of the power functions to $\R$. For $\alpha > 0$, the even
and odd extensions to $\R$ of the power function
$f_\alpha( x ) := x^\alpha$ are defined to be
$\phi_\alpha( x ) := | x |^\alpha$ and
$\psi_\alpha( x ) := \sign( x ) | x |^\alpha$.  The first study of
powers $\alpha > 0$ for which $\phi_\alpha$ preserves positivity
entrywise on $\cP_N(\R)$ was carried out by Bhatia and Elsner
\cite{Bhatia-Elsner}. Subsequently, Hiai studied the power functions
$\phi_\alpha$ and $\psi_\alpha$ that preserve Loewner
positivity, monotonicity, and convexity entrywise, and showed for
positivity preservers that the same phase transition occurs at
$n - 2$ for $\phi_\alpha$ and $\psi_\alpha$, as demonstrated in
\cite{FitzHorn}.  The work was generalized in \cite{GKR-crit-2sided}
to matrices satisfying rank constraints.

\begin{definition}
Fix non-negative integers $n \geq 2$ and $n \geq k$, and a set
$I \subset \R$. Let $\cP_n^k( I )$ denote the subset of matrices in
$\cP_n( I )$ that have rank at most $k$, and let
\begin{align}\label{Epos}
\pos( n, k ) & := \{ \alpha > 0 : x^\alpha %
\text{ preserves positivity on } \cP_n^k( \R_+ ) \}, \notag \\
\pos^\phi( n, k ) & := \{ \alpha > 0 : \phi_\alpha %
\text{ preserves positivity on } \cP_n^k( \R ) \}, \\
\pos^\psi( n, k ) & := \{ \alpha > 0 : \psi_\alpha %
\text{ preserves positivity on } \cP_n^k( \R ) \}. \notag
\end{align}
Similarly, let $\cH_J( n, k )$, $\cH_J^\phi( n, k )$ and
$\cH_J^\psi( n, k )$ denote sets of the entrywise powers preserving
Loewner properties on $\cP_n^k( \R_+ )$ or $\cP_n^k( \R )$, where
$J \in %
\{ \text{monotonicity}, \text{convexity}, \text{super-additivity} \}$.
\end{definition}

The set of entrywise powers preserving the above notions are given in
the table below (see \cite[Theorem 1.2]{GKR-crit-2sided}).
\begin{table}[ht]
{\footnotesize
\begin{tabular}{|c|c|c|c|}
\hline
$J$ & $\cH_J( n, k )$ & $\cH_J^\phi( n, k )$ & $\cH_J^\psi( n, k )$ \\
\hline\hline
\textbf{Positivity} & \multicolumn{3}{c|}{} \\
\hline
$k = 1$ & $\R$ & $\R$ & $\R$ \\
& G--K--R & G--K--R & G--K--R \\ \hline
& $\N \cup [ n - 2, \infty )$ & $2 \N \cup [ n - 2, \infty )$ &
$( -1 + 2 \N ) \cup [ n - 2, \infty )$ \\
$2 \leq k \leq n$ & FitzGerald--Horn & Hiai, Bhatia--Elsner, & Hiai,
G--K--R\\ & & G--K--R & \\ \hline
\textbf{Monotonicity} & \multicolumn{3}{c|}{} \\
\hline
$k = 1$ & $\R_+$ & $\R_+$ & $\R_+$\\
& G--K--R & G--K--R & G--K--R \\ \hline
$2 \leq k \leq n$ & $\N \cup [ n - 1, \infty )$ & %
$2 \N \cup [ n - 1, \infty )$ & %
$( -1 + 2 \N ) \cup [ n - 1, \infty )$ \\
& FitzGerald--Horn & Hiai, G--K--R & Hiai, G--K--R \\ \hline
\textbf{Convexity} & \multicolumn{3}{c|}{} \\
\hline
$k = 1$ & $[ 1, \infty )$ & $[ 1, \infty )$ & $[ 1, \infty )$ \\
& G--K--R & G--K--R & G--K--R \\ \hline
$2 \leq k \leq n$ & $\N \cup [ n, \infty )$ & %
$2 \N \cup [ n, \infty )$ & $( -1 + 2 \N ) \cup [ n, \infty )$ \\
& Hiai, G--K--R & Hiai, G--K--R & Hiai, G--K--R \\ \hline
\textbf{Super-additivity} & \multicolumn{3}{c|}{} \\
\hline
$1 \leq k \leq n$ & $\N \cup [ n, \infty )$ & %
$2 \N \cup [ n, \infty )$ & $( -1 + 2 \N ) \cup [ n, \infty ) $ \\
& G--K--R & G--K--R & G--K--R \\ \hline
\end{tabular}}
\vspace*{2ex}
\caption{Summary of real Hadamard powers preserving Loewner properties,
with additional rank constraints. See Bhatia--Elsner
\cite{Bhatia-Elsner}, FitzGerald--Horn \cite{FitzHorn},
Guillot--Khare--Rajaratnam \cite{GKR-crit-2sided}, and Hiai
\cite{Hiai2009}.}\label{H-list}
\end{table}

\section{Motivation from statistics}\label{Sstats}

The study of entrywise functions preserving positivity has recently
attracted renewed attraction due to its importance in the estimation
and regularization of covariance/correlation matrices. Recall that the
covariance between two random variables $X_j$ and $X_k$ is given by
\[
\sigma_{j k} = \Cov( X_j, X_k ) = %
E\bigl[ ( X_j - E[ X_j ] ) ( X_k - E[ X_k ] )\bigr], 
\]
where $E[ X_j]$ denotes the expectation of $X_j$. In particular,
$\Cov( X_j, X_j) = \Var( X_j )$, the variance of $X_j$. The
\emph{covariance matrix} of a random vector
$\bX := ( X_1, \ldots, X_m )$, is the matrix
$\Sigma := [ \Cov( X_j, X_k ) ]_{j, k = 1}^m$. Covariance
matrices are a fundamental tool that measure linear dependencies
between random variables. In order to discover relations between
variables in data, statisticians and applied scientists need to obtain
estimates of the covariance matrix $\Sigma$ from observations
$\bx_1$, \ldots, $\bx_n \in \R^m$ of $\bX$. A traditional estimator
of $\Sigma$ is the \emph{sample covariance matrix} $S$ given by
\begin{equation}\label{ESampleCov}
S = [ s_{j k} ]_{j, k = 1}^m = \frac{1}{n - 1} %
\sum_{i = 1}^n ( \bx_i - \overline{\bx} ) %
( \bx_i - \overline{\bx} )^T, 
\end{equation}
where $\overline{\bx} := \frac{1}{n} \sum_{i = 1}^n \bx_i$ is the
average of the observations. In the case where the random vector $\bX$
has a multivariate normal distribution with mean $\mu$ and covariance
matrix $\Sigma$, one can show that $\overline{\bx}$ and
$\frac{n - 1}{n} S$ are the maximum likelihood estimators of $\mu$ and
$\Sigma$, respectively \cite[Chapter~3]{Anderson2003}. It is not
difficult to show that $S$ is an unbiased estimator of $\Sigma$. More
generally, under weak assumptions, one can show that the distribution
of $\sqrt{n}( S - \Sigma )$ is asymptotically normal as
$n \to \infty$. The exact description of the limiting distribution
depends on the moments and the cumulants of $\bX$ (see
\cite[Chapter~6.3]{Bilodeau99}). For example, in the
two-dimensional case, we have the following result.

Let $N_m( \mu, \Sigma )$ denote the $m$-dimensional normal
distribution with mean $\mu$ and covariance matrix $\Sigma$.

\begin{proposition}[{see \cite[Example~6.4]{Bilodeau99}}]
Let $\bx_1$, \ldots, $\bx_n \in \R^2$ be an independent and
identically distributed sample from a bivariate vector
$\bX = ( X_1, X_2 )$ with mean $\mu = ( \mu_1, \mu_2 )$
and finite fourth-order moments, and let $S$ be as in
Equation~(\ref{ESampleCov}). Then
\[
\sqrt{n} \left[
\begin{bmatrix} s_1^2 \\ s_{1 2} \\ s_2^2 \end{bmatrix} - %
\begin{bmatrix} \sigma_1^2 \\ \sigma_{1 2} \\ \sigma_2^2 \end{bmatrix}
\right] \overset{\rd}{\longrightarrow} N_3( \bzero, \Omega ),
\]
where $\Omega$ is the symmetric $3 \times 3$ matrix
\[
\Omega = \begin{bmatrix}
\mu_4^1 - ( \mu_2^1 )^2 & %
\mu_{3 1}^{1 2} - \mu_{1 1}^{1 2} \mu_2^1 & %
\mu_{2 2}^{1 2} - \mu_2^1 \mu_2^2 \\[1ex]
\mu_{3 1}^{1 2} - \mu_{1 1}^{1 2} \mu_2^1 & %
\mu_{2 2}^{1 2} - ( \mu_{1 1}^{1 2} )^2 & %
\mu_{1 3}^{1 2} - \mu_{1 1}^{1 2} \mu_2^2 \\[1ex]
\mu_{2 2}^{1 2} - \mu_2^1 \mu_2^2 & %
\mu_{3 1}^{1 2} - \mu_{1 1}^{1 2} \mu_2^1 & %
\mu_4^2 - ( \mu_2^2 )^2
\end{bmatrix}, 
\]
and $\mu^i_k = E[ ( X_i - \mu_i )^k ]$ and
$\mu^{i j}_{k l} = E[ ( X_ i -\mu_i )^k ( X_j - \mu_j )^l ]$.
\end{proposition}

In traditional statistics, one usually assumes the number of samples
$n$ is large enough for asymptotic results such as the one above to
apply. In covariance estimation, one typically requires a sample size
at least a few times the number of variables $m$ for that to apply. In
such a case, the sample covariance matrix provides a good
approximation of the true covariance matrix $\Sigma$. However, this
ideal setting is rarely seen nowadays. Indeed, our systematic and
automated way of collecting data today yields datasets where the
number of variables is often orders of magnitude larger than the
number of instances available for study \cite{Donoho2000}. Classical
statistical methods were not designed and are not suitable to analyze
data in such settings. Developing new methodologies that are adapted
to modern high-dimensional problems is the object of active
research. In the case of covariance estimation, several strategies
have been proposed to replace the traditional sample covariance matrix
estimator $S$. These approaches typically leverage low-dimensional
structures in the data (low rank, sparsity, \ldots) to obtain
reasonable covariance estimates, even when the sample size is small
compared to the dimension of the problem (see \cite{Pourahmadi2013}
for a detailed description of such techniques). One such approach
involves applying functions to the entries of sample covariance
matrices to improve their properties (see e.g.~\cite{bai2007semicircle, BickelLevina2008, ElKaroui2008, hero_rajaratnam, Hero_Rajaratnam2012, Li_Horvath, Rothman2009, Zhang_Horvath}). For example, \emph{hard thresholding}
a matrix entails setting to zero the entries of the matrix that are
smaller in absolute value than a prescribed value $\epsilon > 0$
(thinking the corresponding variables are independent, for
example). Letting
\begin{equation}\label{Ethresholding}
f_\epsilon^H( x ) = \begin{cases}
 x & \text{ if } | x | > \epsilon, \\
 0 & \text{ otherwise},
\end{cases}
\end{equation}
thresholding is equivalent to applying the function $f_\epsilon^H$
entrywise to the entries of the matrix. Another popular example that
was first studied in the context of wavelet shrinkage \cite{Donoho94}
is \emph{soft thresholding}, where $f_\epsilon^H$ is replaced by
\[
f_\epsilon^S : x \mapsto \sign( x ) \bigl( | x | - \epsilon \bigr)_+ %
\qquad \text{with } y_+ := \max\{ y, 0 \}.
\]
Soft thresholding not only sets small entries to zero, it also shrinks
all the other entries continuously towards zero. Several other
thresholding and shrinkage procedures were also recently proposed in
the context of covariance estimation (see \cite{Fan2016} and the
references therein).

Compared to other techniques, the above procedure has several
advantages.  Firstly, the resulting estimators are often significantly
more precise than the sample covariance matrices. Secondly, applying a
function to the entries of a matrix is very simple and not
computationally intensive. The procedure can therefore be performed in
very high dimensions and in real-time applications. This is in
contrast to several other techniques that require solving optimization
problems and often become too intensive to be used in modern
applications. A downside of the entrywise calculus, however, is that
the positive definiteness of the resulting matrices is not
guaranteed. As the parameter space of covariance matrices is the cone
of positive definite matrices, it is critical that the resulting
matrices be positive definite for the technique to be useful and
widely applicable. The problem of characterizing positivity preservers
thus has an immediate impact in the area of covariance estimation by
providing useful functions that can be applied entrywise to covariance
estimates in order to regularize them.

Several characterizations of when thresholding procedures preserve
positivity have recently been obtained.

\subsection{Thresholding with respect to a graph}

In \cite{GuillotRajaratnam2012}, the concept of thresholding with
respect to a graph was examined. In this context, the elements to
threshold are encoded in a graph $G = ( V, E )$ with
$V = \{ 1, \ldots, p \}$. If $A = ( a_{j k} )$ is a $p \times p$
matrix, we denote by $A_G$ the matrix with entries
\[
( A_G )_{j k} = \begin{cases}
 a_{j k} & \text{if } ( j, k ) \in E \text{ or } j = k, \\
 0 & \text{otherwise}.
\end{cases}
\]
We say that $A_G$ is the matrix obtain by thresholding $A$ with
respect to the graph $G$. The main result of
\cite{GuillotRajaratnam2012} characterizes the graphs $G$ for which
the corresponding thresholding procedure preserves positivity. Denote
by $\cP_N^+$ the set of real symmetric $N \times N$ positive definite
matrices and by $\cP_G^+$ the subset of positive definite matrices
contained in $\cP_G$ (see Equation (\ref{EPG})).

\begin{theorem}[{Guillot--Rajaratnam \cite[Theorem 3.1]{GuillotRajaratnam2012}}]\label{TGraphThreshold}
The following are equivalent:
\begin{enumerate}
\item $A_G \in \cP_N^+$ for all $A \in \cP_N^+$;
\item $G = \bigcup_{i = 1}^d G_i$, where $G_1$, \ldots, $G_d$
are disconnected and complete components of $G$.
\end{enumerate}
\end{theorem}
The implication $(2) \implies (1)$ of the theorem is intuitive and
straightforward, since principal submatrices of positive definite
matrices are positive definite. That $(1) \implies (2)$ may come as a
surprise though, and shows that indiscriminate or arbitrary
thresholding of a positive definite matrix can quickly lead to loss of
positive definiteness.

Theorem \ref{TGraphThreshold} also generalizes to matrices that
already have zero entries. In that case, the characterization of the
positivity preservers remains essentially the same.

\begin{theorem}[{Guillot--Rajaratnam \cite[Theorem 3.3]{GuillotRajaratnam2012}}]\label{TGraphThresholdGeneral}
Let $G = ( V, E )$ be an undirected graph and let $H = ( V, E' )$ be a
subgraph of $G$, so that $E' \subset E$. Then $A_H$ is positive definite
for every $A \in \cP_G^+$ if and only if $H = G_1 \cup \cdots \cup G_k$,
where $G_1$, \ldots, $G_k$ are disconnected induced subgraphs of $G$.
\end{theorem}

\subsection{Hard and soft thresholding}
Theorems \ref{TGraphThreshold} and \ref{TGraphThresholdGeneral}
address the case where positive definite matrices are thresholded with
respect to a given pattern of entries, regardless of the magnitude of
the entries of the original matrix. The more natural case where the
entries are hard or soft-thresholded was studied in
\cite{GuillotRajaratnam2012, GuillotRajaratnam2012b}. In applications,
it is uncommon to threshold the diagonal entries of estimated
covariance matrices, as the diagonal contains the variance of the
underlying variables. Hence, for a given function $f: \R \to \R$ and a
real matrix $A = [ a_{j k} ]$, we let the matrix $f^*[ A ]$ be defined
by setting
\[
f^*[A]_{j k} := \begin{cases}
 f( a_{j k} ) & \text{if } j \neq k, \\
 a_{j k} & \text{otherwise}.
\end{cases}
\]

\begin{theorem}[{Guillot--Rajaratnam \cite[Theorem 3.6]{GuillotRajaratnam2012}}]\label{THardThres}
Let $G$ be a connected undirected graph with $n \geq 3$ vertices. The
following are equivalent.
\begin{enumerate}
\item There exists $\epsilon > 0$ such that, for every
$A \in \cP_G^+$, we have $( f_\epsilon^H )^*[ A ] \in \cP_n^+$.
\item For every $\epsilon > 0$ and every $A \in \cP_G^+$, we have
$f_\epsilon^H[ A ] \in \cP_n^+$.
\item $G$ is a tree.
\end{enumerate}
\end{theorem}

The case of soft-thresholding was considered in
\cite{GuillotRajaratnam2012b}. Surprisingly, the characterization of
the thresholding levels that preserve positivity is exactly the same
as in the case of hard-thresholding.

\begin{theorem}[{Guillot--Rajaratnam \cite[Theorem 3.2]{GuillotRajaratnam2012b}}]\label{TSoftThres}
Let $G = ( V, E )$ be a connected graph with $n \geq 3$ vertices. Then the following are equivalent:
\begin{enumerate}
\item There exists $\epsilon > 0$ such that for every $A \in \cP_G^+$, we have $(f_\epsilon^S)^*[A] \in \cP_n^+$.
\item For every $\epsilon > 0$ and every $A \in \cP_G^+$, we have $f_\epsilon^S[A] \in \cP_n^+$.
\item $G$ is a tree.
\end{enumerate}
\end{theorem}
 
An extension of Schoenberg's theorem (Theorem~\ref{Tschoenberg}) to
the case where the function $f$ is only applied to the off-diagonal
entries of the matrix was also obtained in
\cite{GuillotRajaratnam2012b}.

\begin{theorem}[{Guillot--Rajaratnam \cite[Theorem 4.21]{GuillotRajaratnam2012b}}]\label{TschoenbergOff}
Let $0 < \rho \leq \infty$ and $f : (-\rho, \rho) \to \R$.
The matrix $f^*[ A ]$ is positive semidefinite for all
$A \in \cP_n\bigl( ( -\rho, \rho ) \bigr)$ and all $n \geq 1$
if and only if $f( x ) = x g( x )$, where
\begin{enumerate}
\item $g$ is analytic on the disc $D(0,\rho)$; 
\item $\|g\|_\infty \leq 1$;
\item $g$ is absolutely monotonic on $(0, \rho)$. 
\end{enumerate}
When $\rho = \infty$, the only functions satisfying the above conditions are the affine functions $f(x) = ax$ for $0 \leq a \leq 1$.  
\end{theorem}

\subsection{Rank and sparsity constraints}

An explicit and useful characterization of entrywise functions preserving
positivity on $\cP_N$ for a fixed $N$ still remains out of reach as of
today. Motivated by applications in statistics, the authors in
\cite{GKR-sparse, GKR-lowrank} examined the cases where the matrices in
$\cP_N$ satisfy supplementary rank and sparsity constraints that are
common in applications.

Observe that the sample covariance matrix (Equation (\ref{ESampleCov}))
has rank at most $n$, where $n$ is the number of samples used to compute
it. Moreover, as explained in Chapter~\ref{Sstats}, it is common in
modern applications that $n$ is much smaller than the dimension $p$.
Hence, when studying the regularization approach described in
Chapter~\ref{Sstats}, it is natural to consider positive semidefinite
matrices with rank bounded above.  

An immediate application of Schoenberg's theorem on spheres (see
Equation~\eqref{EschoenbergSphere}) provides a characterization of
entrywise positivity preservers of correlation matrices of all
dimensions, with rank bounded above by~$n$. Recall that a correlation
matrix is the covariance matrix of a random vector where each variable
has variance~$1$, so is a positive semidefinite matrix with diagonal
entries equal to~$1$. As in Equation~\eqref{EschoenbergSphere}, we
denote the ultraspherical orthogonal polynomials by $P_k^{(\lambda)}$.

\begin{theorem}[{Reformulation of \cite[Theorem 1]{Schoenberg-Duke42}}]
Let $n \in \N$ and let $f: [ -1, 1 ] \to \R$.  The following are
equivalent.
\begin{enumerate}
\item $f[ A ] \in \cP_N$ for all correlation matrices
$A \in \cP_N\bigl( [ -1, 1 ] \bigr)$ with rank no more than $n$ and
all $N \geq 1$. 
\item
$f( x ) = \sum_{j = 0}^\infty a_j P_j^{(\lambda)}( x )$ with
$a_j \geq 0$ for all $j \geq 0$ and $\lambda = ( n - 1 ) / 2$.
\end{enumerate}
\end{theorem}

\begin{proof}
The result follows from \cite[Theorem 1]{Schoenberg-Duke42} and the
observation that correlation matrices of rank at most $n$ are in
correspondence with Gram matrices of vectors in $S^{n - 1}$.
\end{proof}

In order to approach the case of matrices of a fixed dimension, we
introduce some notation.

\begin{definition}
Let $I \subset \R$. Define $\cS_n( I )$ to be the set of
$n \times n$ symmetric matrices with entries in $I$.  Let $\rk A$
denote the rank of a matrix $A$. We define:
\begin{align*}
\cS_n^k(I) &:= \{A \in \cS_n(I) : \rk A \leq k\}, \\
\cP_n^k(I) &:= \{A \in \cP_n(I) : \rk A \leq k\}. 
\end{align*}
\end{definition}

The main result in \cite{GKR-lowrank} provides a characterization of
entrywise functions mapping $\cP_n^l$ into $\cP_n^k$.

\begin{theorem}[{Guillot--Khare--Rajaratnam \cite[Theorem
B]{GKR-lowrank}}]\label{Tltok}
Let $0 < R \leq \infty$ and $I = [ 0, R )$ or $( -R, R )$. Fix
integers $n \geq 2$, $1 \leq k < n - 1$, and $2 \leq l \leq n$.
Suppose $f \in C^k( I )$. The following are equivalent.
\begin{enumerate}
\item $f[ A ] \in \cS_n^k$ for all $A \in \cP_n^l( I )$; 
\item $f( x ) = \sum_{k = 1}^r c_t x^{i_t}$ for some $c_t \in \R$ and
some $i_t \in \N$ such that 
\begin{equation}\label{eqn:sum_binom}
\sum_{t = 1}^r \binom{i_t + l - 1}{l - 1} \leq k. 
\end{equation}
\end{enumerate}
Similarly, $f[-]: \cP_n^l( I ) \to \cP_n^k$ if and only if $f$
satisfies (2) and $c_t \geq 0$ for all~$t$. Moreover, if
$I = [ 0, R )$ and $k \leq n - 3$, then the assumption that
$f \in C^k(I)$ is not required. 
\end{theorem}

Notice that Theorem~\ref{Tltok} is a fixed-dimension result with rank
constraints. This may be considered a refinement of a similar,
dimension-free result with rank constraints shown in~\cite{AP}, in
which the authors arrive at the same conclusion as in part (2)
above. We compare the two settings: in~\cite{AP}, (a)~the hypotheses
held for all dimensions $N$ rather than in a fixed dimension; (b)~the
test matrices were a larger set in each dimension, compared to just
the positive matrices considered in Theorem~\ref{Tltok}; (c)~the test
matrices did not consist only of rank-one matrices, similar to
Theorem~\ref{Tltok}; and (d)~the test functions $f$ in the
dimension-free case were assumed to be measurable, rather than $C^k$
as in the fixed-dimension case.  Thus, Theorem~\ref{Tltok} is (a
refinement of) the fixed-dimension case of the first main result
in~\cite{AP}.%
\footnote{We also point out the second main result in
\textit{loc. cit.}, that is, \cite[Theorem~2]{AP}, which classifies
all continuous entrywise maps $f : \C \to \C$ that obey similar rank
constraints in all dimensions. Such maps are necessarily of the form
$g( z ) = \sum_{j = 1}^p \beta_j z^{m_j} (\overline{z})^{n_j}$, where
the exponents $m_j$ and $n_j$ are non-negative integers. This should
immediately remind the reader of Rudin's conjecture in the
`dimension-free' case, and its resolution by Herz; see
Theorem~\ref{Therz}.}

The $(2) \implies (1)$ implication in Theorem \ref{Tltok} is clear.
Indeed, let $i \geq 0$ and
$A = \sum_{j = 1}^l u_j u_j^T \in \cP_n^l( I )$. Then
\begin{equation*}
A^{\circ i} = \sum_{m_1 + \cdots + m_l = i} %
\binom{i}{ m_1, \ldots, m_l } \bw_\bm \bw_\bm^T \qquad %
\text{where } \bw_\bm := %
u_1^{\circ m_1} \circ \cdots \circ u_l^{\circ m_l}
\end{equation*}
and $\displaystyle \binom{i}{m_1, \ldots, m_l}$ is a
multinomial coefficient. Note that there are exactly
$\binom{i + l - 1}{l - 1}$ terms in the previous summation. Therefore
$\rk A^{\circ i} \leq \binom{i + l - 1}{l - 1}$, and so $(1)$ easily
follows from $(2)$. The proof that $(1) \implies (2)$ is much more
challenging; see \cite{GKR-lowrank} for details.

In \cite{GKR-sparse}, the authors focus on the case where sparsity
constraints are imposed to the matrices instead of rank constraints.
Positive semidefinite matrices with zeros according to graphs arise
naturally in many applications. For example, in the theory of Markov
random fields in probability theory (\cite{lauritzen,whittaker}), the
nodes of a graph $G$ represent components of a random vector, and
edges represent the dependency structure between nodes. Thus, absence
of an edge implies marginal or conditional independence between the
corresponding random variables, and leads to zeros in the associated
covariance or correlation matrix (or its inverse). Such models
therefore yield parsimonious representations of dependency structures.
Characterizing entrywise functions preserving positivity for matrices
with zeros according to a graph is thus of tremendous interest for
modern applications. Obtaining such characterizations is, however,
much more involved than the original problem considered by Schoenberg
as one has to enforce and maintain the sparsity constraint. The
problem of characterizing functions preserving positivity for sparse
matrices is also intimately linked to problems in spectral graph
theory and many other problems (see
e.g.~\cite{Hogben_2005,Agler_et_al_88, pinkus04,Brualdi2011}).

As before, for a given graph $G = ( V, E )$ on the finite vertex set
$V = \{ 1, \ldots, N \}$, we denote by $\cP_G( I )$ the set of
positive-semidefinite matrices with entries in $I$ and zeros according
to $G$, as in (\ref{EPG}). Given a function $f: \R \to \R$ and
$A \in \cS_{| G |}( \R )$, denote by $f_G[ A ]$ the matrix such that
\begin{equation*}
f_G[A]_{j k} := \begin{cases}
f( a_{j k} ) & \text{ if } ( j, k ) \in E \text{ or } j = k, \\
 0 & \text{otherwise}.
\end{cases}
\end{equation*} 

The first main result in \cite{GKR-sparse} is an explicit
characterization of the entrywise positive preservers of $\cP_G$ for
any collection of trees (other than copies of $K_2$). Following
Vasudeva's classification for $\cP_{K_2}$ in Theorem~\ref{TK2}, trees are
the only other graphs for which such a classification is currently known.

\begin{theorem}[{Guillot--Khare--Rajaratnam \cite[Theorem
A]{GKR-sparse}}]\label{GKR-sparse-ThmA}
Suppose $I = [ 0, R )$ for some $0 < R \leq \infty$, and
$f : I \to \R_+$. Let $G$ be a tree with at least $3$ vertices, and
let $A_3$ denote the path graph on $3$ vertices. The following are
equivalent.
\begin{enumerate}
\item $f_G[ A ] \in \cP_G$ for every $A \in \cP_G( I )$;
\item $f_T[ A ] \in \cP_T$ for all trees $T$ and all matrices
$A \in \cP_T( I )$;
\item $f_{A_3}[ A ] \in \cP_{A_3}$ for every $A \in \cP_{A_3}( I )$;
\item The function $f$ satisfies
\begin{equation}\label{Emidconvex}
f\bigl( \sqrt{x y} \bigr)^2 \leq f( x ) f( y ) \qquad %
\text{for all } x, y \in I
\end{equation}
and is super-additive on $I$, that is,
\begin{equation}\label{Esuper}
f( x + y ) \geq f( x ) + f( y ) \qquad %
\text{whenever } x, y, x + y \in I.
\end{equation}
\end{enumerate}
\end{theorem}

The implication $(4) \implies (1)$ was further extended to all chordal
graphs: it is the following result with $c = 2$ and $d = 1$.

\begin{theorem}[Guillot--Khare--Rajaratnam \cite{GKR-critG}]
Let $G$ be a chordal graph with a perfect elimination ordering of its
vertices $\{ v_1, \ldots, v_n \}$. For all $1 \leq k \leq n$, denote
by $G_k$ the induced subgraph on $G$ formed by
$\{ v_1, \ldots, v_k \}$, so that the neighbors of $v_k$ in $G_k$ form
a clique. Define $c = \omega( G )$ to be the clique number of $G$, and
let
\begin{equation*}
d := \max\{ \deg_{G_k}( v_k ) : k = 1, \ldots, n \}.
\end{equation*}
If $f : \R \to \R$ is any function such that $f[-]$ preserves
positivity on $\cP_c^1( \R )$ and $f[ M + N ] \geq f[ M ] + f[ N ]$
for all $M \in \cP_d$ and $N \in \cP_d^1$, then $f[-]$ preserves
positivity on $\cP_G( \R )$. [Here, $\cP_d^1$ denotes the matrices in
$\cP_d$ of rank at most one.]
\end{theorem}

See \cite{GKR-critG} for other sufficient conditions for a general
entrywise function to preserve positivity on $\cP_G$ for $G$ chordal.

To state the final result in this section, recall that Schoenberg's
theorem (Theorem~\ref{Tschoenberg}) shows that entrywise functions
preserving positivity for all matrices (that is, according to the
family of complete graphs $K_n$ for $n \geq 1$) are absolutely
monotonic on the positive axis. It is not clear if functions
satisfying (\ref{Emidconvex}) and (\ref{Esuper}) in Theorem
\ref{GKR-sparse-ThmA} are necessarily absolutely monotonic, or even
analytic. As shown in \cite[Proposition 4.2]{GKR-sparse}, the critical
exponent (see Definition \ref{DcriticalExp}) of every tree is~$1$.
Hence, functions satisfying (\ref{Emidconvex}) and (\ref{Esuper}) do
not need to be analytic. The second main result in \cite{GKR-sparse}
demonstrates that even if the function is analytic, it can in fact
have arbitrarily long strings of negative Taylor coefficients.

\begin{theorem}%
[{Guillot--Khare--Rajaratnam \cite[Theorem B]{GKR-sparse}}]
There exists an entire function
$f( z ) = \sum_{n = 0}^\infty a_n z^n$ such that 
\begin{enumerate}
\item $a_n \in [ -1, 1 ]$ for every $n \geq 0$;
\item The sequence $( a_n )_{n \geq 0}$ contains arbitrarily long
strings of negative numbers;
\item For every tree $G$, $f_G[ A ] \in \cP_G$ for every
$A \in \cP_G\bigl( \R_+ \bigr)$.
\end{enumerate}
In particular, if $\Delta( G )$ denotes the maximum degree of the
vertices of $G$, then there exists a family $G_n$ of graphs and an
entire function $f$ that is not absolutely monotonic, such that
\begin{enumerate}
\item $\sup_{n \geq 1} \Delta( G_n ) = \infty$;
\item $f_{G_n}[ A ] \in \cP_{G_n}$ for every
$A \in \cP_{G_n}( \R_+ )$.
\end{enumerate}
\end{theorem}

\bibliographystyle{plain}
\bibliography{Positive-bib}

\end{document}